\documentclass[14pt]{amsart}
\usepackage{amsfonts, amsthm, amssymb, amsmath, stmaryrd}
\usepackage{amsaddr}
\usepackage{mathrsfs,array}
\usepackage[english]{babel}
\usepackage[margin = 2.35cm]{geometry}
\usepackage{blindtext}
\usepackage{ucs}
\usepackage[utf8]{inputenc}
\usepackage[T1]{fontenc}  
\usepackage{latexsym}
\usepackage{amsfonts}
\usepackage{euscript}
\usepackage[pdftex]{graphicx}
\usepackage{color}
\usepackage{epic}
\usepackage{amsmath}
\usepackage{amssymb}
\usepackage{bbm}
\usepackage{url}
\usepackage{amsthm}
\usepackage{bbm}
\usepackage{textcomp}
\usepackage{enumitem}
\usepackage[linktoc=all]{hyperref}
\hypersetup{linktocpage} 
\usepackage[all, cmtip]{xy}
\usepackage{framed}
\usepackage{array}
\usepackage{tikz-cd}

\hypersetup{
    colorlinks,
    linkcolor={red!50!black},
    citecolor={blue!50!black},
    urlcolor={blue!80!black}
}
  
\numberwithin {equation}{section}

\newcommand\restr[2]{{
  \left.\kern-\nulldelimiterspace 
  #1 
  \right|_{#2} 
  }}
  
\input cyracc.def
\font\tencyr=wncyr6
\def\cyr{\tencyr\cyracc}

\newtheorem{theorem}{Theorem}[section]

\newtheorem*{theorem*}{Theorem}

\newtheorem*{theorem1}{Theorem 1}

\newtheorem*{theorem2}{Theorem 2}

\newtheorem*{theorem3}{Theorem 3}

\newtheorem{proposition}{Proposition}
\numberwithin{proposition}{subsubsection}

\newtheorem{lemma}[proposition]{Lemma}

\newtheorem{corollary}[proposition]{Corollary}

\theoremstyle{definition}

\newtheorem{definition}[theorem]{Definition}

\newtheorem*{definition*}{Definition}

\theoremstyle{remark}

\newtheorem{remark}[proposition]{Remark}

\newtheorem{notation}[proposition]{Notation}

\newtheorem{construction}[proposition]{Construction}

\newcounter{para}

\title{On the boundary and intersection motives of genus 2 Hilbert-Siegel varieties}

\author{Mattia Cavicchi}
\email{cavicchi@math.univ-paris13.fr}
\address{LAGA, Institut Galilée, Université Paris 13, F-93430 Villetaneuse,
France}

\begin{document}

\newcommand{\BA}{{\mathbb{A}}}
\newcommand{\BB}{{\mathbb{B}}}
\newcommand{\BC}{{\mathbb{C}}}
\newcommand{\BD}{{\mathbb{D}}}
\newcommand{\BE}{{\mathbb{E}}}
\newcommand{\BF}{{\mathbb{F}}}
\newcommand{\BG}{{\mathbb{G}}}
\newcommand{\BH}{{\mathbb{H}}}
\newcommand{\BI}{{\mathbb{I}}}
\newcommand{\BJ}{{\mathbb{J}}}
\newcommand{\BK}{{\mathbb{K}}}
\newcommand{\BL}{{\mathbb{L}}}
\newcommand{\BM}{{\mathbb{M}}}
\newcommand{\BN}{{\mathbb{N}}}
\newcommand{\BO}{{\mathbb{O}}}
\newcommand{\BP}{{\mathbb{P}}}
\newcommand{\BQ}{{\mathbb{Q}}}
\newcommand{\BR}{{\mathbb{R}}}
\newcommand{\BS}{{\mathbb{S}}}
\newcommand{\BT}{{\mathbb{T}}}
\newcommand{\BU}{{\mathbb{U}}}
\newcommand{\BV}{{\mathbb{V}}}
\newcommand{\BW}{{\mathbb{W}}}
\newcommand{\BX}{{\mathbb{X}}}
\newcommand{\BY}{{\mathbb{Y}}}
\newcommand{\BZ}{{\mathbb{Z}}}


\newcommand{\Fa}{{\mathfrak{a}}}
\newcommand{\Fb}{{\mathfrak{b}}}
\newcommand{\Fc}{{\mathfrak{c}}}
\newcommand{\Fd}{{\mathfrak{d}}}
\newcommand{\Fe}{{\mathfrak{e}}}
\newcommand{\Ff}{{\mathfrak{f}}}
\newcommand{\Fg}{{\mathfrak{g}}}
\newcommand{\Fh}{{\mathfrak{h}}}
\newcommand{\Fi}{{\mathfrak{i}}}
\newcommand{\Fj}{{\mathfrak{j}}}
\newcommand{\Fk}{{\mathfrak{k}}}
\newcommand{\Fl}{{\mathfrak{l}}}
\newcommand{\Fm}{{\mathfrak{m}}}
\newcommand{\Fn}{{\mathfrak{n}}}
\newcommand{\Fo}{{\mathfrak{o}}}
\newcommand{\Fp}{{\mathfrak{p}}}
\newcommand{\Fq}{{\mathfrak{q}}}
\newcommand{\Fr}{{\mathfrak{r}}}
\newcommand{\Fs}{{\mathfrak{s}}}
\newcommand{\Ft}{{\mathfrak{t}}}
\newcommand{\Fu}{{\mathfrak{u}}}
\newcommand{\Fv}{{\mathfrak{v}}}
\newcommand{\Fw}{{\mathfrak{w}}}
\newcommand{\Fx}{{\mathfrak{x}}}
\newcommand{\Fy}{{\mathfrak{y}}}
\newcommand{\Fz}{{\mathfrak{z}}}

\newcommand{\FA}{{\mathfrak{A}}}
\newcommand{\FB}{{\mathfrak{B}}}
\newcommand{\FC}{{\mathfrak{C}}}
\newcommand{\FD}{{\mathfrak{D}}}
\newcommand{\FE}{{\mathfrak{E}}}
\newcommand{\FF}{{\mathfrak{F}}}
\newcommand{\FG}{{\mathfrak{G}}}
\newcommand{\FH}{{\mathfrak{H}}}
\newcommand{\FI}{{\mathfrak{I}}}
\newcommand{\FJ}{{\mathfrak{J}}}
\newcommand{\FK}{{\mathfrak{K}}}
\newcommand{\FL}{{\mathfrak{L}}}
\newcommand{\FM}{{\mathfrak{M}}}
\newcommand{\FN}{{\mathfrak{N}}}
\newcommand{\FO}{{\mathfrak{O}}}
\newcommand{\FP}{{\mathfrak{P}}}
\newcommand{\FQ}{{\mathfrak{Q}}}
\newcommand{\FR}{{\mathfrak{R}}}
\newcommand{\FS}{{\mathfrak{S}}}
\newcommand{\FT}{{\mathfrak{T}}}
\newcommand{\FU}{{\mathfrak{U}}}
\newcommand{\FV}{{\mathfrak{V}}}
\newcommand{\FW}{{\mathfrak{W}}}
\newcommand{\FX}{{\mathfrak{X}}}
\newcommand{\FY}{{\mathfrak{Y}}}
\newcommand{\FZ}{{\mathfrak{Z}}}


\newcommand{\CA}{{\cal A}}
\newcommand{\CB}{{\cal B}}
\newcommand{\CC}{{\cal C}}
\newcommand{\CE}{{\cal E}}
\newcommand{\CF}{{\cal F}}
\newcommand{\CG}{{\cal G}}
\newcommand{\CI}{{\cal I}}
\newcommand{\CJ}{{\cal J}}
\newcommand{\CK}{{\cal K}}
\newcommand{\CL}{{\cal L}}
\newcommand{\CM}{{\cal M}}
\newcommand{\CN}{{\cal N}}
\newcommand{\CO}{{\cal O}}
\newcommand{\CP}{{\cal P}}
\newcommand{\CQ}{{\cal Q}}
\newcommand{\CR}{{\cal R}}
\newcommand{\CS}{{\cal S}}
\newcommand{\CT}{{\cal T}}
\newcommand{\CU}{{\cal U}}
\newcommand{\CV}{{\cal V}}
\newcommand{\CW}{{\cal W}}
\newcommand{\CX}{{\cal X}}
\newcommand{\CY}{{\cal Y}}
\newcommand{\CZ}{{\cal Z}}

\newcommand{\cyrb}{{\cyr B}} 

\newcommand{\one}{\mathds{1}}

\newcommand{\omg}{\mathop{\rm \omega}\nolimits}


\newcommand{\End}{\mathop{\rm End}\nolimits}
\newcommand{\Hom}{\mathop{\rm Hom}\nolimits}
\newcommand{\Sym}{\mathop{\rm Sym}\nolimits}
\newcommand{\Rep}{\mathop{\rm Rep}\nolimits}
\newcommand{\Res}{\mathop{\rm Res}\nolimits}
\newcommand{\Spec}{\mathop{\rm Spec}\nolimits}
\newcommand{\pt}{\boldsymbol{\cdot}}
\newcommand{\Gal}{\mathop{\rm Gal}\nolimits}

\newcommand{\CH}{\mathop{\rm CH}\nolimits}
\newcommand{\DBM}{\mathop{DM_{\text{\cyrb}}}\nolimits}
\newcommand{\DBcM}{\mathop{DM_{\text{\cyrb},c}}\nolimits}
\newcommand{\PL}{\mathop{\rm PL}\nolimits}
\newcommand{\Id}{\mathop{\rm Id}\nolimits}

\newcommand{\SL}{\mathop{\rm SL} \nolimits}
\newcommand{\So}{\mathop{\rm SO} \nolimits}
\newcommand{\Sp}{\mathop{\rm Sp} \nolimits}
\newcommand{\GL}{\mathop{\rm GL} \nolimits}
\newcommand{\PGL}{\mathop{\rm PGL} \nolimits}
\newcommand{\Or}{\mathop{\rm O} \nolimits}
\newcommand{\POr}{\mathop{\rm PO} \nolimits}
\newcommand{\GSp}{\mathop{\rm GSp} \nolimits}
\newcommand{\U}{\mathop{\rm U} \nolimits}

\newcommand{\Stab}{\mathop{\rm Stab} \nolimits}
\newcommand{\der}{\mathop{\rm \tiny{der}} \nolimits}
\newcommand{\ad}{\mathop{\rm \tiny{ad}} \nolimits}
\newcommand{\Gr}{\mathop{\rm Gr} \nolimits}
\newcommand{\im}{\mathop{\rm Im} \nolimits}
\newcommand{\can}{\mathop{\rm \tiny{can}} \nolimits}
\newcommand{\cusp}{\mathop{\rm \tiny{cusp}} \nolimits}
\newcommand{\et}{\mathop{\rm \tiny{ét}} \nolimits}
\newcommand{\op}{\mathop{\rm \tiny{op}} \nolimits}

\newcommand{\q}[1]{``#1''}

\begin{abstract}
We study genus 2 Hilbert-Siegel varieties, i.e. Shimura varieties $S_K$ corresponding to the group $\mbox{GSp}_{4,F}$ over a totally real field $F$, along with the relative Chow motives $^\lambda \mathcal{V}$ of abelian type over $S_K$ obtained from irreducible representations $V_{\lambda}$ of $\mbox{GSp}_{4,F}$. We analyse the weight filtration on the degeneration of such motives at the boundary of the Baily-Borel compactification and we find a criterion on the highest weight $\lambda$, potentially generalisable to other families of Shimura varieties, which characterizes the absence of the \emph{middle weights} 0 and 1 in the corresponding degeneration. Thanks to Wildeshaus' theory, the absence of these weights allows us to construct Hecke-equivariant Chow motives over $\mathbb{Q}$, whose realizations equal interior (or intersection) cohomology of $S_K$ with $V_{\lambda}$-coefficients. We give applications to the construction of homological motives associated to automorphic representations. 
\end{abstract} 
\maketitle

\tableofcontents
\addtocontents{toc}{\setcounter{tocdepth}{1}}
\newpage
\section*{Introduction}
\subsection*{Background: motives for automorphic representations.} Let $S_K$ be a Shimura variety associated to a reductive $\mathbb{Q}$-group $G$ and to a neat open compact subgroup $K$ of $G(\mathbb{A}_f)$. The variety $S_K$ is then smooth and quasi-projective, and defined over its \emph{reflex field} $E$ (a number field). Every algebraic representation $V$ of $G$ defines a local system 
$\mu(V)$ on $S_K(\mathbb{C})$, whose \emph{interior cohomology} $H^*_!$, i.e. the image of cohomology with compact supports into ordinary cohomology, contains very rich analytical and arithmetic information: in particular, the $K$-invariants  
of \emph{cohomological} cuspidal automorphic representations of $G(\mathbb{A})$ appear exactly inside the spaces of the form $H^*_!(S_K(\mathbb{C}), \mu(V))$.
On the other hand, for every prime $\ell$, $V$ also defines an $\ell$-adic sheaf $\mu_{\ell}(V)$ on $S_K$, whose étale cohomology spaces are equipped with a Galois action and the action of Hecke operators. The study of the interaction of these different structures plays a pivotal role in the Langlands program.

Hence, following the general philosophy explained for example in \cite{Cloz90}, it is desirable to construct a Chow motive whose realizations equal interior cohomology - a subspace of the cohomology which is of \emph{pure weight}, in the Galois or Hodge-theoretic sense. Moreover, one would like this construction to be functorial, in order to further decompose such a motive according to the Hecke action (maybe switching to homological motives). The first successful example of such a construction was given by Scholl (\cite{Sch90}), who defined motives realizing to the Galois representations associated by Deligne to weight $k \geq 2$ modular cusp forms (\cite{Del69}). 

The results of this paper imply that analogous motives exist for \emph{most} coefficient systems in the case of \emph{genus 2 Hilbert-Siegel varieties}, which are Shimura varieties $S_K$ associated to (a subgroup of) the group $G=\mbox{Res}_{F|\mathbb{Q}}\mbox{GSp}_{4,F}$ over a totally real field $F$ of degree $d$ over $\mathbb{Q}$. More precisely, let $I_F$ be the set of real embeddings of $F$, and let $V_{\lambda}$ be an irreducible representation of $G$ of highest weight $\lambda$: such a weight is in particular specified by a couple of vectors of non-negative integers $(k_{1,\sigma})_{\sigma \in I_F}, \ (k_{2,\sigma})_{\sigma \in I_F}$ such that $k_{1,\sigma} \geq k_{2,\sigma} \geq 0$ for every $\sigma$ (i.e., the weight is \emph{dominant}). It is called \emph{regular at} $\sigma$ if $k_{1,\sigma} > k_{2,\sigma} > 0$, and \emph{regular} if it is regular at $\sigma$ for every $\sigma$. Denote moreover by $CHM(\mathbb{Q})$ the category of \emph{Chow motives over} $\mathbb{Q}$. Then, the main consequence of our results is the following:

\begin{theorem1}{(Corollaries \ref{heckeaction} and \ref{cor:realint}.\ref{itm:real} and Rmk. \ref{rmk:proprealint}.\ref{itm:degreal})} \label{mainthmintro}
If $\lambda$ is regular, there exists an object $s_*j_{!*}^\lambda \mathcal{V}$ of $CHM(\mathbb{Q})$ whose Hodge-theoretical, resp. $\ell$-adic realization, equals $H_!^{3d}(S_K(\mathbb{C}), \mu(V_{\lambda}))$, resp. $H_!^{3d}(S_K \times_{\mathbb{Q}} \bar{\mathbb{Q}}, \mu_{\ell}(V_{\lambda}))$, and such that every element\footnote{We adopt this phrasing in order to stress that this is \emph{not} known to be an algebra action, see Footnote \ref{algebract}.} of the \emph{Hecke algebra} $\mathfrak{H}(K, G(\mathbb{A}_f))$ acts naturally on it. 
\end{theorem1}

Actually, we show that a functorial Chow motive realizing to interior cohomology exists under a less restrictive hypothesis on the weight $\lambda$. To see how this is achieved, let us sketch the actual contents of this work, and explain the interest of the particular family of Shimura varieties which we consider.

\subsection*{The role of the boundary motive.} Let us first come back to the general case of a Shimura variety $S_K$ over $E$, and denote by $j: S_K \hookrightarrow S_K^*$ the open immersion into the \emph{Baily-Borel compactification} (a projective variety, still defined over $E$) and by $i: \partial S_K^* \hookrightarrow S_K^*$ the closed immersion of the boundary $\partial S_K^*:=S_K^* \backslash S_K$. The latter is itself stratified by (quotients by the action of a finite group of) Shimura varieties, still defined over $E$. 

Assume $S_K$ to be of \emph{PEL type} (loosely speaking, a moduli space of abelian varieties equipped with polarizations, endomorphisms and level structure). Then, according to \cite{Anc15}, every irreducible representation $V_{\lambda}$ of $G$ of highest weight $\lambda$ gives rise, in a functorial way, to an object $^\lambda \mathcal{V}$ of the category $CHM(S_K)$ of \emph{Chow motives over} $S_K$, whose $\ell$-adic realization is $\mu_{\ell}(V_{\lambda})$. In this context, the theory developed by Wildeshaus (especially the works \cite{Wil09}, \cite{Wil17a}, \cite{Wil19a}) implies that there exists a cohomological condition on the \emph{degeneration at the boundary} $i^* j_* \mu_{\ell} (V_{\lambda})$ of the $\ell$-adic sheaves $\mu_{\ell}(V_{\lambda})$, which, once satisfied, allows for the construction of a Chow motive with the properties stated in Theorem 1: the condition consists in the \emph{absence of weights 0 and 1} in the latter complex of $\ell$-adic sheaves, where the weights are those defined by the Galois action. More precisely, this hypothesis allows to construct a functorial Chow motive which realizes to \emph{intersection cohomology} with values in $\mu_{\ell}(V_{\lambda})$ (hence the name \emph{intersection motive}) and to identify the latter with interior cohomology (cfr. Cor. \ref{cor:realint}). Moreover, the criterion contained in \cite{Wil19a} (recorded here as Theorem \ref{critint_a}) implies that, in order to detect this weight avoidance, it suffices to analyse the weights of the \emph{perverse} cohomology objects of such complexes, \emph{stratum by stratum} on the boundary: this is what we actually do in this paper. Notice that the complex $i^* j_* \mu_{\ell} (V_{\lambda})$ is actually the $\ell$-adic realization of the (relative) \emph{boundary motive} $i^* j_* ^\lambda \mathcal{V}$, an object of the category $\DBcM(\partial S_K^*)$ of \emph{constructible Beilinson motives over} $\partial S_K^*$, defined over general bases in \cite{CD12}, along with its six functors formalism. If $\tilde{s}$ denotes the structural morphism of $S_K$, the intersection motive should be then thought as the \q{lowest weight-graded object} of the motive $\tilde{s}_* ^\lambda \mathcal{V} \in \DBcM(E)$, in the sense of the \emph{weight structure} introduced by Bondarko (\cite{Bon10}). 

The analysis of the weight filtration brings into consideration contributions coming from different variants of group cohomology: abstract cohomology of arithmetic and finitely generated free abelian groups, and cohomology of algebraic groups. In previous work (\cite{Wil12} for Hilbert-Blumenthal varieties, \cite{Wil15} for Picard surfaces, \cite{Cloî17} for Picard varieties of arbitrary dimension and \cite{Wil19b} for Siegel threefolds), these facts have been employed in order to show that \emph{regularity} of the coefficient systems implies the avoidance of the weights 0 and 1 at the boundary. Two natural questions then arise. First: does this hold for other families of Shimura varieties? Second, since, in the first and third case above, it can be seen that there exist non-regular representations, which nonetheless satisfy the weight avoidance, one is lead to ask: does there exist a general condition on the highest weights of irreducible representations of $G$, which is \emph{equivalent} to the absence of the weights 0 and 1 in the degeneration at the boundary? In this paper we answer both questions in the case of genus 2 Hilbert-Siegel varieties. 

Before explaining what the response is, let us just discuss the role of this special case. In the first three of the examples studied before, the strata of the boundary of the Baily-Borel compactification are simply of dimension 0, while one-dimensional strata appear in the boundary of Siegel threefolds: this makes the analysis sensibly more difficult and forces one to use the \emph{relative} formalism of Beilinson motives. Genus 2 Hilbert-Siegel varieties represent then a natural following step: the boundary still presents only two types of strata, but the higher-dimensional ones can be of \emph{arbitrary} dimension - equal, in fact, to the degree of the field $F$ over $\mathbb{Q}$. More importantly, for the first time with respect to the preceding cases, the phenomena, which are caused by the three types of cohomology listed before, make their appearance all together (especially because of the arithmetic of the field $F$, a point which we will comment on later in more detail). 
\\

Let us now define the \emph{corank} of a weight $\lambda =((k_{1,\sigma})_{\sigma}, \ (k_{2,\sigma})_{\sigma})$ as 0 if $k_{2,\sigma}$ is not the same integer for every $\sigma$, as 1 if $k_{2,\sigma}$ is constant but there exists a $\sigma$ such that $k_{1,\sigma} \neq k_{2,\sigma}$, and as 2 if there exists an integer $\kappa$ such that $k_{1,\sigma}=k_{2,\sigma}=\kappa$ for every $\sigma$. Moreover, let us call the weight \emph{completely irregular} if, for every $\sigma$, it is not regular at $\sigma$. Then, for genus 2 Hilbert-Siegel varieties, as a consequence of our main technical result (Thm. \ref{thm:mainthm}), we are able to exhibit the sought-for characterization of the absence of weights, by showing that it is precisely the notion of corank which allows for its formulation: 

\begin{theorem2}{(Cfr. Cor. \ref{evit01})}
The weights 0 and 1 appear in the complex $i^* j_* \mu_{\ell} (V_{\lambda})$ if and only if $\lambda$ is completely irregular of corank $\geq 1$.
\end{theorem2} 

Notice that, in particular, we are saying that either \emph{both} weights 0 and 1 appear, or \emph{none} of them appears. By the theory cited above, and explained in more detail in Subsection \ref{critexistence} and Section \ref{consapp}, Theorem 2 implies the validity of Theorem 1, and more generally of the following fact: a Chow motive which realizes to interior cohomology, equipped with a Hecke action, exists as soon as the weight $\lambda$ is either not completely irregular or of corank 0.
 
The characterization given by Theorem 2 subsumes the previously treated cases of Hilbert-Blumenthal varieties and Siegel threefolds, and lends itself to further generalization. In order to conclude this introduction, let us put this result into perspective, by explaining why we think of the corank as an important invariant for the study of the cohomology of Shimura varieties. 

\subsection*{Corank and weight filtration in the cohomology of Shimura varieties.} 

For a general (smooth) Shimura variety $S_K$, it is a very important problem to describe the weight filtration (say, in the Hodge-theoretical sense) on the spaces $H^n(S_K(\mathbb{C}), \mu(V))$. Suppose the variation of Hodge structure $\mu(V)$ to be pure of weight $w$. Consider moreover the long exact sequence associated to the complementary, closed-open immersions $i$ and $j$ given by \emph{any} (topological) compactification $\bar{S}$ of $S_K(\mathbb{C})$ and by its boundary $\partial \bar{S}$. By Hodge theory, one then sees that, for each $n$, interior cohomology is contained in the weight $n+w$ subspace of $H^n(S_K(\mathbb{C}), \mu(V))$ (the \emph{pure} part, of lowest possible weight), while the rest of the weight filtration is determined by a subspace of the \emph{boundary cohomology} $\partial H^n(S_K(\mathbb{C}), \mu(V)):=H^n(\partial \bar{S}, i^* j_{*}\mu(V))$. The latter is precisely the hypercohomology of (the Hodge-theoretical analogues of) the complexes studied in this paper.

In this context, one approach to understanding the cohomology of $i^* j_{*}\mu(V)$ consists in taking a closed cover of $\partial \bar{S}$ and in studying the associated spectral sequence, abutting to $\partial H^*(S_K(\mathbb{C}), \mu(V))$. For $\bar{S}$ equal to the analytification of a smooth toroidal compactification\footnote{But also, in the - \emph{a priori} - purely topological case, equal to the \emph{Borel-Serre compactification} of $S_K$. The fact that this non-algebraic compactification still gives rise to spectral sequences of Hodge structures represents one of the most important discoveries in \cite{HZ94}.}, the study of this spectral sequence (of mixed Hodge structures) is the subject of \cite{HZ94}, \cite{HZ01}. In these works, the authors remark (in a slightly different language) that the interaction between the two filtrations of $i^* j_{*}\mu(V)$, the one coming from the closed cover of the boundary and the weight filtration, is far from being understood. The same is then true for the two corresponding filtrations on the spaces $H^n(S_K(\mathbb{C}), \mu(V))$.

On the other hand, after \cite{Na13}, these same spaces are endowed with a third interesting filtration, whose graded objects are equipped with \emph{semisimple} mixed Hodge structures: in \emph{loc. cit.}, the author uses the work of Franke (\cite{Fra98}) to show in particular that this filtration has an \emph{automorphic} origin. Once again, the precise relationship with the weight filtration remains mysterious. 
\\

In order to read our main result in the light of the above considerations, let us come back to the situation of a genus 2 Hilbert-Siegel variety $S_K$. One knows (cfr. paragraphs \ref{strctcomp}-\ref{strates}) that the boundary $\partial S_K^*$ of its Baily-Borel compactification admits a stratification of the form 
\begin{equation*}
\partial S_K^*=\bigsqcup\limits_{i=0}^1 \bigsqcup_{g \in \mathcal{C}_i} S_{i,g},
\end{equation*}
where the $\mathcal{C}_i$'s are some finite subsets of $G(\mathbb{A}_f)$ and the open, resp. closed subschemes $S_{1,g}$, resp. $S_{0,g}$ of $\partial S_K^*$ are (quotients by the action of a finite group of) Shimura varieties of dimension $d$, resp. 0, for $d$ the degree of $F$ over $\mathbb{Q}$. 

For an irreducible representation $V_{\lambda}$ of $G$ determined by a dominant weight $\lambda$, let us recall without definition, just for the needs of this introduction, the associated \emph{automorphic (coherent) sheaf} $\omega(\lambda)$ over a fixed toroidal compactification $S_{K,\Sigma}$ of $S_K$. By pushforward along the projection $\pi:S_{K, \Sigma} \rightarrow S_K^*$, this sheaf gives rise to a coherent sheaf on $S_K^*$, whose space $\mbox{M}_{\lambda,K}$ of global sections over $S_K^*$ is called the \emph{space of} $\lambda$-\emph{automorphic forms of level $K$} (non-standard terminology). Then, pose the following: 
\begin{definition*}{(cfr. \cite[Def. 1.5.3]{BR16})}
Let $f \in \mbox{\emph{M}}_{\lambda,K}$ be a non-zero automorphic form. Let $q$ be the maximal integer in $\{ -1,0,1 \}$ such that 
$\oplus_{g \in \mathcal{C}_{q}} \restr{f}{S_{q,g}}=0$ 
(with $\mathcal{C}_{-1}=\varnothing$ by convention). The \emph{corank} of $f$ is then defined as $\mbox{cor}(f):=1-q$.
\end{definition*}

Notice that this notion gives a measure, in some sense, of the degree of cuspidality of $f$: for example, $\mbox{cor}(f)=0 \iff f$ is cuspidal, and we could define \emph{completely non-cuspidal forms} those $f$'s for which $\mbox{cor}(f)=2$. The notation would have been more natural if we had ordered the strata by \emph{decreasing} dimension, but we have preferred to remain coherent with the notation in the rest of the paper. 

Recall then the following theorem from \cite{BR16} (where the role of $\lambda$, resp. $\mbox{M}_{\lambda,K}$, is played by $k$, resp. $\mbox{M}_{k}(\mathcal{H},R)$), in which the notion of corank of $\lambda$ as defined before arises in an essential way:

\begin{theorem*}{(\cite[Thm. 1.5.6]{BR16})}
If $f \in \mbox{\emph{M}}_{\lambda,K}$ is non-zero, then $\mbox{\emph{cor}}(\lambda) \geq \mbox{\emph{cor}}(f)$.
\end{theorem*}

This implies for example that, in order to have non-zero non-cuspidal forms, it is necessary for $\lambda$ to be of corank $\geq 1$; moreover, in order to have non-zero \emph{completely non-cuspidal} forms, it is necessary for $\lambda$ to be of corank 2, and in particular \emph{completely irregular}. 

Observe now that, reasoning along the lines of Corollary \ref{cor:propint}.\ref{itm:filt}, we can immediately deduce the following\footnote{The results stated so far concern a priori the weights of the complexes of $\ell$-\emph{adic sheaves} obtained from the $\ell$-adic realization of $i^* j_* ^\lambda \mathcal{V}$, but they actually give information on the \emph{weight filtration} in the sense of Bondarko (\cite{Bon10}) on the boundary motive itself. This doesn't imply directly anything on the Hodge side, because the \emph{Hodge realization} functor on Beilinson motives over \emph{singular} bases (as $\partial S^*$) has not been constructed yet. However, thanks to the complete formal analogy between the results of \cite{Pin92} ($\ell$-adic context) and \cite{BW04} (Hodge-theoretic context), and between the formalisms of six functors in the respective derived categories, our computations are also valid in the case of mixed Hodge modules; it is in fact the Hodge theoretic picture which guides these computations (cfr. Rmk. \ref{rmk:poidsrep}.\ref{itm:poidsfilt}).} from Theorem 2: 
\begin{theorem3}
Let $V_{\lambda}$ be an irreducible representation and let $w$ be the weight of the \emph{pure} variation of Hodge structure $\mu(V_{\lambda})$. Then, there can exist a $n$ such that weight $n+w+1$ appears in $H^n(S_K(\mathbb{C}), \mu(V_{\lambda}))$ only if $\lambda$ is completely irregular and of corank $\geq 1$.
\end{theorem3}
Even if the proofs of Theorem 2 and of the theorem from \cite{BR16} are completely independent, the analogy between the two is striking. In fact, thanks to the notion of corank, we see that weight 1 can appear in the complex computing $V_{\lambda}$-valued cohomology of $S_K$ only if $\lambda$ is completely irregular and satisfies the necessary conditions for the presence of a non-zero non-cuspidal form on $S_K$. Here we have in mind a non-zero non-cuspidal form as giving a class in the quotient of $H^n(S_K(\mathbb{C}), \mu(V_{\lambda}))$ of weight strictly bigger than $n+w$, for some $n$. Since the corank of an automorphic form is in turn defined in terms of its behaviour along the different types of strata in $\partial S_K^*$, Theorem 3 pushes us to think of our result as a (very little!) hint towards the understanding of the links between the weight filtration, the \q{simplicial} filtration coming from the stratification of the boundary and the \q{automorphic} filtration in the cohomology of Shimura varieties. One could ask if investigating the relationship with the filtration by \emph{holomorphic rank} considered in \cite[Sec. 4.4-4.5]{HZ01} could be a good starting point for studying such questions.

Notice that, in order to prove the \emph{presence} of the weights 0 and 1, we need exactly the existence of a non-zero automorphic form of a certain type. However, such a form is a \emph{cuspidal} Hilbert modular form over a \q{virtual} Hilbert-Blumenthal variety, which does not appear in $\partial S_K^*$ (cfr. Prop. \ref{poidsapp} and Rmk. \ref{fibres}). 

As a last remark, let us stress the fact that the non-triviality of the totally real extension $F$ gives rise to an action of some subgroups of units of the integers of $F$, which coincides with the action of the \emph{local Hecke operator} from \cite{LR91} (which, by the way, plays a crucial role for the results of \cite{Na13} cited before) and puts essential restrictions on the possible weights (cfr. Rmks. \ref{rmk:coincpoids} and \ref{rmk:coincpoids2}). This manifestation of the interaction between different group cohomologies, alluded to before, is the main phenomenon which leads us to the notion of corank for $\lambda$. 

\subsection*{Organisation of the paper} In the preliminary Section \ref{prelim} we first introduce the Shimura datum $(G, X)$ underlying the Hilbert-Siegel varieties $S_K$. Second, we recall the structure of the Baily-Borel compactification $S^*_K$ and we introduce the \emph{canonical construction} functors, which produce a variation of Hodge structure $\mu_H(V)$ or an $\ell$-adic sheaf $\mu_{\ell}(V)$ over $S_K$ from a representation $V$ of $G$, along with the relative Chow motives $^\lambda \mathcal{V}$ over $S_K$ associated to irreducible representations $V_{\lambda}$. Third, we recall the criterion which we will crucially make use of, i.e. Thm. \ref{critint_a}: it reduces the weight avoidance for the Beilinson motive $i^* j_* ^\lambda \mathcal{V}$ to an equivalent condition on the weights of the $\ell$-adic perverse cohomology sheaves of each of its restrictions to a stratum of the boundary. Finally, we recall some useful tools to study the weights of the latter sheaves: $(a)$ a theorem of Pink, which gives a formula for the $\ell$-adic \emph{classical} cohomology sheaves of the restriction of the degeneration to each stratum, in terms of cohomology of unipotent algebraic groups and of arithmetic groups; $(b)$ a theorem of Kostant, which allows one to express the cohomology of unipotent groups in terms of representations of subgroups, which are attached to the Shimura data underlying the strata of the boundary; $(c)$ some general lemmas which are useful for studying the cohomology of free abelian subgroups of arithmetic groups. 

In Section \ref{deg}, the heart of the paper, we begin by stating the main result (the description of the \emph{limit weights} of $i^* j_* ^\lambda \mathcal{V}$ in terms of the corank of $\lambda$, Thm. \ref{thm:mainthm}) and by showing its main consequence (the characterization of the absence of weights 0 and 1,  Cor. \ref{evit01}, our Theorem 2 here). The rest of the section is occupied by the proof of Thm. \ref{thm:mainthm}, which is divided in the following steps: 
\begin{itemize}
\item we begin by studying separately the classical cohomology sheaves of the degeneration along the 0-dimensional strata (Subsection \ref{degsieg}) and along the strata of dimension $d=[F:\mathbb{Q}]$ Subsection \ref{degkling}), via Pink's theorem;
\item in each of the two above cases, we decompose the cohomology of unipotent groups into a sum of irreducible representations carrying pure Hodges structures (paragraphs \ref{sec-cohunip0}, \ref{sec-cohunip1}); this is done thanks to Kostant's theorem and to a detailed study of the representation theory of $\GSp_4$. Thus, we get a list of the \emph{possible} weights appearing;
\item in each case, we use the cohomology of arithmetic groups to give restrictions on the non-triviality of the occurring spaces, and hence to give \emph{necessary} conditions for certain weights to appear (paragraphs \ref{arithm0}, \ref{arithm1}). The main technical ideas here consist in exploiting, often by spectral sequence arguments, the aforementioned action of suitable subgroups of units of $F$ on the fibers of the degeneration (Lemmas \ref{cohunip0}, \ref{trivbis} - where the action is exploited in two \q{orthogonal} ways, Lemma \ref{cohunip1}) and some vanishing theorems on the cohomology of locally symmetric spaces (Lemma \ref{restrcoh0}); 
\item some additional work is needed (as an existence statement for suitable non-zero Hilbert cusp modular forms, Prop. \ref{poidsapp}) to show that the above conditions are also \emph{sufficient} for some weights to appear;
\item finally, we have to relate the weights of the classical cohomology sheaves to those of the perverse cohomology sheaves appearing in Thm. \ref{critint_a}. For this, we are led to study the \emph{double} degeneration along the cusps of the $d$-dimensional strata, along the same lines as before (Subsection \ref{doubledeg}), and to use the properties of the \emph{intermediate extension} of $\ell$-adic perverse sheaves (Subsection \ref{evit}, where we complete the proof of the main theorem, i.e. we find the description of the relation between $\lambda$ and the weights of the motive $i^* j_* ^\lambda \mathcal{V}$).
\end{itemize}

In Section \ref{consapp} we give the applications of our result: the construction of a Hecke-equivariant Chow motive which realizes to interior cohomology of $S_K$ with coefficients which are not \emph{completely irregular of corank $\geq 1$}. In the case of regular coefficients, we describe the consequences of the above for the construction of motives associated to automorphic representations, obtaining in particular (homological) motives corresponding to the representations studied in \cite{Fli05} (cfr. Rmk. \ref{modmot}). 

\subsection*{Notations} 

The symbols $\mathbb{A}$, resp. $\mathbb{A}_f$ will denote the ring of adèles, resp. finite adèles. Throughout the whole paper, $F$ will be a fixed totally real field of degree $d$ over $\mathbb{Q}$, $I_F$ its set of real embeddings (thus, of cardinality $d$), and $L$ a fixed Galois closure of $F$ in $\BC$. An empty entry in a ring-valued matrix will mean that the corresponding coefficient is zero.

We will use the symbol $\pi_0(X)$ for the set of connected components of a topological space $X$. 

If $\mathcal{C}$ is a category, then $\mbox{Gr}_{\mathbb{Z}}\mathcal{C}$ will denote its category of graded objects. If $\mathcal{A}$ is an abelian category, $D^b(\mathcal{A})$ will denote its bounded derived category. 

\subsection*{Acknowledgements}
This work is part of the author's PhD thesis at the \emph{Laboratoire Analyse, Géométrie et Applications} of the \emph{Université Paris 13}, under the direction of Jörg Wildeshaus. I want to thank him very heartily for suggesting the problem and for all the discussions we have had. I also thank Kai-Wen Lan for answering my questions, Jacques Tilouine and Guillaume Cloître for interesting conversations, Giuseppe Ancona for precious suggestions and Vincent De Daruvar for reading a preliminary version of this text. Last but not least, I would like to thank the referee for his/her very careful reading and detailed comments, which have considerably contributed to improve both the organization of the paper and the exposition of several proofs. 

\addtocontents{toc}{\setcounter{tocdepth}{2}}

\section{The Shimura datum, the Baily-Borel compactification and the canonical construction} \label{prelim}

\noindent In this preliminary section we introduce the basic objects of interest: the Shimura datum $(G, X)$ which defines the \emph{genus 2 Hilbert-Siegel varieties} $S_K$, their \emph{Baily-Borel compactifications} $S_K^*$ and the relative Chow motives $^\lambda \mathcal{V}$ over $S_K$ associated to irreducible representations $V_{\lambda}$ of $G$. 

\subsection{The group for the Shimura datum}

\subsubsection{\textbf{The group of symplectic similitudes and the group $G$.}}
Define, for every positive integer $n$, the $2n \times 2n$ matrix
\begin{equation*}
J_n:= 
\begin{pmatrix}
& I_n \\
-I_n & 
\end{pmatrix}
\in \mbox{GL}_{2n}(\mathbb{Q}).
\end{equation*}

The group $\mbox{GSp}_{2n}$ of \emph{symplectic similitudes} of dimension $2n$ is the algebraic group over $\mathbb{Q}$ defined, for every $\mathbb{Q}$-algebra $R$, by posing
\begin{equation*}
\mbox{GSp}_{2n}(R):= \left\{ g \in \mbox{GL}_{2n}(R) | ^{t}gJ_ng=\nu(g)J_n, \nu(g) \in \mathbb{G}_m(R) \right\}.
\end{equation*}

It is a reductive group, whose center $Z$ is isomorphic to $\mathbb{G}_m$ and whose derived subgroup is isomorphic to $\mbox{Sp}_{2n}$, the usual \emph{symplectic group} over $\mathbb{Q}$ of dimension $2n$. The morphism $\nu:\mbox{GSp}_{2n} \to \mathbb{G}_m$ is called the \emph{multiplier} (or \emph{similitude factor}).

For the rest of this paper, fix $n=2$. Recall that we have fixed a totally real field $F$ of degree $d$ over $\mathbb{Q}$ along with a Galois closure $L$, and that we denote by $I_F$ the set of real embeddings of $F$. Define then the $\mathbb{Q}$-algebraic group $\tilde{G}$ by posing 
\begin{equation*}
\tilde{G}:=\mbox{Res}_{F|\mathbb{Q}}\mbox{GSp}_{4,F}.
\end{equation*}  
For every subfield $k$ of $\mathbb{C}$ containing $L$, one has, for every $k$-algebra $R$, an isomorphism 
\begin{equation}\label{isogalois}
F \otimes_{\mathbb{Q}} R \tilde{\rightarrow} \prod\limits_{\sigma \in I_F} R, \mbox{     } f \otimes r \mapsto (\sigma(f) \cdot r)_{\sigma}
\end{equation}
which induces a canonical isomorphism
\begin{equation*}
\tilde{G}_k \simeq \prod\limits_{\sigma \in I_F} (\mbox{GSp}_{4, k})_{\sigma},
\end{equation*}

Consider now the canonical adjunction morphism $\mathbb{G}_m \rightarrow \mbox{Res}_{F|\mathbb{Q}} \mathbb{G}_{m, F}$, induced by the fact that Weil restriction is right adjoint to base change, and the morphism $\tilde{G} \rightarrow \mbox{Res}_{F|\mathbb{Q}} \mathbb{G}_{m, F}$, induced by functoriality by the multiplier $\nu: \mbox{GSp}_{2n} \rightarrow \mathbb{G}_m$.
\begin{definition} \label{defgrp} 
The reductive $\mathbb{Q}$-group $G$ is defined by
\begin{equation}
G:= \mathbb{G}_m \times_{\mbox{Res}_{F|\mathbb{Q}} \mathbb{G}_{m, F}} \tilde{G},
\end{equation}
where the fibred product has been taken with respect to the preceding morphisms.

\end{definition}

\begin{remark}
\begin{enumerate} [wide, labelwidth=!, labelindent=0pt, label=(\arabic*)]
\item The isomorphism \eqref{isogalois} induces, for every subfield $k$ of $\mathbb{C}$ containing $L$, an isomorphism

\begin{equation} \label{grouppoints}
G_k \simeq \mathbb{G}_m \times_{\prod\limits_{\sigma \in I_F} (\mathbb{G}_{m, k})_{\sigma}, \prod\nu} \prod\limits_{\sigma \in I_F} (\mbox{GSp}_{4, k})_{\sigma}.
\end{equation}
\item \label{itm:centre}
The center of $G$ is such that $Z(G) \simeq \mathbb{G}_m \times_{\mbox{Res}_{F|\mathbb{Q}} \mathbb{G}_{m, F}, x \mapsto x^2 } \mbox{Res}_{F|\mathbb{Q}} \mathbb{G}_{m, F}$. Its neutral component is then isogenous to $\mathbb{G}_m$.
\end{enumerate}
\label{rmk:centre}
\end{remark}

\subsubsection{\textbf{The structure of parabolic subgroups of $G$.}} \label{parab}

The group $\mbox{GSp}_{4,F}(F)$ acts on $F^{\oplus 4}$ through the natural action induced by its inclusion into $\mbox{GL}_{4,F}(F)$. 

The standard $F$-basis $\{ e_1, e_2, e_3, e_4 \}$ gives then a symplectic basis for the non-degenerated, $F$-bilinear alternated form defined by $J_2 \in \mbox{GSp}_{4,F}(F)$, which we will also denote $J_2$. Fix as a maximal torus of $\mbox{GSp}_{4,F}$ the standard diagonal torus $\tilde{T}$ defined on $F$-points by 
\begin{equation} \label{pointstore}
\tilde{T}(F):=\{ \mbox{diag}(\alpha_1, \alpha_2, \alpha_1^{-1}\nu, \alpha_2^{-1}\nu) \  \vert \ \alpha_1, \alpha_2, \nu \in \mathbb{G}_m(F) \},
\end{equation}
along with the standard Borel $\tilde{B}$ containing it, defined on $F$-points as the subgroup of matrices in $\mbox{GSp}_{4,F}(F)$ of the form

\[
\left(
\begin{array} {cccc}
*&*&*&* \\
&*&*&*\\
& & * & \\
& & * & * 
\end{array}
\right)
\]

One knows that the parabolic subgroups of $\mbox{GSp}_{4,F}(F)$ correspond bijectively to subgroups of the form $\mbox{Stab}(V)$, for $V$ a sub-$F$-vector space of $F^{\oplus 4}$ which is totally isotropic for the form $J_2$. The case $V=\{0\}$ corresponds to the whole group $\mbox{GSp}_{4,F}(F)$, while $V=\langle e_1 \rangle$ gives the \emph{Klingen parabolic}

\[
\tilde{Q}_1(F):=\{ \left(
\begin{array} {cccc}
\alpha & * & * & * \\
 & a & * & b \\
 &  & \beta &  \\
 & c & * & d 
\end{array}
\right)
\vert ad-bc=\alpha \beta \in \mathbb{G}_{m,F}(F) \} \cap \mbox{GSp}_{4,F}(F)
\]
and $V=\langle e_1, e_2 \rangle$ gives the \emph{Siegel parabolic}

\[
\tilde{Q}_0(F):=\{ \left(
\begin{array} {cc}
\alpha A & AM \\
& ^{t}A^{-1}
\end{array}
\right)
\vert \alpha \in \mathbb{G}_{m,F}(F), A \in \mbox{GL}_{2,F}(F), ^{t}M=M \}.
\]
Every other parabolic subgroup is conjugated to one of the above.

One also knows that a maximal torus, resp. a Borel, of $\tilde{G}$ are given by $\mbox{Res}_{F|\mathbb{Q}}(\tilde{T})$, resp. $\mbox{Res}_{F|\mathbb{Q}}(\tilde{B})$, which we will still denote by $\tilde{T}$, $\tilde{B}$ in the following; note that $\tilde{T}$ is not split over $\mathbb{Q}$. A maximal torus and a Borel containing it in $G$ are then respectively defined by $T:=\mathbb{G}_m \times_{\mbox{Res}_{F|\mathbb{Q}} \mathbb{G}_{m, F}} \tilde{T}$ and $B:=\mathbb{G}_m \times_{\mbox{Res}_{F|\mathbb{Q}} \mathbb{G}_{m, F}} \tilde{B}$.

In the same way, the standard maximal parabolics of $\tilde{G}$ corresponding to the choice $(\tilde{T}, \tilde{B})$ are exactly given, up to conjugation, by $\mbox{Res}_{F|\mathbb{Q}}\tilde{Q}_0$, $\mbox{Res}_{F|\mathbb{Q}}\tilde{Q}_1$, which we will still denote by $\tilde{Q}_0$, $\tilde{Q}_1$. Then, $Q_0:=\mathbb{G}_m \times_{\mbox{Res}_{F|\mathbb{Q}} \mathbb{G}_{m, F}} \tilde{Q}_0$, $Q_1:=\mathbb{G}_m \times_{\mbox{Res}_{F|\mathbb{Q}} \mathbb{G}_{m, F}} \tilde{Q}_1$ are the standard maximal parabolics of $G$ with respect to $(T,B)$, still called the \emph{Siegel} and the \emph{Klingen} one. 

\subsubsection{\textbf{The Levi components of parabolic subgroups.}} \label{levicomp} Let $W_0$ and $W_1$ be the unipotent radicals of the groups $Q_0$ and $Q_1$ defined above. The quotients $Q_{i}/W_{i}$ will be canonically identified with subgroups of the $Q_i$'s, thanks to the Levi decomposition of the latter. 

Fix now a subfield $k$ of $\mathbb{C}$ which contains $L$. One has the following explicit description of the diagonal embedding of $Q_0/W_0 (\mathbb{Q})$ into $Q_0/W_0 (k)$:

\[
Q_0/W_0 (\mathbb{Q}) \simeq \{ ( \left( \begin{array} {cc} 
\alpha \sigma(A) &  \\
 & (\sigma(A)^{-1})^t 
\end{array} \right) )_{\sigma \in I_F}  \ \vert \alpha \in \mathbb{Q}^{\times}, A \in \mbox{GL}_2(F) \} \hookrightarrow{} 
\]

\[ 
\hookrightarrow Q_0/W_0 (k) = \{ ( \left( \begin{array} {cc} 
\alpha A_{\sigma} &  \\
 & (A_{\sigma}^{-1})^t 
\end{array} \right) )_{\sigma \in I_F}  \ \vert \alpha \in k^{\times}, A_{\sigma} \in \mbox{GL}_2(k) \ \mbox{for every} \ \sigma \}
\]
and of the diagonal embedding of $(Q_1/W_1)(\mathbb{Q})$ into $(Q_1/W_1)(k)$:

\[
Q_1/W_1 (\mathbb{Q}) \simeq \{ ( \left( \begin{array} {cccc} 
\sigma(t) \cdot (ad-bc) & & & \\
 & \sigma(a) & & \sigma(b)  \\
 & & \sigma(t^{-1}) & \\
& \sigma(c) & & \sigma(d) 
\end{array} \right) )_{\sigma \in I_F}  \ \vert t \in F^{\times},
\]

\[
a, b, c, d \in F \ \mbox{such that} \ ad-bc \in \mathbb{Q}^{\times} \} \hookrightarrow{} 
\]

\[ 
\hookrightarrow 
Q_1/W_1 (k) = \{ ( \left( \begin{array} {cccc} 
t_{\sigma} \cdot (a_{\sigma}d_{\sigma}-b_{\sigma}c_{\sigma}) & & & \\
& a_{\sigma} & & b_{\sigma} \\
& & t_{\sigma}^{-1} & \\
& c_{\sigma} & & d_{\sigma}
\end{array} \right) )_{\sigma \in I_F}  \ \vert
t_{\sigma} \in k^{\times} \ \mbox{for every} \ \sigma,
\]
\[
a_{\sigma}, b_{\sigma}, c_{\sigma}, d_{\sigma} \in k \ \mbox{such that} \ a_{\sigma}d_{\sigma}-b_{\sigma}c_{\sigma} = a_{\hat{\sigma}}d_{\hat{\sigma}}-b_{\hat{\sigma}}c_{\hat{\sigma}} \in k^{\times} \ \mbox{for every} \ \sigma, \hat{\sigma} \in I_F
\}.
\]
Thus, there is an isomorphism
\begin{equation} \label{isosiegel}
Q_0/W_0 \simeq \mathbb{G}_m \times \mbox{Res}_{F|\mathbb{Q}} \mbox{GL}_{2,F},
\end{equation}
given on $k$-points by

\begin{align*}
(Q_0/W_0)(k) \simeq \mathbb{G}_{m}(k) &\times (\prod\limits_{\sigma \in I_F} (\mbox{GL}_{2}(k))_{\sigma}) \\
\left( \begin{array} {cc} 
\alpha A_{\sigma} &  \\
 & (A_{\sigma}^{-1})^t 
\end{array} \right) )_{\sigma \in I_F} &\mapsto (\alpha, (A_{\sigma})_{\sigma \in I_F}),
\end{align*}
and an isomorphism

\begin{equation} \label{isoklingen}
Q_1/W_1 \simeq  (\mbox{Res}_{F|\mathbb{Q}} \mbox{GL}_{2,F} \times_{\mbox{Res}_{F|\mathbb{Q}} \mathbb{G}_{m,F}, \mbox{det}} \mathbb{G}_m) \times \mbox{Res}_{F|\mathbb{Q}} \mathbb{G}_{m,F}
\end{equation}
given on $k$-points by

\begin{align*}
(Q_1/W_1)(k) \simeq ((\prod\limits_{\sigma \in I_F} (\mbox{GL}_{2,k} )_{\sigma}) &\times_{\prod\limits_{\sigma \in I_F} (\mathbb{G}_{m, k})_{\sigma}} \mathbb{G}_{m,k} )(k) \times \prod\limits_{\sigma \in I_F} (\mathbb{G}_{m}(k))_{\sigma} \\
(\left( \begin{array} {cccc} 
t_{\sigma} \cdot (a_{\sigma}d_{\sigma}-b_{\sigma}c_{\sigma}) & & & \\
& a_{\sigma} & & b_{\sigma} \\
& & t_{\sigma}^{-1} & \\
& c_{\sigma} & & d_{\sigma}
\end{array} \right) )_{\sigma \in I_F} &\mapsto \left( ( \left( \begin{array}{cc}
a_{\sigma} & b_{\sigma} \\
c_{\sigma} & d_{\sigma}
\end{array} \right) )_{\sigma \in I_F}, \ (t_{\sigma})_{\sigma \in I_F} \right).
\end{align*}

\subsubsection{\textbf{Characters and dominant weights.}}

Consider our fixed Galois closure $L$ of $F$. Using the isomorphism \eqref{grouppoints} (for $k=L$) and Eq. \eqref{pointstore}, we get the following description for the points of the maximal torus $T_L$ of $G_L$:

\begin{equation*}
T_L(L)=\{ (\mbox{diag}(\alpha_{1, \sigma}, \alpha_{2, \sigma}, \alpha_{1, \sigma}^{-1}\nu, \alpha_{2, \sigma}^{-1}\nu))_{\sigma \in I_F} \ \vert \ \alpha_{1, \sigma}, \alpha_{2, \sigma} \in L^*, \nu \in \mathbb{Q}^{\times} \} 
\end{equation*}
This description naturally identifies $T_L$ with a subtorus of rank $2d+1$ of the rank-$3d$ torus $\prod\limits_{\sigma \in I_F} T_{\sigma}$, where each $T_ {\sigma}$ is a copy of the diagonal maximal torus of $\mbox{GSp}_{4, L}$.

The elements $\lambda$ of the group $X^*(T_L)$ of characters (or \q{weights}) of $T_L$ (a subgroup of $\bigoplus\limits_{\sigma \in I_F} X^*(T_{\sigma})$) are then parametrized by the ($2d+1$)-tuples of integers of the form
\begin{equation} \label{paramweight}
((k_{1,\sigma}, k_{2,\sigma})_{\sigma \in I_F}, c) \ \mbox{such that} \ \sum\limits_{\sigma \in I_F} (k_{1,\sigma}+k_{2,\sigma}) \equiv c \ \mbox{(mod} \ 2 \mbox{)}
\end{equation}
where the character $\lambda((k_{1,\sigma}, k_{2,\sigma})_{\sigma \in I_F}, c)$ corresponding to $((k_{1,\sigma}, k_{2,\sigma})_{\sigma \in I_F}, c)$ is defined by
\begin{equation} \label{actcar}
\lambda((k_{1,\sigma}, k_{2,\sigma})_{\sigma \in I_F}, c): (\mbox{diag}(\alpha_{1, \sigma}, \alpha_{2, \sigma}, \alpha_{1, \sigma}^{-1}\nu, \alpha_{2, \sigma}^{-1}\nu))_{\sigma \in I_F} \mapsto \prod\limits_{\sigma \in I_F} \alpha_{1,\sigma}^{k_{1,\sigma}} \cdot \prod\limits_{\sigma \in I_F} \alpha_{2,\sigma}^{k_{2,\sigma}} \cdot \nu^{\frac{1}{2} \ \cdot [c - \sum\limits_{\sigma \in I_F} (k_{1,\sigma}+k_{2,\sigma})] }.
\end{equation}

The \emph{dominant} weights are the characters such that $k_{1,\sigma} \geq k_{2,\sigma} \geq 0 \ \forall \sigma$. A weight is called \emph{regular at} $\sigma$ if $k_{1,\sigma} > k_{2,\sigma} > 0$ and \emph{regular} if it is regular at $\sigma$ for every $\sigma$.

\subsubsection{\textbf{Root system and Weyl group.}} \label{racines} Take the couple $(T,B)$ fixed at the end of \ref{parab} and apply base change to our fixed Galois closure $L$ of $F$. By choosing the resulting couple $(T_L, B_L)$ as maximal torus and Borel, we identify the set of roots $\mathfrak{r}$ of $G_L$ with $\bigsqcup\limits_{\sigma \in I_F} \mathfrak{r}_{\sigma}$, where each $\mathfrak{r}_{\sigma}$ is a copy of the set of roots of $\mbox{GSp}_{4, L}$ corresponding to the diagonal torus and the standard Borel. For every fixed $\hat{\sigma} \in I_F$, $\mathfrak{r}_{\hat{\sigma}}$ contains two simple roots $\rho_{1, \hat{\sigma}}$ and $\rho_{2, \hat{\sigma}}$, which, through the inclusion of $\mathfrak{r}_{\hat{\sigma}}$ into $\mathfrak{r}$, can respectively be written $\rho_{1, \hat{\sigma}} = \rho_{1, \hat{\sigma}}((k_{1,\sigma}, k_{2,\sigma})_{\sigma \in I_F}, c)$, with

\[
k_{1,\sigma}=\left\{ \begin{array}{cc}
1 \ & \mbox{if} \ \sigma = \hat{\sigma} \\
0 \ & \mbox{otherwise}
\end{array} 
\right. ,\
k_{2,\sigma}=\left\{ \begin{array}{cc}
-1 \ & \mbox{if} \ \sigma = \hat{\sigma} \\
0 \ & \mbox{otherwise}
\end{array} 
\right. ,\
c=0,
\]
and $\rho_{2, \hat{\sigma}} = \rho_{2, \hat{\sigma}}((k_{1,\sigma}, k_{2,\sigma})_{\sigma \in I_F}, c)$, with

\[
k_{1,\sigma}=0 \ \ \forall \sigma,\
k_{2,\sigma}=\left\{ \begin{array}{cc}
2 \ & \mbox{if} \ \sigma = \hat{\sigma} \\
0 \ & \mbox{otherwise}
\end{array} 
\right. ,\
c=0.
\]

The Weyl group $\Upsilon$ of $G_L$ is in turn isomorphic to the product $\prod\limits_{\sigma \in I_F} \Upsilon_{\sigma}$, where, for every fixed $\hat{\sigma} \in I_F$, $\Upsilon_{\hat{\sigma}}$ is a copy of the Weyl group of $\mbox{GSp}_{4, L}$. The latter is a finite group of order 8 acting on $X^*(T_{\hat{\sigma}})$, generated by two elements $s_1$ and $s_2$, whose images $s_{\rho_{1, \hat{\sigma}}}$ and $s_{\rho_{2, \hat{\sigma}}}$ through the inclusion $\Upsilon_{\hat{\sigma}}$ into $\Upsilon$ are characterized as follows by their action on the elements of $X^*(T_L)$: if $\lambda=\lambda((k_{1,\sigma}, k_{2,\sigma})_{\sigma \in I_F}, c)$, then $s_{\rho_{1, \hat{\sigma}}}.\lambda=\lambda((h_{1,\sigma}, h_{2,\sigma})_{\sigma \in I_F}, c)$, with

\[
h_{1,\sigma}=\left\{ \begin{array}{cc}
k_{2,\sigma} \ & \mbox{if} \ \sigma = \hat{\sigma} \\
k_{1,\sigma} \ & \mbox{otherwise}
\end{array} 
\right. ,\
h_{2,\sigma}=\left\{ \begin{array}{cc}
k_{1,\sigma} \ & \mbox{if} \ \sigma = \hat{\sigma} \\
k_{2,\sigma}\ & \mbox{otherwise}
\end{array} 
\right.
\]
and $s_{\rho_{2, \hat{\sigma}}}.\lambda=\lambda((h_{1,\sigma}, h_{2,\sigma})_{\sigma \in I_F}, c)$, with

\[
h_{1,\sigma}=k_{1,\sigma} \ \forall \sigma \in I_F,\
h_{2,\sigma}=\left\{ \begin{array}{cc}
-k_{2,\sigma} \ & \mbox{if} \ \sigma = \hat{\sigma} \\
k_{2,\sigma}\ & \mbox{otherwise}
\end{array} 
\right. .
\]
These descriptions mean that $s_{\rho_{1, \hat{\sigma}}}$ corresponds to the reflection associated to $\rho_{1, \hat{\sigma}}$ and that $s_{\rho_{2, \hat{\sigma}}}$ corresponds to the reflection associated to $\rho_{2, \hat{\sigma}}$.

\subsubsection{\textbf{Irreducible representations.}} \label{irred}

Irreducible representations of a split reductive group over a field of characteristic 0 are parametrized by its dominant weights. By the description of the dominant weights of $G_L$ given in \eqref{paramweight}, we see that isomorphism classes of irreducible $L$-representations of $G_L$ are in bijection with the set 

\begin{equation*}
\Lambda:= \{ \lambda((k_{1,\sigma}, k_{2,\sigma})_{\sigma \in I_F}, c)) \ \vert \ k_{1,\sigma}, k_{2,\sigma}, c \in \mathbb{Z} \ \mbox{and} \ k_{1,\sigma} \geq k_{2,\sigma} \geq 0 \ \mbox{for every} \ \sigma, 
\end{equation*}
\begin{equation*}
\sum\limits_{\sigma \in I_F} (k_{1,\sigma}+k_{2,\sigma}) \equiv c \mod 2 \mbox{)} \}.
\end{equation*}

\subsection{Hilbert-Siegel varieties} \label{hilbsieg} 

In this subsection, let $G$ denote the group defined in \ref{defgrp}. We are going to define a \emph{pure Shimura datum} $(G, X)$ in the sense of \cite[Def. 2.1]{Pin90}.

\begin{definition}
The complex analytic space $\mathcal{H}_2$ is the subspace of $M_2(\mathbb{C})$ formed by complex $2\times2$ matrices which are symmetric and whose imaginary part is definite (positive or negative). 
\end{definition}

Recall that $G(\mathbb{R})= \{ (A_{\sigma})_{\sigma \in I_F} \in \prod_{\sigma \in I_F} \mbox{GSp}_{4}(\mathbb{R}) \ \mbox{such that} \ \nu(A_{\sigma})=\nu(A_{\hat{\sigma}}), \ \forall \ \sigma, \hat{\sigma} \in I_F  \}$ and that $\vert I_F \vert=d$. Then, $G(\mathbb{R})$ acts on $\mathcal{H}_2^d$ by analytical isomorphisms. Let $J_2$ be the element of $\mbox{GSp}_{4}(\mathbb{R})$ introduced in \ref{parab}.

\begin{proposition}\label{hilbsiegdatum}
Let $h:\mathbb{S} \rightarrow G_{\mathbb{R}}$ be the morphism defined on real points by
\begin{align*}
\mathbb{S}(\mathbb{R}) & \rightarrow G(\mathbb{R}) \\
x+iy &\mapsto \left( (xI_2 + yJ_2)_{\sigma} \right)_{\sigma \in I_F}
\end{align*}
The $G(\mathbb{R})$-conjugacy class $X$ of $h$ has a canonical structure of complex analytic space (of dimension $3d$), such that there exists a $G(\mathbb{R})$-equivariant isomorphism $X \simeq \mathcal{H}_2^d$ as complex manifolds. Moreover, $(G, X)$ is a pure Shimura datum.
\end{proposition}

Fix now a compact open subgroup $K$ of $G(\mathbb{A}_f)$ which is moreover \emph{neat} (\cite[Section 0.6]{Pin90}). Then, the double quotient $G(\mathbb{Q}) \backslash X \times G(\mathbb{A}_f) /K$ has the structure of a smooth complex analytic variety, which is the analytification of a canonical smooth quasi-projective variety $S(G, X)_K$ (the \emph{Shimura variety} corresponding to the datum $(G, X)$ and to the subgroup $K$), defined over a number field $E(G,X)$, the \emph{reflex field} of $(G, X)$ (independent of $K$). In our case, the reflex field is just $\mathbb{Q}$, $S_K:=S(G, X)_K$ is of dimension $3d$ and it is called the \emph{genus 2 Hilbert-Siegel variety (of level $K$, associated to $F$)}. 

\begin{remark}
The datum $(G, X)$ is a \emph{PEL datum}. In particular, according to \cite[4.12]{Del71}, $S_K$ admits an interpretation as moduli space of abelian varieties of dimension $2d$ with additional structures, among which real multiplication by a sub-algebra $\mathcal{O}$ of $F$ of rank $2d$ over $\mathbb{Z}$, which depends on $K$. Thus, there exists a universal family $p:\mathcal{A}_K \rightarrow S_K$ of abelian varieties over $S_K$. 
\label{rmk:abuniv}
\end{remark}

\subsection{The Baily-Borel compactification} \label{comp}

Let $S_K$ be a Shimura variety corresponding to a general pure Shimura datum $(\mathcal{G}, \mathfrak{X})$ and to a neat compact open subgroup $K \subset \mathcal{G}(\mathbb{A}_f)$. Recall that $S_K$ has a canonical compactification $S_K^*$, called \emph{Baily-Borel compactification}, a projective variety, in general singular, defined over the reflex field $E(\mathcal{G}, \mathfrak{X})$ (\cite[Main Theorem 12.3, (a), (b)]{Pin90}), in which $S_K$ embeds as an open dense sub-scheme. In the case of Hilbert-Siegel varieties, $S_K^*$ is then defined over $\mathbb{Q}$.

\subsubsection{\textbf{Structure of the stratification of $S_K^*$.}} \label{strctcomp}
The variety $S_K^*$ admits a stratification by locally closed strata, amongst which $S_K$ is the only open stratum. The other ones form a stratification of the boundary $\partial S_K^*:=S_K^* \backslash S_K$. If $(\mathcal{Q}_m)_{m \in \Phi}$ is any ordering of any (finite) set of representatives of the conjugacy classes of \emph{admissible parabolic subgroups} (\cite[Def. 4.5]{Pin90}) in $\mathcal{G}$, one sees from \cite[Section 3.6-3.7]{Pin92} that, for every $m \in \Phi$, there exist suitable finite subsets $\mathcal{C}_m$ of $\mathcal{G}(\mathbb{A}_f)$ such that the set of strata of $\partial S_K^*$ is given by
\begin{equation}\label{repstrat}
\{ S_{m,g} \vert m \in \Phi, g \in \mathcal{C}_m \}.
\end{equation}
Here, the locally closed subscheme $S_{m,g}$ of $\partial S_K^*$ is the image of a canonical morphism \begin{equation} \label{defstrat}
i_g: S_{\pi_m(K_{m,g})}:=S_{\pi_m(K_{m,g})}(\mathcal{G}_m, \mathfrak{H}_m) \rightarrow S_K^*
\end{equation}
where the compact open subgroup $\pi_m(K_{m,g})$ and the pure Shimura datum $(\mathcal{G}_m, \mathfrak{H}_m)$ defining the Shimura variety $S_{\pi_m(K_{m,g})}$ are given as follows: there exists a canonical normal subgroup $P_m$ of $Q_m$ (\cite[4.7]{Pin90}) whose unipotent radical $W_m$ coincides with the unipotent radical of $Q_m$ (cfr. \cite[proof of Lemma 4.8]{Pin90}), and we denote $K_{m,g}:=P_m(\mathbb{A}_f) \cap g \cdot K \cdot g^{-1}$, $\pi_m: P_m \rightarrow P_m/W_m$ the natural projection, $(\mathcal{G}_m, \mathfrak{H}_m)$ the pure Shimura datum obtained by quotienting by $W_m$ any of the \emph{rational boundary components} (\cite[4.11]{Pin90}) associated to $P_m$. In particular, $\mathcal{G}_m$ is a reductive subgroup of the Levi component of $Q_m$. In the rest of the paper, we will rather use the stratification of $\partial S_K^*$ indexed by $m \in \Phi$, each of whose strata $Z_m$ corresponds to the $m$-th conjugacy class of admissible parabolics of $\mathcal{G}$ and coincides with the disjoint union of those subschemes $S_{m,g}$ with $g \in \mathcal{C}_m$. The latter will be called \emph{strata of $\partial S_K^*$ contributing to $Z_m$}. 

\begin{remark} \label{c+}
Suppose that the pure Shimura datum $(\mathcal{G}, \mathfrak{X})$ satisfies condition ($+$) from \cite[page 7]{BW04}. Then, it follows from \cite[Lemma 8.2]{Wil17a} that the stratum $S_{m,g}$ equals the quotient of $S_{\pi_m(K_{m,g})}$ by the action of a suitable finite group and is smooth over $E(\mathcal{G}, \mathfrak{X})$.
Thanks to Remark \ref{rmk:centre}.\ref{itm:centre}, the condition is in particular satisfied by the Shimura datum $(G, X)$ of \ref{hilbsiegdatum}. 
\end{remark}

\subsubsection{\textbf{Explicit description of strata in the boundary of genus 2 Hilbert-Siegel varieties.}} \label{strates}
Let us now describe in detail the (pure) Shimura data underlying the strata of $\partial S_K^*$ in the case of the Hilbert-Siegel datum $(G,X)$ of Proposition \ref{hilbsiegdatum}. 

Each admissible parabolic subgroup $Q$ of $G$ is conjugated to exactly one of the subgroups $Q_0$ (Siegel parabolic) or $Q_1$ (Klingen parabolic) defined in \ref{parab}. 

Denote respectively by $P_0$ and $P_1$ the canonical normal subgroups of $Q_0$ and $Q_1$ considered in \ref{strctcomp}. Denote also by $G_0$, resp. $G_1$ their quotients by the respective unipotent radicals, and by $(G_0, \mathfrak{H}_0)$, resp. $(G_1, \mathfrak{H}_1)$ the associated Shimura data. An immediate generalisation to $\mbox{Res}_{F|\mathbb{Q}}\mbox{GSp}_{4,F}$ (and then to $G$) of \cite[4.25]{Pin90} (which treats the case of $\mbox{GSp}_{4}$) gives us the following: 

\begin{itemize}[leftmargin=*]
\renewcommand{\labelitemi}{$\bullet$} 
\item The group $G_0$ is identified with the factor $\mathbb{G}_m$ inside $Q_0/W_0 \simeq \mathbb{G}_m \times \mbox{Res}_{F|\mathbb{Q}} \mbox{GL}_{2,F}$ (remember \eqref{isosiegel}). Moreover, let $k$ be the morphism $\mathbb{S} \rightarrow G_{0, \mathbb{R}}$ which induces on real points, via the above identification, 

\begin{align} 
\mathbb{S}(\mathbb{R}) & \rightarrow G_0(\mathbb{R}) \nonumber \\
z &\mapsto ( \left( \begin{array} {cc} 
z \bar{z} \cdot I_2 &  \\
 & I_2 
\end{array} \right) )_{\sigma \in I_F} \label{donnée0}
\end{align}
and let $\mathfrak{H}_0$ be the set of isomorphisms between $\mathbb{Z}$ and $\mathbb{Z}(1)$. Consider the unique transitive action of $\pi_0(\mathbb{G}_m(\mathbb{R}))$ on $\mathfrak{H}_0$ and denote by $h_0$ the constant map $\mathfrak{H}_0 \rightarrow \{ k \} \subset \mbox{Hom}(\mathbb{S}, G_{0, \mathbb{R}}))$. 
Then, the Shimura datum corresponding to $G_0$ is given by $(G_0, \mathfrak{H}_0)$. 
Thus, $G_0$ contributes with $0$-dimensional strata to $\partial S_K^*$.

\item The group $G_1$ is identified with the factor $\mbox{Res}_{F|\mathbb{Q}} \mbox{GL}_{2,F} \times_{\mbox{Res}_{F|\mathbb{Q}} \mathbb{G}_{m,F}} \mathbb{G}_m$ inside 
\begin{equation*}
Q_1/W_1 \simeq (\mbox{Res}_{F|\mathbb{Q}} \mbox{GL}_{2,F} \times_{\mbox{Res}_{F|\mathbb{Q}} \mathbb{G}_{m,F}} \mathbb{G}_m) \times \mbox{Res}_{F|\mathbb{Q}} \mathbb{G}_{m,F}
\end{equation*}
(remember \eqref{isoklingen}).
Denoting by $\mathfrak{H}_1$ the $G_1(\mathbb{R})$-conjugacy class of the morphism 
\begin{align}
h_1: \mathbb{S}(\mathbb{R}) & \rightarrow G_1(\mathbb{R}) \nonumber \\
x+iy &\mapsto (\left( \begin{array} {cccc} 
x^2+y^2 & & & \\
& x & & y \\
& & 1 & \\
& -y & & x
\end{array} \right) )_{\sigma \in I_F} \label{donnée1}
\end{align}
the Shimura datum corresponding to $G_1$ is then given by $(G_1, \mathfrak{H}_1)$. Thus, $G_1$ contributes with $d$-dimensional strata to $\partial S_K^*$. The description of the Shimura datum shows that these strata are in particular isomorphic to (quotients by the action of a finite group of) \emph{Hilbert-Blumenthal varieties}.

\end{itemize}

By the description in \eqref{repstrat}, each stratum of $\partial S_K^*$ corresponds to a Shimura datum of one of the above two types. In particular, it is either of dimension 0 (and it will be then called a \emph{Siegel stratum}) or of dimension $d$ (and it will be then called a \emph{Klingen stratum}). 

\subsection{The canonical construction functor and its motivic version}
\subsubsection{\textbf{Conventions for Hodge structures.}} \label{convhodge} Let $w: \mathbb{G}_{m, \mathbb{R}} \rightarrow \mathbb{S}$ be the cocharacter which induces the inclusion $\mathbb{R}^{\times} \hookrightarrow \mathbb{C}^{\times}$ on real points. A representation $(\rho, V)$ of $\mathbb{S}$ induces a (semisimple) mixed Hodge structure on the real vector space $V$; coherently with the convention of \cite{Pin90} 1.3, we will say that the subspace of $V$ where $\rho \circ w$ acts as multiplication by $t^{-k}$ is the subquotient of $V$ of \emph{weight k}. 

Moreover, if $(\mathcal{G}, \mathfrak{X})$ is a Shimura datum, defined by $h: \mathfrak{X} \rightarrow \mbox{Hom}(\mathbb{S}, \mathcal{G}_{\mathbb{R}})$, then every representation $\rho: \mathcal{G} \rightarrow \mbox{GL}(V)$ gives rise, for every $x \in \mathfrak{X}$, to a Hodge structure on $V$, by applying the above observation to $\rho \circ h(x) \circ w$.

\subsubsection{\textbf{The canonical construction functor.}} \label{constcan} Let $(\mathcal{G}, \mathfrak{X})$ be a Shimura datum satisfying condition ($+$) from \cite[page 7]{BW04}, $K$ a neat compact open subgroup of $\mathcal{G}(\mathbb{A}_f)$, and $S_K$ the corresponding Shimura variety. Denote by $\mbox{Rep}(\mathcal{G}_R)$ the Tannakian category of algebraic representations of $\mathcal{G}$ in finite dimensional $R$-vector spaces. If $R$ is a subfield of $\mathbb{R}$, we write $\mbox{MVar}_{R}(S_K(\mathbb{C}))$ for the category of \emph{graded-polarizable admissible variations of mixed $R$-Hodge structure} over $S_K(\mathbb{C})$. Then, we have at our disposal (\cite[Part II, Chap. 2]{Wil97}) the (exact tensor) \emph{Hodge canonical construction functor}

\begin{equation}
\mu_{H}^{K}: \mbox{Rep}(\mathcal{G}_R) \rightarrow \mbox{MVar}_{R}(S_K(\mathbb{C})).
\end{equation}
Moreover, let $R$ be finite over $\mathbb{Q}$, let $\ell$ be a fixed prime and fix a prime $l$ of $R$ above $\ell$. Write $\mbox{Et}_{\ell, R}(S_K)$ for the $R_l$-linear version of the category of \emph{lisse $\ell$-adic sheaves} over $S_K$. We have then (\cite[Part II, Chap. 4]{Wil97}) the (exact tensor) \emph{$\ell$-adic canonical construction functor}

\begin{equation}
\mu_{\ell}^{K}: \mbox{Rep}(\mathcal{G}_R) \rightarrow \mbox{Et}_{\ell, R}(S_K).
\end{equation}

\begin{remark} 
\begin{enumerate} [wide, labelwidth=!, labelindent=0pt, label=(\arabic*)]
\item \label{itm:poidsfilt} If $(\mathcal{G}, \mathfrak{X})$ is a Shimura datum of \emph{abelian type}, and $S_K$ any of the corresponding Shimura varieties, then the functor $\mu_{\ell}^{K}$ takes values in the full subcategory $\mbox{Et}^M_{\ell,R}(S_K)$ of $\mbox{Et}_{\ell, R}(S_K)$ formed by the \emph{mixed lisse sheaves with weight filtration}, in the sense of \cite[Part I, Definition before Theorem 2.8]{Wil97}. Moreover, if $V \in \mbox{Rep}(\mathcal{G}_R)$, then the weights of $\mu_{\ell}^{K}(V)$ in the sense of the latter definition are identical to the weights of $\mu_{H}^{K}(V)$ as a variation of mixed Hodge structure. These facts follow from \cite[Proposition (5.6.2)]{Pin92}. 
The hypothesis is in particular satisfied by PEL type Shimura data, and so by the data $(G, X)$, $(G_0, \mathfrak{H}_0)$ and $(G_1, \mathfrak{H}_1)$ defined in Subsections \ref{hilbsieg} and \ref{comp}.

\item \label{itm:poidsrep} If $G$ is the group defined in \ref{defgrp} and $V_{\lambda}$ is an irreducible representation of $G$ of highest weight $\lambda=\lambda((k_{1,\sigma}, k_{2,\sigma})_{\sigma \in I_F}, c))$, then $\mu_{H}^{K}(V_{\lambda})$, resp. $\mu_{\ell}^{K}(V_{\lambda})$, is a variation of Hodge structure, resp. an $\ell$-adic sheaf, pure of weight $w(\lambda):=-c$ (cfr. the convention fixed in \ref{convhodge}, which is extended to variations of Hodge structure in the obvious way).  
\end{enumerate}

\label{rmk:poidsrep}
\end{remark}

\begin{remark} 
\begin{enumerate} [wide, labelwidth=!, labelindent=0pt, label=(\arabic*)]
\item \label{itm:extgrad} Let $\mbox{Var}_{R}(S_K(\mathbb{C}))$ denote the category of \emph{pure} polarizable variations of $R$-Hodge structure. In the following, we will abusively denote by the same symbol $\mu_{H}^{K}$ the obvious factorization of the latter functor through the category $\mbox{Gr}_{\mathbb{Z}}\mbox{Var}_{R}(S_K(\mathbb{C}))$. Analogously, if $(\mathcal{G}, \mathfrak{X})$ is a Shimura datum of abelian type, Rmk. \ref{rmk:poidsrep}.\ref{itm:poidsfilt} will give us an obvious factorisation of the functor $\mu_{\ell}^K$ through the category $\mbox{Gr}_{\mathbb{Z}}\mbox{Et}_{\ell}(S_K)$, still abusively denoted by $\mu_{\ell}^K$.

\item \label{itm:exttriang} The exact functor $\mu_{\ell}^K$ extends to a triangulated functor, denoted by the same symbol, 

\begin{equation} \label{triangext}
\mu_{\ell}^K: D^b(\mbox{Rep}(\mathcal{G}_R)) \rightarrow D^b_c(S_K)_R,
\end{equation}
where $D^b_c(S_K)_R$ is the $R_l$-linear version of the \q{derived} bounded category of $\ell$-adic constructible sheaves over $S_K$ (\cite[Section 6]{Eke90}).
\end{enumerate}
\label{rmk:extcan}
\end{remark}

\subsubsection{\textbf{The motivic version of the canonical construction.}} \label{motcan}

Let us adopt again the notation of the beginning of \ref{comp}, applied to the Hilbert-Siegel Shimura datum $(G,X)$ of \ref{hilbsiegdatum}. Recall, for a base scheme $X$ (say, for our purposes, a separated, finite type $\mathbb{Q}$-scheme), the triangulated $R$-linear category $\DBcM(X)_R$ of \emph{constructible Beilinson motives} over $X$ (\cite{CD12}) and the \emph{$\ell$-adic realization functor} $\mathcal{R}_{\ell}$ on it, with values in the category $D^b_c(X)_R$ of \eqref{triangext} (\cite[Sec. 7.2]{CD16}). In our case, we will consider a base $\mathcal{S} \in \{S_K, S^*_K, \partial S^*_K \} \cup \{ \mbox{strata of} \ \partial S^*_K \}$. Then, composition with the collection of cohomology functors, resp. perverse cohomology functors, $R^*:D^b_c(\mathcal{S})_R \rightarrow \mbox{Gr}_{\mathbb{Z}}\mbox{Et}_{\ell, R}(\mathcal{S})$, resp. $\mathcal{H}^*:D^b_c(\mathcal{S})_R \rightarrow \mbox{Gr}_{\mathbb{Z}}\mbox{Perv(Et)}_{\ell,R}(\mathcal{S})$ (where $\mbox{Perv(Et)}_{\ell,R}(\mathcal{S})$ is the $R_l$-linear version of the category of \emph{ $\ell$-adic perverse sheaves} over $\mathcal{S}$), gives rise to the $\ell$-adic \emph{cohomological realization}, resp. \emph{perverse cohomological realization} functors. 

Consider in particular the case $\mathcal{S}=S_K$. The $R$-linear, tensor pseudo-abelian category of \emph{Chow motives over} $S_K$ (\cite{CH00}), denoted by $CHM(S_K)_R$, faithfully embeds into $\DBcM(S_K)_R$ (more on this in paragraph \ref{weighstr}). If $\mathcal{S}=S_K$, the restriction of $R^* \circ \mathcal{R}_{\ell}$ to $CHM(S_K)_R$ (still denoted by $\mathcal{R}_{\ell}$) equals the usual $\ell$-adic cohomological realization on this category; over $CHM(S_K)_R$, there also exists the  \emph{Hodge cohomological realization}, with values in the category $\mbox{Gr}_{\mathbb{Z}}\mbox{Var}_{R}(S_K(\mathbb{C}))$ of Rmk. \ref{rmk:extcan}. \ref{itm:extgrad}, denoted by $\mathcal{R}_H$ (one constructs such realizations on relative Chow motives as in \cite[1.8]{DM91}).  

Recall now the universal family $p: \mathcal{A}_K \rightarrow S_K$ from Remark \ref{rmk:abuniv}. The following result, valid for every PEL-type Shimura variety, is crucial: 

\begin{theorem}{(\cite[Thm. 8.6]{Anc15}, stated as in \cite[Thm. 5.1]{Wil19a})} \label{constmot}

Let $R$ be a subfield of $\mathbb{R}$. There exists a $R$-linear tensor functor 
\begin{equation}
\tilde{\mu}: \mbox{\emph{Rep}}(G_R) \rightarrow CHM(S_K)_R
\end{equation}
with the following properties:

\begin{enumerate}
\item The composition of $\tilde{\mu}$ with the Hodge cohomological realization is isomorphic to $\mu_H^{K}$ (with the convention of Rmk. \ref{rmk:extcan}. \ref{itm:extgrad}).
\item for every prime $\ell$, the composition of $\tilde{\mu}$ with the $\ell$-adic cohomological realization is isomorphic to $\mu_{\ell}^K$ (with the convention of Rmk. \ref{rmk:extcan}. \ref{itm:extgrad}).
\item $\tilde{\mu}$ commutes with Tate twists.
\item If $V$ is the standard $G_R$-representation on $R^{\oplus 4}$, then $\tilde{\mu}$ sends $V$ to the dual of the Chow motive $p_*^1 \mathbbm{1}_{\mathcal{A}_K}$ over $S_K$ (the first \emph{Chow-Künneth component} of the Chow motive $p_* \mathbbm{1}_{\mathcal{A}_K}$ over $S_K$, cfr. \cite[Thm. 3.1]{DM91}). 
\end{enumerate}
\end{theorem}

\begin{remark} \label{factpowab}
For every positive integer $n$, let $p_{n}: \mathcal{A}_K^{n} \rightarrow S_K$ be the $n$-fold fibred product of $\mathcal{A}_K$ with itself over $S_K$. Observe that the direct sum $V \oplus V^{\vee}$ of the standard representation $V$ with its dual generates the Tannakian category $\mbox{Rep}(G_R)$, by taking tensor products and direct summands. Hence, Theorem \ref{constmot} implies that every object in the essential image of $\tilde{\mu}$ is isomorphic to a finite direct sum $\bigoplus \limits_{i} M_i$, where each $M_i$ is a direct factor of a Tate twist of a Chow motive of the form $p_{n_i,*} \mathbbm{1}_{\mathcal{A}_K^{n_i}}$, for suitable $n_i$'s. 
\end{remark}

Let now $V_{\lambda}$ be a irreducible $L$-representation of $G_L$ of highest weight $\lambda$, where the latter is as in \ref{irred}. 

\begin{definition} \label{motrep}
The Chow motive $^{\lambda}\mathcal{V}$ over $S_K$ is defined by 
\begin{equation*}
^{\lambda}\mathcal{V}:=\tilde{\mu}(V_{\lambda}).
\end{equation*}
\end{definition}

\begin{remark}
\begin{enumerate} [wide, labelwidth=!, labelindent=0pt, label=(\arabic*)]
\item \label{itm:degclss} Since $\mu_H^{K}(V_{\lambda})$ and $\mu_{\ell}^K(V_{\lambda})$ are pure of weight $w(\lambda)$ (cfr. \ref{rmk:poidsrep}.\ref{itm:poidsrep}), the Hodge, resp. $\ell$-adic cohomological realizations of $^\lambda \mathcal{V}$ are zero in degree $ \neq w(\lambda)$, and identical to $\mu_H^{K}(V_{\lambda})$, resp. $\mu_{\ell}^K(V_{\lambda})$, in degree $w(\lambda)$. 
\item \label{itm:degperv} Since $S_K$ is a variety of dimension $3d$, \ref{itm:degclss} can be reformulated by saying that the \emph{perverse} Hodge, resp. $\ell$-adic cohomological realizations are zero in perverse degree $ \neq w(\lambda)+3d$, and identical to $\mu_H^{K}(V_{\lambda})$, resp. $\mu_{\ell}^K(V_{\lambda})$, in perverse degree $w(\lambda)+3d$. 
\item \label{itm:autodual} Let $\mathbb{D}_{\ell, S_K}$ denote the $\ell$-adic local duality endofunctor over $S_K$. Then
\begin{equation}\label{realshift}
 R_{\ell, S_K}(^\lambda \mathcal{V})=\mu_{\ell}^K(V_{\lambda})[-w(\lambda)]
\end{equation}
and $\mathbb{D}_{\ell, S_K}(R_{\ell, S_K}(^\lambda \mathcal{V}))\simeq R_{\ell, S_K}(^\lambda \mathcal{V})(w(\lambda)+3d))[2w(\lambda)+6d]$, as can be seen by imitating the first part of \cite[proof of Thm. 1.6]{Wil19b}.
\end{enumerate}
\label{rmk:degdual}
\end{remark}

\subsection{A criterion for the existence of the intersection motive} \label{critexistence}
Let $^{\lambda} \mathcal{V}$ be one of the Chow motives over the Hilbert-Siegel variety $S_K$ introduced in \ref{motcan}, and let $\tilde{s}:S_K \rightarrow \Spec \BQ$ be the structural morphism. Having in mind the problem of defining motives associated to cuspidal automorphic representations, we want to construct the \emph{lowest weight-graded object} of $\tilde{s}_* \ ^{\lambda} \mathcal{V}$ as a canonical Chow motive over $\Spec \BQ$. The first aim of this subsection is to recall a criterion (Theorem \ref{critint_a}) which allows one to perform this construction, formulated in the language of \emph{weight structures}. The second aim is to gather some tools (paragraphs \ref{pink}-\ref{cohunits}) which we will use in order to characterize the validity of this criterion in the case of genus 2 Hilbert-Siegel varieties. This characterization (Corollary \ref{evit01}) will then be a consequence of our main result, Theorem \ref{thm:mainthm}.   

\subsubsection{\textbf{Weight structures and a criterion for the weight avoidance}} \label{weighstr}

Fix a subfield $R$ of $\mathbb{R}$ and a scheme $X$ of finite type over $\mathbb{Q}$. According to \cite{Héb11}, the category $\DBcM(X)_R$ (see \ref{motcan}) is equipped with a canonical \emph{weight structure}, the \emph{motivic weight structure}, whose \emph{heart} (the subcategory of weight 0 objects) is denoted by $CHM(X)_R$ and called the ($R$-linear version of the) category of \emph{Chow motives over} $X$. If $X=S_K$ as in paragraph \ref{motcan}, then, after \cite{Fan16}, this category is equivalent to the homonymous category introduced in that paragraph. 

\begin{definition}{(cfr. \cite[Defs. 1.6-1.10]{Wil09})} \label{def:evit}
Let $M \in \DBcM(X)_R$ and let $\alpha$, $\beta$ be integers. We say that $M$ \emph{avoids weights} $\alpha, \dots, \beta$ if $\alpha \leq \beta$ and there exists an exact triangle in $\DBcM(X)_R$
\[
M_{\leq \alpha-1} \rightarrow M \rightarrow M_{\geq \beta + 1} \rightarrow M_{\leq \alpha-1}[1]
\]
such that $M_{\leq \alpha-1}$ is of weight at most $\alpha-1$ and $M_{\geq \beta + 1}$ of weight at least $\beta+1$. 

Such a triangle is called a \emph{weight filtration} of $M$ \emph{avoiding weights} $\alpha, \dots, \beta$.  
\end{definition}

Let now $V_{\lambda}$ be the irreducible representation of $G_L$ of highest weight 
\begin{equation*}
\lambda=\lambda((k_{1,\sigma}, k_{2,\sigma})_{\sigma \in I_F}, c)),
\end{equation*}
(defined in \eqref{irred}), $S_K$ the genus 2 Hilbert-Siegel variety of level $K$ corresponding to $(G,X)$ and $^{\lambda}\mathcal{V} \in CHM(S_K)$ the Chow motive over $S_K$ introduced in Definition \ref{motrep}. Let moreover $j:S_K \rightarrow S_K^*$, resp. $i:\partial S_K^* \rightarrow S_K^*$ denote the open, resp. closed immersion in the Baily-Borel compactification $S_K^*$ of $S_K$ (as in Subsection \ref{comp}). The theory of \cite{Wil19a} tells us that we will be able to construct the desired Chow motive described at the beginning of this subsection (the \emph{intersection motive}) whenever the Beilinson motive $i^* j_* ^\lambda \mathcal{V}$  over $\partial S_K^*$ \emph{avoids weights 0 and 1}.  

For a discussion of the definition of the intersection motive when the weights 0 and 1 are avoided, of its properties, and of the applications to the construction of motives associated to automorphic representations, we refer to Section \ref{consapp}. Here we recall instead a criterion which allows one to prove the avoidance of such weights, by looking at the structure of $\partial S_K^*$. By \ref{strctcomp}, $\partial S_K^*$ admits a natural stratification $ \partial S_K^* = Z_0 \sqcup Z_1$, indexed by $\Phi:= \{ 0, 1 \}$. We denote by $i_0: Z_0 \hookrightarrow \partial S_K^*$ the immersion of the disjoint union of strata corresponding to the conjugacy class of the parabolic $Q_0$ (i.e. of the closed, 0-dimensional Siegel strata), and by $i_1: Z_1 \hookrightarrow \partial S_K^*$ the immersion of the disjoint union of strata corresponding to the conjugacy class of the parabolic $Q_1$ (i.e. of the open, $d$-dimensional Klingen strata). Given this, one can prove the weight avoidance by reduction to a \emph{stratum-by-stratum} study of the weights (over $\partial S_K^*$) of the $\ell$-adic realization of $i^* j_* ^{\lambda}\mathcal{V}$. Consider in fact the \emph{intermediate extension} functor $j_{!*}$ from the category $\mbox{Perv(Et)}_{\ell, L}(S_K)$ to the category $\mbox{Perv(Et)}_{\ell, L}(S_K^*)$ (cfr. \ref{motcan}). Then, \cite{Wil19a} gives us the sought-for criterion: 

\begin{theorem}\label{critint_a}

Let $\beta \geq 1$ be an integer and fix a prime number $\ell$. For any stratum $\mathcal{Z}$ of $\partial S_K^*$, denote by $\mathcal{H}^n$ the $n$-th perverse cohomology functor on $D^b_c(\mathcal{Z})_L$ and write $j_{!*}(\mathcal{R}_{\ell}(^{\lambda}\mathcal{V}))$ for 
\begin{equation*}
\left( j_{!*}(\mathcal{R}_{\ell}(^{\lambda}\mathcal{V})[w(\lambda)+3d]) \right)[-w(\lambda)-3d].
\end{equation*}
The following assertions are then equivalent:
\begin{enumerate} [label=(\arabic*)]
\item the motive $i^*j_*^{\lambda}\mathcal{V}$ avoids weights $-\beta+1, -\beta +2, \dots, \beta$;
\item \label{realpoids} for every $n \in \mathbb{Z}$, $\mathcal{H}^n i_0^* i^* j_{!*}(\mathcal{R}_{\ell}(^{\lambda}\mathcal{V}))$ and $\mathcal{H}^n i_1^* i^* j_{!*}(\mathcal{R}_{\ell}(^{\lambda}\mathcal{V}))$ are of weights $\leq n- \beta$.
\end{enumerate}
\end{theorem}
\begin{proof} Reasoning exactly as in the proof of \cite[Thm. 2.2]{Wil19b}, we see that the motive $i^* j_* ^{\lambda}\mathcal{V} \in \DBcM(\partial S_K^*)_L$ is \emph{of abelian type}, and that the stratification $\Phi$ is \emph{adapted to} $i^* j_* ^{\lambda}\mathcal{V}$ (\cite[Def. 2.1]{Wil19b}). Then, it suffices to use the fact that the $\ell$-adic realization of the motive $^{\lambda}\mathcal{V}$ is concentrated in only one perverse degree (Remark \ref{rmk:degdual}.\ref{itm:degperv}) together with its auto-duality (up to a twist and a shift) (Remark \ref{rmk:degdual}.\ref{itm:autodual}) in order to apply \cite[Corollary (4.6)(b)]{Wil19a}. 
\end{proof}

\subsubsection{\textbf{Pink's theorem.}} \label{pink}

With notation as in the previous paragraph, Thorem \ref{critint_a} leads us to analyze the weights of the perverse sheaves 
\begin{equation*}
\mathcal{H}^n i_m^* i^* j_{!*}(\mathcal{R}_{\ell}(^{\lambda}\mathcal{V})), 
\end{equation*}
for $m \in \{0, 1 \}$. This will be done by studying first some strongly related \emph{classical} sheaves, which we describe now in the general setting. Hence, let $j:S_K \hookrightarrow S_K^*$ be the open immersion of a Shimura variety $S_K$, associated to a datum $(\mathcal{G}, \mathfrak{X})$ and to a neat compact open subgroup $K \subset \mathcal{G}(\mathbb{A}_f)$, into its Baily-Borel compactification. Recall the finite stratification $(Z_m)_{m \in \Phi}$ of $\partial S^*_K$ introduced in \ref{strctcomp} and for $m \in \Phi$, denote by $i_m: Z_m \hookrightarrow \partial S^*_K$ the corresponding locally closed immersion. In the case of the Hilbert-Siegel Shimura datum of \ref{hilbsiegdatum}, this gives back the notation used to state Thm. \ref{critint_a}. 

In the following, with the notation of \ref{strctcomp}, denote by $Z$ a fixed stratum $S_{m,g}$ of $\partial S_K^*$ contributing to $Z_m$, and denote by $\pi_m(K_m)$ the associated compact open subgroup $\pi_m(K_{m,g})$ of $\mathcal{G}_m(\mathbb{A}_f)$ (i.e., drop the subscript $g$), so that $Z$ is the quotient of a Shimura variety $S_{\pi_m(K_m)}$. Associated to $Z$, there is a non-trivial arithmetic subgroup $\Gamma_m$ of $Q_m/W_m(\mathbb{Q})$, as defined in \cite[Sec. 2]{BW04}) (where it is denoted by $\bar{H}_C$).   

\begin{notation} \label{groupcompl}
Since $Q_m/W_m$ is reductive, there exists a \emph{complement} $M_m$ of $\mathcal{G}_m$ inside $Q_m/W_m$, i.e. a normal subgroup $M_m$ of $Q_m/W_m$ which is connected and reductive and such that $Q_m/W_m \simeq \mathcal{G}_m \cdot M_m$, with $\mathcal{G}_m \cap M_m$ finite.
\end{notation}

\begin{remark} \label{compl}
Since $K$ is neat, $\Gamma_m$ is torsion-free. Moreover, it is such that $\Gamma_m \cap \mathcal{G}_m (\mathbb{Q})= \{ 1 \}$ (cfr. \cite[Sec. 2]{BW04}). We will then see $\Gamma_m$ as a subgroup of the complement $M_m(\mathbb{Q})$ introduced above. 
\end{remark}

Denote now by $\mu_{\ell}^K$, $\mu_{\ell}^{\pi_m(K_m)}$ the extensions of the $\ell$-adic canonical construction functors introduced in Remark \ref{rmk:extcan}.\ref{itm:exttriang}, and by $R^n$ the $n$-th \emph{classical}, i.e. non-perverse, cohomology functor on the category $D^b_c(Z_m)_L$ (see \ref{motcan}), for any stratum $Z_m$ of $\partial S_K^*$. Our first main tool for the analysis of the weights is the following theorem of Pink:
 
\begin{theorem}{(\cite[Thms. (4.2.1)-(5.3.1)]{Pin92}, stated in the shape of \cite[Thms. 2.6-2.9]{BW04})}

Let $R$ be a subfield of $\mathbb{R}$, $\mathbb{V}^{\boldsymbol{\cdot}} \in D^b(\mbox{\emph{Rep}}_R(\mathcal{G}))$, $m \in \Phi$ and $Z$ a stratum of $\partial S_K^*$ contributing to $Z_m$.
\begin{enumerate} [wide, labelwidth=!, labelindent=0pt, label=(\arabic*)]
\item \label{itm:decder}
There exists a canonical isomorphism in $D^b_c(Z)_R$  

\[
\restr{i^*_m i^* j_* \mu_{\ell}^{K}(\mathbb{V}^{\boldsymbol{\cdot}})}{Z} \simeq \bigoplus\limits_{n} \restr{R^n i^*_m i^* j_* \mu_{\ell}^K(\mathbb{V}^{\boldsymbol{\cdot}})}{Z} [-n].
\]
\item \label{itm:decobj}
For every $n$, there exists a canonical and functorial isomorphism in $\mbox{\emph{Et}}_{\ell, R}(Z)$

\[
\restr{R^n i^*_m i^* j_* \mu_{\ell}^K(\mathbb{V}^{\boldsymbol{\cdot}})}{Z}\simeq\bigoplus_{\substack{p+q=n}} \mu_{\ell}^{\pi_m(K_m)} \left( H^p(\Gamma_m, H^q(W_{m,R}, \mathbb{V}^{\boldsymbol{\cdot}})) \right).
\]

\item \label{itm:gradobj}
Suppose that the datum $(\mathcal{G}_m, \mathfrak{H}_m)$ is of abelian type. Then, denoting by $\mathbb{W}$ both the weight filtration in the sense of Remark \ref{rmk:poidsrep}.\ref{itm:poidsfilt} and the one induced on $\mathcal{G}_{m,R}$-representations as explained in \ref{convhodge}, the sheaf $\restr{R^n i^*_m i^* j_* \mu_{\ell}^K(\mathbb{V}^{\boldsymbol{\cdot}})}{Z}$ is the direct sum of its weight-graded objects (in particular, it is a semisimple object) and there exist canonical and functorial isomorphisms in $\mbox{\emph{Et}}_{\ell, R}(Z)$ 

\[
\mbox{\emph{Gr}}_k^{\mathbb{W}} \restr{R^n i^*_m i^* j_* \mu_{\ell}^K(\mathbb{V}^{\boldsymbol{\cdot}})}{Z}\simeq\bigoplus_{\substack{p+q=n}} \mu_{\ell}^{\pi_m(K_m)} \left( H^p(\Gamma_m, \mbox{\emph{Gr}}_k^{\mathbb{W}} H^q(W_{m,R}, \mathbb{V}^{\boldsymbol{\cdot}})) \right).
\]

\end{enumerate} 
\label{thm:PinkThm}
\end{theorem}

In order to explain the above statements, some remarks are in order:
\begin{remark}
\begin{enumerate}[wide, labelwidth=!, labelindent=0pt, label=(\arabic*)]
\item \label{itm:foncstrat}  Reasoning as in \cite{BW04}, before Definition 2.2, we see that the functor $\mu_{\ell}^{\pi_m(K_m)}$, \emph{a priori} with values in $\mbox{Et}_{\ell, R}(S_{\pi_m(K_m)})$, gives rise to a functor with values in $\mbox{Et}_{\ell, R}(Z)$, still denoted by the same symbol.
\item $(Q_{m}/W_{m})_R$ (seen as a subgroup of $Q_{m,R}$ via the Levi decomposition) acts on $H^q(W_{m,R}, \mathbb{V}^{\boldsymbol{\cdot}})$ via its action on $W_m$ and on $\mathbb{V}^{\boldsymbol{\cdot}}$, and so it acts on $H^p(\Gamma_m, H^q(W_{m,R}, \mathbb{V}^{\boldsymbol{\cdot}}))$. Hence, the latter space is seen as a representation of $\mathcal{G}_{m,R}$ via the inclusion $\mathcal{G}_{m,R} \subset (Q_{m}/W_{m})_R$.
\item \label{itm:weightfilt} The statement of point \ref{itm:gradobj} contains in particular the fact that 
\begin{equation*}
\mbox{Gr}_k^{\mathbb{W}} H^p(\Gamma_m, H^q(W_{m,R}, \mathbb{V}^{\boldsymbol{\cdot}})) \simeq H^p(\Gamma_m, \mbox{Gr}_k^{\mathbb{W}} H^q(W_{m,R}, \mathbb{V}^{\boldsymbol{\cdot}})).
\end{equation*}
\end{enumerate}
\label{rmk:explpink}
\end{remark}

\subsubsection{\textbf{Kostant's theorem.}} \label{kost}
The second ingredient for the analysis of the weights is a theorem of Kostant which allows one to make explicit the $(Q_{m}/W_{m})_R$-representations $H^q(W_{m,R}, \mathbb{V}^{\boldsymbol{\cdot}})$ appearing in Theorem \ref{thm:PinkThm}.

Fix a split reductive group $\mathcal{G}$ over a field of characteristic zero, with root system $\mathfrak{r}$ and Weyl group $\Upsilon$. Denote by $\mathfrak{r}^+$ the set of positive roots and fix moreover a parabolic subgroup $Q$ with its unipotent radical $W$. Let $\mathfrak{w}$ be the Lie algebra of $W$ and $\mathfrak{r}_W$ the set of roots appearing inside $\mathfrak{w}$ (necessarily positive). For every $w \in \Upsilon$, we define:

\begin{equation}
\mathfrak{r}^+(w):=\{\alpha \in \mathfrak{r}^+ \vert w^{-1} \alpha \notin \mathfrak{r}^+ \},
\end{equation}

\begin{equation}
l(w):=\vert \mathfrak{r}^+(w) \vert,
\end{equation} 

\begin{equation} \label{subweyl}
\Upsilon^\prime:= \{ w \in \Upsilon \vert \mathfrak{r}^+(w) \subset \mathfrak{r}_W \}.
\end{equation}   

We can now state Kostant's theorem:
\begin{theorem}{(\cite[Thm. 3.2.3]{Vog81})} \label{KostThm}

Let $V_{\lambda}$  be an irreducible $\mathcal{G}$-representation of highest weight $\lambda$, and let $\rho$ be the half-sum of the positive roots of $\mathcal{G}$. Then, as $(Q/W)$-representations,
\[
H^q(W, V_{\lambda}) \simeq \bigoplus \limits_{w \in \Upsilon^\prime \vert l(w)=q} U_{w.(\lambda+\rho)-\rho},
\]
where $U_{\mu}$ denotes an irreducible $(Q/W)$-representation of highest weight $\mu$. 
\end{theorem}

In order to spell out the consequences of this theorem in the cases of interest to us, consider our fixed Galois closure $L$ of the totally real field $F$ and the group $G$ underlying the Hilbert-Siegel Shimura datum of \ref{hilbsiegdatum}. We will apply Kostant's theorem by choosing $\mathcal{G}=G_L$ and by setting $Q/W$ equal to, for $i=0,1$, the Levi components $(Q_i/W_i)_L$ of the standard parabolics $Q_{i,L}$, defined as in \ref{levicomp}. We have seen in \ref{racines} that the root system of $G_L$ is given by $\mathfrak{r}=\bigsqcup\limits_{\sigma \in I_F} \mathfrak{r}_{\sigma}$ and that every component $\mathfrak{r}_{\sigma}$ contains two simple roots $\rho_{1, \sigma}$, $\rho_{2, \sigma}$; the other positive roots in such a component are then given by $\rho_{1, \sigma}+\rho_{2, \sigma}$ and $2 \rho_{1, \sigma}+\rho_{2, \sigma}$.

\begin{lemma} 
\begin{enumerate} [wide, labelwidth=!, labelindent=0pt, label=(\arabic*)]
\item \label{itm:partweyl} Let $\Upsilon$ be the Weyl group of $G_L$ (cfr. \ref{racines}) and denote by $\Upsilon_m^\prime$, for $m\in\{ 0,1 \}$, the sets defined in \eqref{subweyl}, corresponding to the choices $\mathcal{G}=G_L$ and $Q=Q_{m,L}$. Then, in both cases, for every $\sigma \in I_F$, there exist sets $\Upsilon_{m, \sigma}^\prime = \{ w_{\sigma}^i \}_{i=0, \dots, 3} \subset \Upsilon_{\sigma}$ such that $l(w_{\sigma}^i)=i$ for every $i \in \{ 0, \dots, 3 \}$ and that $\Upsilon_m^\prime=\prod\limits_{\sigma \in I_F} \Upsilon_{m, \sigma}^\prime$. 

\item For $m=0,1$, one has $0 \leq l(w) \leq 3d$ for every $w \in \Upsilon_m^\prime$. Moreover, for every integer $q \in \{0, \dots, 3d \} $, there exists a bijection between the set $\{ w  \in \Upsilon_m^\prime \ \vert \ l(w)=q \}$ and the set of $q$-\emph{admissible decompositions} of $I_F$

\begin{equation} \label{admpart}
\mathcal{P}_q:= \{ \mbox{decompositions} \ I_F= \bigsqcup_{\substack{i=0, \dots, 3}} I_F^i \ \vert \sum\limits_{i=0}^3 i \vert I^i_F \vert =q \}.
\end{equation}
\end{enumerate}
\label{lemma:partweyl}
\end{lemma}

\begin{proof}
\begin{enumerate} [wide, labelwidth=!, labelindent=0pt, label=(\arabic*)]
\item In the case of $(Q_0/W_0)_L$, by fixing a component $\mathfrak{r}_{\sigma}$ of the root system of $G_L$ (cfr. \ref{racines}), one easily sees that the positive roots which are contained in such a component and which appear in the Lie algebra of $W_{0,L}$ are given by $\{ \rho_{1, \sigma}+\rho_{2, \sigma}, 2 \rho_{1, \sigma}+\rho_{2, \sigma}, \rho_{2, \sigma} \}$. 

Coherently with the notation of \ref{racines}, denote by $s_{\rho}$ the reflection, belonging to the Weyl group $\Upsilon$, whose axis is orthogonal to the root ${\rho}$. By direct inspection of the action of the component $\Upsilon_{\sigma}$ on $\mathfrak{r}_{\sigma}$, we find that $\Upsilon_0^\prime=\prod\limits_{\sigma \in I_F} \Upsilon_{0, \sigma}^\prime$, where the elements of the sets $\Upsilon_{0, \sigma}^\prime = \{ w_{\sigma}^i \}_{i=0, \dots, 3}$ are given by 

\[
w_{\sigma}^0 = \mbox{id}, 
\]
\[
w_{\sigma}^1 = s_{\rho_{2, \sigma}}, 
\]
\[
w_{\sigma}^2 = s_{\rho_{1, \sigma}+\rho_{2, \sigma}} s_{\rho_{2, \sigma}},
\]
\[
w_{\sigma}^3 = s_{\rho_{1, \sigma}+\rho_{2, \sigma}}.
\] 
These are the sets defined in the statement (for $m=0$).

In the case of $(Q_1/W_1)_L$, fix again a component $\mathfrak{r}_{\sigma}$ of the root system of $G_L$: the positive roots contained in this component which appear inside the Lie algebra of $W_{1,L}$ are given this time by $\{ \rho_{1, \sigma}, \rho_{1, \sigma}+\rho_{2, \sigma}, 2 \rho_{1, \sigma}+\rho_{2, \sigma} \}$. 

With notations as in the previous case, we find that $\Upsilon_1^\prime=\prod\limits_{\sigma \in I_F} \Upsilon_{1, \sigma}^\prime$, where the elements of the sets $\Upsilon_{1, \sigma}^\prime = \{ w_{\sigma}^i \}_{i=0, \dots, 3} \subset \Upsilon_{\sigma}$ are given by

\[
w_{\sigma}^0 = \mbox{id}, 
\]
\[
w_{\sigma}^1 = s_{\rho_{1, \sigma}}, 
\]
\[
w_{\sigma}^2 = s_{\rho_{1, \sigma}+\rho_{2, \sigma}} s_{\rho_{1, \sigma}},
\]
\[
w_{\sigma}^3 = s_{2\rho_{1, \sigma}+\rho_{2, \sigma}}.
\]
Again, these are the sets appearing in the statement (for $m$=1). 

\item The preceding point implies that every $w=(w_{\sigma})_{\sigma \in I_F} \in \Upsilon_m^\prime$ determines a decomposition 
\begin{equation*}
I_F= \bigsqcup_{\substack{i=0, \dots, 3}} I_F^i
\end{equation*}  
where $I_F^i:= \{ \sigma \in I_F \vert w_{\sigma}=w_{\sigma}^i \}$. Hence, since $l((w_{\sigma})_{\sigma \in I_F})= \sum \limits_{\sigma \in I_F} l(w_{\sigma})$, we get the desired bounds on $l(w)$. The bijection in the statement follows immediately. 
\end{enumerate}
\end{proof}

\begin{notation} \label{notadmpart}
For a integer $q \in \{0, \dots, 3d \}$, a $q$-admissible decomposition $\Psi$ of $I_F$ will be denoted by $\Psi=(I_F^0,I_F^1,I_F^2,I_F^3)$. If only one of the four subsets, say $I_F^i$, is non-empty, we will denote $\Psi$ by the symbol $I_F^i$ itself.
\end{notation}

Fix now a irreducible $G_L$-representation $V_{\lambda}$ of highest weight $\lambda=\lambda((k_{1,\sigma}, k_{2,\sigma})_{\sigma \in I_F}, c)$ (as defined in \ref{irred}) and, for $m=0,1$, apply Theorem \ref{KostThm} to identify the cohomology spaces $H^q(W_{m,L}, V_{\lambda})$ as $(Q_m/W_m)_L$-representations: employing the notation fixed in \eqref{admpart}, we get isomorphisms

\begin{equation} \label{isocohunip}
H^q(W_{m,L}, V_{\lambda}) \simeq \bigoplus\limits_{\Psi \in \mathcal{P}_q} V^{m,q}_{\Psi},
\end{equation}
where each $V^{m,q}_{\Psi}$ is an irreducible $(Q_m/W_m)_L$-representation. With the notations of Lemma \ref{lemma:partweyl}.\ref{itm:partweyl}, the explicit computation of $w.(\lambda + \rho)-\rho$ for $w \in \Upsilon_m^\prime$ (as in \cite[Sec. 4.3]{Lem15}) gives us the highest weight of such irreducible representations, as stated in the following lemma: 

\begin{lemma}
\begin{enumerate} [wide, labelwidth=!, labelindent=0pt, label=(\arabic*)]
\item \label{charakost1} Adopting Notation \ref{notadmpart}, the highest weight of the irreducible $(Q_0/W_0)_L$-representation $V^{0,q}_{\Psi}$ in \eqref{isocohunip} is given by the restriction (along the inclusion $(Q_0/W_0)_L \subset Q_{0,L} \subset G_L)$ of the character 

\begin{equation} \label{restrcarac0}
\lambda((\eta_{1,\sigma}, \eta_{2,\sigma})_{\sigma \in I_F}, c),
\end{equation}
where 

\[
\eta_{1,\sigma} = \left\{ \begin{array}{cc}
k_{1,\sigma} \  & \mbox{if} \ \sigma \in I_F^0 \\ k_{1,\sigma} \  & \mbox{if} \ \sigma \in I_F^1 \\
k_{2,\sigma}-1 \ & \mbox{if} \ \sigma \in I_F^2 \\
-k_{2,\sigma}-3 \ & \mbox{if} \ \sigma \in I_F^3
\end{array} 
\right. ,\
\eta_{2,\sigma} = \left\{ \begin{array}{cc}
k_{2,\sigma} \  & \mbox{if} \ \sigma \in I_F^0 \\
-k_{2,\sigma}-2 \ & \mbox{if} \ \sigma \in I_F^1 \\
-k_{1,\sigma}-3 \ & \mbox{if} \ \sigma \in I_F^2 \\
-k_{1,\sigma}-3 \ & \mbox{if} \ \sigma \in I_F^3
\end{array} 
\right.
\]

\item \label{charakost2} Adopting Notation \ref{notadmpart}, the highest weight of the irreducible $(Q_1/W_1)_L$-representation $V^{1,q}_{\Psi}$ is the restriction (along the inclusion $(Q_1/W_1)_L \subset Q_{1,L} \subset G_L)$ of the character 

\begin{equation} \label{carackling}
\lambda((\epsilon_{1,\sigma}, \epsilon_{2,\sigma})_{\sigma \in I_F}, c)
\end{equation}
where 

\[
\epsilon_{1,\sigma} = \left\{ \begin{array}{cc}
k_{1,\sigma} \  & \mbox{if} \ \sigma \in I_F^0 \\
k_{2,\sigma} -1 \ & \mbox{if} \ \sigma \in I_F^1 \\
-k_{2,\sigma}-3 \ & \mbox{if} \ \sigma \in I_F^2 \\
-k_{1,\sigma}-4 \ & \mbox{if} \ \sigma \in I_F^3
\end{array} 
\right. ,\
\epsilon_{2,\sigma} = \left\{ \begin{array}{cc}
k_{2,\sigma} \  & \mbox{if} \ \sigma \in I_F^0 \\
k_{1,\sigma}+1 \ & \mbox{if} \ \sigma \in I_F^1 \\
k_{1,\sigma}+1 \ & \mbox{if} \ \sigma \in I_F^2 \\
k_{2,\sigma} \ & \mbox{if} \ \sigma \in I_F^3
\end{array} 
\right.
\]
\end{enumerate}
\label{charkost}
\end{lemma}

\subsubsection{\textbf{Cohomology of groups of units}} \label{cohunits} We finish this section by recalling, for the convenience of the reader, some standard arguments that will be useful in the analysis of the cohomology of arithmetic groups appearing in Theorem \ref{thm:PinkThm}.

\begin{lemma} \label{nultriv}
Let $\Gamma$ be a free abelian group of finite rank $r$, acting on a finite-dimensional vector space $V$ over a field $L$ by $L$-linear automorphisms. Suppose that $\Gamma$ acts through a character $\lambda$. Then, there exists an integer $s$ such that the cohomology space $H^s(\Gamma, V)$ is non-trivial if and only if the action of $\Gamma$ on $V$ is trivial. 

In this case, $H^s(\Gamma, V)$ is non-trivial if and only if $0 \leq s \leq r$, and for such integers $s$ we have (non-canonically)
\begin{equation*}
H^s(\Gamma, V) \simeq V^{\binom{r}{s}}.
\end{equation*} 
\end{lemma}
\begin{proof}
We proceed by induction on the rank $r$, denoting by $V^{\Gamma}$ the space of invariants of the $\Gamma$-action on $V$ and by $V_{\Gamma}$ the space of coinvariants. 
\begin{itemize}[wide, labelwidth=!, labelindent=0pt]
\item If $r=1$, we have $\Gamma \simeq \BZ$ and it is then well-known that 
\[
  H^s(\Gamma, V) =
  \begin{cases}
  V^{\Gamma} & \mbox{if } s=0 \\
  V_{\Gamma} & \mbox{if } s=1 \\
  \{ 0 \} & \mbox{otherwise}
  \end{cases}
\]  
Now, choose a generator $\gamma$ of $\Gamma$. Then, the space $V^{\Gamma}$ is non-trivial if and only if there exists a non-zero element $v \in V$ such that $\lambda(n \gamma) \cdot v = v$ for every $n \in \BZ$, which is equivalent to asking that $\lambda(\gamma)=1$, i.e. that $\Gamma$ act through the trivial character. Analogously, the space $V_{\Gamma}$ is non-trivial if and only if $\Gamma$ acts trivially. Thus, $H^s(\Gamma, V) \simeq V$ if $s=0,1$, and is trivial otherwise.
\item Suppose the assertion to be true for free abelian groups of rank $r$. If $\Gamma \simeq \BZ^{r+1}$, then choose a basis of $\Gamma$ as a $\BZ$-module and use it to define an exact sequence 
\begin{equation} \label{induction}
0 \rightarrow \BZ^r \rightarrow \Gamma \rightarrow \BZ \rightarrow 0
\end{equation}
By the case $r=1$, the Lyndon-Hochschild-Serre spectral sequence associated to this exact sequence, i.e.
\begin{equation}
E_2=H^m(\BZ, H^n(\BZ^r, V)) \Rightarrow H^{m+n}(\Gamma, V) 
\end{equation}
has only two non-trivial columns (for $n=0,1$). Hence, for each $s \geq 1$, we have an exact sequence
\begin{equation*}
0 \rightarrow H^0(\BZ, H^s(\BZ^r,V)) \rightarrow H^s(\Gamma, V) \rightarrow H^1(\BZ, H^{s-1}(\BZ^r,V)) \rightarrow 0
\end{equation*}
and moreover
\begin{equation*}
H^0(\Gamma,V) \simeq H^0(\BZ, H^0(\BZ^r, V)).
\end{equation*}
By the induction hypothesis, $H^s(\Gamma, V)$ can be non-trivial only if $0 \leq s \leq r+1$. Moreover, if the action of $\Gamma$ is non-trivial, the subgroup isomorphic to $\BZ^r$ appearing in \eqref{induction} can be chosen as acting non-trivially on $V$, and in this case, by induction, $H^s(\Gamma, V)$ is trivial for every $s$. If, on the contrary, $\Gamma$ acts trivially, then the induced action of $\BZ$ on the spaces $H^n(\BZ^r, V)$ is again trivial, so that, by the case $r=1$, we have
\begin{equation*}
H^0(\Gamma, V) \simeq H^0(\BZ, H^0(\BZ^r, V)) \simeq H^0(\BZ^r, V) \simeq V
\end{equation*}
and for every $ s \in \{ 1, \dots, r+1 \}$,
\begin{align*}
& H^s(\Gamma, V) \simeq \mbox{ (non-canonically) } H^0(\BZ, H^s(\BZ^r, V)) \oplus H^1(\BZ, H^{s-1}(\BZ^r, V)) \simeq \\
& \mbox{(by the case } r=1) \\
& \simeq H^s(\BZ^r, V) \oplus H^{s-1} (\BZ^r, V) \simeq \\
& \mbox{(by the induction hypothesis, and setting } V^{\binom{r}{r+1}}= \{ 0 \} \mbox{ by convention)} \\
& \simeq V^{\binom{r}{s}} \oplus V^{\binom{r}{s-1}} \simeq  V^{\binom{r+1}{s}}.
\end{align*}
\end{itemize}
\end{proof}

The preceding lemma leads to the problem of determining the triviality of the action of some free abelian groups, which in our case will arise as subgroups of units of totally real fields. This is the object of the following lemma:

\begin{lemma} \label{trivact}
Let $\mathcal{O}_F$ be the ring of integers of $F$ and fix $(0, \dots, 0) \neq (n_1, \dots, n_d ) \in \mathbb{Z}^d$. Then, 
\begin{equation*}
\prod \limits_{i=1, \dots, d} \vert \sigma_i(t) \vert ^{n_i} = 1 \ \mbox{for every} \ t \in \mathcal{O}_F^{\times}
\end{equation*} 
if and only if $(n_1, \dots, n_d) \in \mathbb{Z} \cdot (1, \dots, 1)$.
\end{lemma}

\begin{proof}
Choose a basis $\{ \gamma_1, \dots, \gamma_{d-1} \}$ of $\mathcal{O}_F^{\times}$ as $\mathbb{Z}$-module. 
Write $t= \prod\limits_{0, \dots, d-1} \gamma_i^{a_i }$, choose a $d$-tuple of integers $(n_1, \dots, n_d )\neq (0, \dots, 0)$, and define 
\[
\Lambda:= \left(
\begin{array} {ccc} 
\mbox{log}\vert \sigma_1(\gamma_1) \vert & \dots & \mbox{log}\vert \sigma_d(\gamma_1) \vert  \\
\vdots & \dots & \vdots \\
\mbox{log}\vert \sigma_1(\gamma_{d-1}) \vert & \dots & \mbox{log}\vert \sigma_d(\gamma_{d-1}) \vert
\end{array} \right)
\]

Then, we have 
\[ \prod \limits_{i=1, \dots, d} \vert \sigma_i(t) \vert ^{n_i} = 1 \ \forall \ t \iff \sum\limits_{j=1, \dots, d-1} \left( a_j \sum\limits_{i=1, \dots, d} n_i \mbox{log}\vert \sigma_i(\gamma_j) \vert \right) =0 \ \forall \ (a_1, \dots, a_{d-1}) \neq (0, \dots, 0) \in \mathbb{Z}^{d-1}  
\]

\begin{equation*}
\iff \langle \left( \begin{array} {c}
a_1 \\
\vdots \\
a_{d-1}
\end{array} \right),
\Lambda \cdot \left( \begin{array} {c}
n_1 \\
\vdots \\
n_d
\end{array} \right) \rangle = 0,
\forall \ (a_1, \dots, a_{d-1}) \neq (0, \dots, 0) \in \mathbb{Z}^{d-1}
\end{equation*}
(where $\langle \cdot, \cdot \rangle$ is the standard scalar product in $\mathbb{R}^{d-1}$) $\iff (n_1, \dots, n_d ) \in \mbox{ker} \ \Lambda$. But by Dirichlet's unit theorem, $\mbox{ker} \ \Lambda=\mathbb{R} \cdot (1, \dots, 1)$.
\end{proof}

\begin{remark} 
\begin{enumerate} [wide, labelwidth=!, labelindent=0pt, label=(\arabic*)]
\item \label{itm:genarithm} By choosing adapted bases, Lemma \ref{trivact} generalises immediately to the case where $\mathcal{O}^{\times}_F$ is replaced by a finite-index subgroup of $\mathcal{O}^{\times}_F$, and furthermore to the case where it is replaced by an arithmetic subgroup of $F^{\times}$.
\item \label{itm:net} Consider the \emph{norm} morphism $N:\mathcal{O}^{\times}_F \rightarrow \{ \pm 1 \}$. As the image of a neat subgroup by a morphism is again neat, the elements of a \emph{neat} subgroup of $\mathcal{O}^{\times}_F$ are of norm 1. Thus, if the neat subgroup $\Gamma_{0,Z}$ is a finite-index subgroup of $\mathcal{O}^{\times}_F$, we have that $\prod \limits_{i=1, \dots, d} \vert \sigma_i(t) \vert ^{n_i} = 1 \ \forall \ t \in \Gamma_{0,Z} \iff \prod \limits_{i=1, \dots, d} \sigma_i(t) ^{n_i} = 1 \ \forall \ t \in \Gamma_{0,Z}$, and, by \ref{itm:genarithm}, Lemma \ref{trivact} tells us that the action of $\Gamma_{0,Z}$ on a vector space by multiplication by $\prod \limits_{i=1, \dots, d} \sigma_i(t) ^{n_i}$ is trivial if and only if $(n_1, \dots, n_d) \in \mathbb{Z} \cdot (1, \dots, 1)$. 
This equivalence continues to hold in the general case where $\Gamma_{0,Z}$ is a (neat) arithmetic subgroup of $F^{\times}$, again via \ref{itm:genarithm}. 
\end{enumerate}
\label{rmk:gennet}
\end{remark}

\section{The degeneration of the canonical construction at the boundary} \label{deg}

\noindent In this section we prove our main result (Thm. \ref{thm:mainthm}), i.e. the description of the interval of \emph{weight avoidance} of the motive $i^* j_* ^{\lambda}\mathcal{V} \in \DBcM(\partial S_K^*)_L$ in terms of the \emph{corank} of $\lambda$. 

\subsection{Statement of the main result} \label{mainres}

Recall the notations used and introduced in \ref{motcan}. In order to state our central theorem, we need some more notions about $\lambda$, especially the notion of \emph{corank} (see the introduction for a motivation in the context of automorphic forms):

\begin{definition}{(cfr. \cite[Def. 1.9]{BR16})} \label{corang}
Let $\lambda=\lambda((k_{1,\sigma}, k_{2,\sigma})_{\sigma \in I_F}, c))$ (cfr. \ref{irred}) be a weight of $G_L$. 
\begin{enumerate}[wide, labelwidth=!, labelindent=0pt, label=(\arabic*)]
\item $k_1:=(k_{1,\sigma})_{\sigma \in I_F}$ or $k_2:=(k_{2,\sigma})_{\sigma \in I_F}$ is called \emph{parallel} if $k_{i,\sigma}$ is constant on $I_F$, equal to a positive integer $\kappa$ (and we write $k_i=\underline{\kappa}$). For each $\kappa \in \mathbb{Z}$, the weight $\lambda$ is called $\kappa$-\emph{Kostant parallel} if there exists a decomposition $I_F = I_F^0 \bigsqcup I_F^1$ such that 
\[
\left\{ \begin{array}{cc} 
k_{1,\sigma}=\kappa & \forall \sigma \in I_F^0 \\
k_{2,\sigma}=\kappa+1 & \forall \sigma \in I_F^1
\end{array}
\right.
\]
The weight $\lambda$ is called \emph{Kostant parallel} if there exists a $\kappa$ such that $\lambda$ is $\kappa$-Kostant parallel. 
\item We define the \emph{corank} $\mbox{cor}(\lambda)$ of $\lambda$ by
\[ \mbox{cor}(\lambda)=
\begin{cases}
0 & \mbox{ if } k_2 \mbox{ is not parallel } \\
1 & \mbox{ if } k_2 \mbox{ is parallel and } k_1 \neq k_2 \\
2 & \mbox{ if } k_2 \mbox{ is parallel and } k_1 = k_2 \\
\end{cases}
\] 
\item $\lambda$ is \emph{completely irregular} if $(k_{1,\sigma}, k_{2,\sigma})$ is irregular for every $\sigma \in I_F$.
\end{enumerate}
\end{definition}

Assume $\lambda$ to be dominant. We make some observations that may help enlightening the above definitions and their mutual relationships:

\begin{itemize}[wide, labelwidth=!, labelindent=0pt]
\item If $\mbox{cor}(\lambda) = 2$, then $\lambda$ is completely irregular. 
\item If $\mbox{cor}(\lambda) \geq 1$, then $\lambda$ is $\kappa$-Kostant parallel with respect to the decomposition $I_F=I_F^1$, with $k_2=\underline{\kappa+1}$; this decomposition and this $\kappa$ are then the only ones such that both $I_F^1 \neq \varnothing$ and $\lambda$ is Kostant-parallel with respect to them.
\item If $\mbox{cor}(\lambda) = 1$ and $\lambda$ is completely irregular, then necessarily $k_2 = \underline{0}$. 
\item If $\mbox{cor}(\lambda) = 0$, then there are at most a $\kappa$ and a decomposition $I_F=I_F^0 \bigsqcup I_F^1$ with respect to which $\lambda$ is $\kappa$-Kostant parallel. 
\end{itemize}

\begin{remark}
The terminology \emph{Kostant parallel} comes from the more specific terminology that will be introduced in Definitions \ref{triv0} and \ref{triv1} (see also Remark \ref{unipart}) and expresses the fact that some linear combinations of the coordinates of the character $\lambda$ are required to take a constant (\emph{parallel}) value over certain subsets of $I_F$. As the computations leading to those definitions will make clear, this terminology is motivated by the fact that such linear combinations and subsets arise from Lemma \ref{charkost}, which is an application of Kostant's theorem, Thm. \ref{KostThm}.
\end{remark}

We can now state our main result, in the language of weight structures introduced in \ref{weighstr}:

\begin{theorem} 

Let $V_{\lambda}$ the irreducible representation of $G_L$ of highest weight 
\begin{equation*}
\lambda=\lambda((k_{1,\sigma}, k_{2,\sigma})_{\sigma \in I_F}, c)),
\end{equation*}
$S_K$ the genus 2 Hilbert-Siegel variety of level $K$ corresponding to $(G,X)$ and $^{\lambda}\mathcal{V} \in CHM(S_K)$ the Chow motive over $S_K$ introduced in Definition \ref{motrep}. Let moreover $j:S_K \rightarrow S_K^*$, resp. $i:\partial S_K^* \rightarrow S_K^*$ denote the open, resp. closed immersion in the Baily-Borel compactification $S_K^*$ of $S_K$. Then: 
\begin{enumerate} 
\item If $\lambda$ is not Kostant parallel, then the \emph{boundary motive} $i^*j_*^{\lambda}\mathcal{V}$ is zero.
\item Suppose that $\mbox{\emph{cor}}(\lambda)=0$ and that $\lambda$ is $\kappa$-Kostant parallel. Denote $d_1:=\vert I_F^1 \vert$. Then $i^*j_*^{\lambda}\mathcal{V}$ avoids weights $-d_1-d\kappa+1, \dots, d_1+d\kappa$ and the weights $-d_1-d\kappa$, $d_1+d\kappa+1$ \emph{do appear} in $i^*j_*^{\lambda}\mathcal{V}$.
\item Suppose that $\mbox{\emph{cor}}(\lambda)=1$, with $k_2=\underline{\kappa_2}$, and that $k_1$ is not parallel. Then $i^*j_*^{\lambda}\mathcal{V}$ avoids weights $-d\kappa_2+1, \dots, d\kappa_2$ and the weights $-d\kappa_2$, $d\kappa_2+1$ \emph{do appear} in $i^*j_*^{\lambda}\mathcal{V}$. 
\item Suppose that $\mbox{\emph{cor}}(\lambda) \geq 1$, with $k_2=\underline{\kappa_2}$, and that $k_1=\underline{\kappa_1}$. Denote $\kappa:=\mbox{\emph{min}} \{ \kappa_1-\kappa_2, \kappa_2 \}$. Then $i^*j_*^{\lambda}\mathcal{V}$ avoids weights $-d\kappa+1, \dots, d\kappa$. The weights $-d\kappa_2$, $d\kappa_2+1$ \emph{do appear} in $i^*j_*^{\lambda}\mathcal{V}$, and if $\kappa_1$, $\kappa_2$ have the same parity\footnote{Cfr. Footnote \ref{footnoteparity} for this supplementary hypothesis.}, then the weights $-d(\kappa_1-\kappa_2)$, $d(\kappa_1-\kappa_2)+1$ \emph{do appear} in $i^*j_*^{\lambda}\mathcal{V}$.
\end{enumerate}
\label{thm:mainthm}
\end{theorem}

The proof of theorem \ref{thm:mainthm} will be completed at the end of paragraph \ref{weightevitd}, by invoking Theorem \ref{critint_a} and after having employed all the tools recalled in Subsection \ref{critexistence}. Admitting this theorem for the moment, we can prove its most important corollary for the applications to the \emph{intersection motive} (see Section \ref{consapp}), i.e. the following characterization of the absence of the weights 0 and 1:

\begin{corollary}\label{evit01}
The weights 0 and 1 appear in the boundary motive $i^*j_*^{\lambda}\mathcal{V}$ if and only if $\lambda$ is completely irregular of corank $\geq 1$.
\end{corollary} 
\begin{proof}
Suppose $\lambda$ to be $\kappa$-Kostant parallel with respect to $(I_F^0, I_F^1)$ (otherwise, by point $(1)$ of the above theorem, there is nothing to do). 

If $\mbox{cor}(\lambda) = 0$, then, by point $(2)$ of the above theorem, the weights $0$ and $1$ appear if and only if $d_1 =0=\kappa$. But $d_1=0$ means that $I_F^1=\varnothing$, i.e. $I_F^0=I_F$, and by definition of Kostant-parallelism this implies $k_1=\underline{\kappa}$. Now, necessarily $\kappa>0$, because otherwise $k_2=\underline{0}$ (remember that $k_{1,\sigma} \geq k_{2,\sigma}$ for every $\sigma \in I_F$) and $\mbox{cor}(\lambda)=2$, a contradiction. 

If $\mbox{cor}(\lambda) = 1$, with $k_2=\underline{\kappa_2}$, then, by point $(3)$ and $(4)$ of the above theorem, the weights 0 and 1 appear if and only if $\kappa_2=0$; observe in fact that, even if $k_1=\underline{\kappa_1}$, we have $\kappa_1-\kappa_2 > 0$ (otherwise $\mbox{cor}(\lambda) = 2$, a contradiction). But if $\kappa_2=0$, $\lambda$ is completely irregular.

If $\mbox{cor}(\lambda) = 2$, then $k_1=\underline{\kappa}=k_2$; this means that $\lambda$ is completely irregular, and implies that, in point $(4)$ of the above theorem, the parity condition is trivially satisfied and that $\kappa_1-\kappa_2=0$, so that the weights 0 and 1 appear. 

To conclude, we only have to observe that if $\mbox{cor}(\lambda)\geq 1$ and $\lambda$ is completely irregular, either $k_2=\underline{0}$ or $k_1=\underline{\kappa}=k_2$ (cfr. the observations after Def. \ref{corang}).
\end{proof}
The rest of this section is devoted to the proof of Thm. \ref{thm:mainthm}, following the outline given in the introduction. 
\subsection{The degeneration along the Siegel strata} \label{degsieg}

With notation as in the statement of Thm. \ref{thm:mainthm}, fix a irreducible $G_L$-representation $V_{\lambda}$ of highest weight $\lambda=\lambda((k_{1,\sigma}, k_{2,\sigma})_{\sigma \in I_F}, c)$: we want to employ Theorem \ref{thm:PinkThm} to study the degeneration of $\mu_{\ell}^K(V_{\lambda})$ along the Siegel strata, whose underlying Shimura datum is $(G_0, \mathfrak{H}_0)$, where $G_0 \simeq \mathbb{G}_m$, as explained in \ref{strates}.
 
\subsubsection{\textbf{Weights in the cohomology of the unipotent radical.}} \label{sec-cohunip0}

We start by identifying the possible weights appearing in the degeneration along the Siegel strata, i.e. in the $(Q_0/W_0)_L$-representations 
\begin{equation} \label{isocohunip0}
H^q(W_{0,L}, V_{\lambda}) \simeq \bigoplus\limits_{\Psi \in \mathcal{P}_q} V^{0,q}_{\Psi},
\end{equation}
for $q \in \{0, \dots, 3d \}$ (cfr. \eqref{isocohunip}). Recall from \eqref{isosiegel} that we have
\begin{equation*}
(Q_0/W_0)_L \simeq \mathbb{G}_{m,L} \times \prod\limits_{\sigma \in I_F} (\mbox{GL}_{2,L})_{\sigma}
\end{equation*}
Let us then compute the weight of the pure Hodge structure carried by each irreducible summand $V^{0,q}_{\Psi}$.

\begin{lemma}\label{poids0}
For every $q \in \{0, \dots, 3d \}$ and for every $q$-admissible decomposition $\Psi$ as in Notation \ref{notadmpart}, the action of the $\mathbb{G}_{m, L}$-factor inside $(Q_0/W_0)_L$ induces on $V^{0,q}_{\Psi}$ a pure Hodge structure of weight 
\begin{equation*}
w(\lambda) - [\sum\limits_{\sigma \in I_F^0} (k_{1,\sigma}+k_{2,\sigma}) +\sum\limits_{\sigma \in I_F^1} (k_{1,\sigma}-k_{2,\sigma}-2) - \sum\limits_{\sigma \in I_F^2} (k_{1,\sigma}-k_{2,\sigma}+4) - \sum\limits_{\sigma \in I_F^3} (k_{1,\sigma}+k_{2,\sigma}+6) ].
\end{equation*}
\end{lemma}
\begin{proof}
By the discussion in \ref{levicomp}, the $L$-points of the $\mathbb{G}_{m, L}$-factor are identified with the subgroup 
\begin{equation*}
\{ ( \left( \begin{array} {cc} 
\alpha I_2 &  \\
 & I_2
\end{array} \right) )_{\sigma \in I_F}  \ \vert \alpha \in L^{\times} \}
\end{equation*}
of $Q_0/W_0(L)$. With the notation of Lemma \ref{charkost}.\ref{charakost1} for the highest weight of the representation $V^{0,q}_{\Psi}$, and recalling \eqref{actcar}, we see that $\mathbb{G}_{m, L}(L)$ acts on $V^{0,q}_{\Psi}$ via the character
\begin{equation}
\alpha \mapsto \alpha^{\frac{1}{2} \cdot [ c + \sum\limits_{\sigma \in I_F} (\eta_{1,\sigma}+\eta_{2,\sigma})]},
\end{equation}
By the convention fixed in \ref{convhodge} and the definition in \eqref{donnée0} of the Shimura datum $(G_0, \mathfrak{H}_0)$, the expression for $\eta_{1,\sigma}$ and $\eta_{2,\sigma}$ given in Lemma \ref{charkost}.\ref{charakost1} yields the formula in the statement. 
\end{proof}

Notice for later use that if $\mbox{V}$ is the standard 2-dimensional $L$-representation of $\mbox{GL}_{2,L}$, the above computations imply that the representation obtained by restriction to the factor $\prod\limits_{\sigma \in I_F} (\mbox{GL}_{2,L})_{\sigma}$ of $(Q_0/W_0)_L$ is isomorphic to 

\begin{equation} \label{restrGL2}
\begin{aligned}
( \bigotimes\limits_{\sigma \in I_F^0} \mbox{Sym}^{k_{1,\sigma}-k_{2,\sigma}}\mbox{V} \otimes \mbox{det}^{k_{2,\sigma}} ) \otimes ( \bigotimes\limits_{\sigma \in I_F^1} \mbox{Sym}^{k_{1,\sigma}+k_{2,\sigma}+2}\mbox{V} \otimes \mbox{det}^{-k_{2,\sigma}-2} ) \otimes \\
\otimes ( \bigotimes\limits_{\sigma \in I_F^2} \mbox{Sym}^{k_{1,\sigma}+k_{2,\sigma}+2}\mbox{V} \otimes \mbox{det}^{-k_{1,\sigma}-3} ) \otimes ( \bigotimes\limits_{\sigma \in I_F^3} \mbox{Sym}^{k_{1,\sigma}-k_{2,\sigma}}\mbox{V} \otimes \mbox{det}^{-k_{1,\sigma}-3} )
\end{aligned}
\end{equation}

\subsubsection{\textbf{Cohomology of the arithmetic subgroup.}} \label{arithm0}
Consider now the arithmetic group $\Gamma_0$ of Rmk. \ref{compl}: according to Theorem \ref{thm:PinkThm}, and remembering \eqref{isocohunip0}, we need to identify the cohomology spaces

\begin{equation} \label{arithmisocohunip0}
H^p(\Gamma_0, H^q(W_{0,L}, V_{\lambda})) \simeq \bigoplus_{\substack{\Psi \in \mathcal{P}_q}} H^p(\Gamma_0, V^{0,q}_{\Psi}) 
\end{equation}
and their weight-graded objects $\mbox{Gr}_k^{\mathbb{W}}H^p(\Gamma_0, H^q(W_{0,L}, V_{\lambda})) \simeq \bigoplus_{\substack{\Psi \in \mathcal{P}_q}} H^p(\Gamma_0, \mbox{Gr}_k^{\mathbb{W}} V^{0,q}_{\Psi})$ (cfr. Remark \ref{rmk:explpink}\ref{itm:weightfilt}). As the cohomological dimension of $W_{0,L}$ is $3d$, these spaces can be non-zero only for $q \in \{0, \dots, 3d \}$. We are now going to put further restrictions on the non-triviality of such spaces. 

\begin{construction} \label{extarithm}
$\Gamma_0$ is identified with a neat (hence, torsion-free) arithmetic subgroup of 
\begin{equation*}
\mbox{Res}_{F|\mathbb{Q}} \mbox{GL}_{2,F}(\mathbb{Q})=\mbox{GL}_2(F)
\end{equation*}
(Remark \ref{compl}). Let $\pi$ be the projection $\mbox{GL}_{2}(F) \twoheadrightarrow \mbox{GL}_{2}(F) / Z(\mbox{GL}_{2}(F)) $ and define $\Gamma_{0,Z}:= \Gamma_0 \cap Z(\mbox{GL}_{2}(F))$ and $\Gamma_0^{\prime}:=\pi(\Gamma_0)$ (non trivial, torsion-free arithmetic subgroups of $Z(\mbox{GL}_{2}(F)) \simeq F^{\times}$, resp. $\mbox{PGL}_{2}(F)$). Then, $\Gamma_0$ can be written as an extension

\begin{equation}
1 \rightarrow \Gamma_{0,Z} \rightarrow \Gamma_0 \xrightarrow{\pi} \Gamma_0^{\prime} \rightarrow 1,
\end{equation}
and applying the Lyndon-Hochschild-Serre spectral sequence to this extension
\begin{equation}
E_2=H^r(\Gamma_0^{\prime}, H^s(\Gamma_{0,Z}, V^{0,q}_{\Psi})) \Rightarrow H^{r+s}(\Gamma_0, V^{0,q}_{\Psi}) 
\end{equation}
we see that every subspace $H^{p}(\Gamma_0, V^{0,q}_{\Psi})$ is (non-canonically) isomorphic to a direct sum

\begin{equation}
\bigoplus_{\substack{r+s=p}} U^{r,s}
\end{equation}
where every $U^{r,s}$ is a subquotient of $H^r(\Gamma_0^{\prime}, H^s(\Gamma_{0,Z}, V^{0,q}_{\Psi}))$. Thus, if $H^s(\Gamma_{0,Z}, V^{0,q}_{\Psi})$ is zero for every $s$, then $H^{p}(\Gamma_0, V^{0,q}_{\Psi})$ is.

\end{construction}

Lemma \ref{nultriv} gives necessary conditions for the non-triviality of the cohomology of a free abelian group acting on a vector space. The following lemma tells us when these conditions are verified in a specific case: 

\begin{lemma} \label{nececond}
Let $\Gamma_{0,Z}$ be the group defined in Construction \ref{extarithm}. Then, its action on on $V^{0,q}_{\Psi}$ is trivial if and only if there exists an integer $\kappa$ such that

\begin{equation}\label{Kostpar0}
\left\{ \begin{array}{cc} 
k_{1,\sigma}+k_{2,\sigma}=\kappa & \forall \sigma \in I_F^0 \\
k_{1,\sigma}-k_{2,\sigma}-2=\kappa & \forall \sigma \in I_F^1
\\
-(k_{1,\sigma}- k_{2,\sigma}+4)=\kappa & \forall \sigma \in I_F^2 \\
-(k_{1,\sigma}+ k_{2,\sigma}+6)=\kappa & \forall \sigma \in I_F^2
\end{array}
\right.
\end{equation}
(remembering Notation \ref{notadmpart}).
\end{lemma}
\begin{proof}
By Dirichlet's unit theorem, we have that $\Res_{F \vert \BQ} \mathbb{G}_{m,F}(\mathbb{Z}) \simeq \mathcal{O}^{\times}_F \simeq \mathbb{Z}^{d-1} \times \mathbb{Z}/ 2 \mathbb{Z} $. On the other hand, the torsion-free group $\Gamma_{0,Z}$ is commensurable to $\mathcal{O}^{\times}_F$. The group $\Gamma_{0,Z}$ is then isomorphic to $\mathbb{Z}^{d-1}$. By choosing generators $\gamma_1, \dots, \gamma_{d-1}$, and remembering the discussion in \ref{levicomp}, it is then identified with the subgroup  
\begin{equation} \label{sgrpdiag}
\{ ( \left( \begin{array} {cccc} 
\sigma(t) & & & \\
 & \sigma(t) & & \\
 & & \sigma(t)^{-1} & \\
 & & & \sigma(t)^{-1}
\end{array} \right) )_{\sigma \in I_F}  \ \vert t=\gamma_1^{p_1} \dots \gamma_{d-1}^{p_{d-1}}, p_1, \dots, p_{d-1} \in \mathbb{Z} \} \hookrightarrow Q_0/W_0(L).
\end{equation}
Recalling the expression for the highest weight of the representation $V^{0,q}_{\Psi}$ given in Lemma \ref{charkost}.\ref{charakost1}, we see that an element $t=\gamma_1^{p_1} \dots \gamma_{d-1}^{p_{d-1}} \in \Gamma_{0,Z}$ acts on $V^{0,q}_{\Psi}$ via multiplication by 

\begin{equation*}
\prod\limits_{\sigma \in I_F} \sigma(t)^{\eta_{1,\sigma}+\eta_{2,\sigma}^{\prime}}=\prod\limits_{\sigma \in I_F^0} \sigma(t)^{k_{1,\sigma}+k_{2,\sigma}} \cdot \prod\limits_{\sigma \in I_F^1}\sigma(t)^{k_{1,\sigma}-k_{2,\sigma}-2}  \cdot \prod\limits_{\sigma \in I_F^2} \sigma(t)^{-(k_{1,\sigma}- k_{2,\sigma}+4)} \cdot \prod\limits_{\sigma \in I_F^3} \sigma(t)^{-(k_{1,\sigma}+ k_{2,\sigma}+6)}.
\end{equation*}
The condition in the statement then follows by applying Lemma \ref{trivact}, via Remark \ref{rmk:gennet}.
\end{proof}

\begin{definition} \label{triv0}
If $\lambda$ satisfies the above condition with respect to a $q$-admissible decomposition $\Psi$ and to $\kappa \in \mathbb{Z}$, we say that $\lambda$ is $(\kappa,0)$\emph{-Kostant parallel} with respect to $\Psi$.
\end{definition}

\begin{definition} \label{def:parttype1}
A $q$-admissible decomposition $\Psi$ is said to be $(\lambda,0)$\emph{-admissible} if there exists $\kappa \in \mathbb{Z}$ such that $\lambda$ is $(\kappa,0)$-Kostant parallel with respect to $\Psi$. The set of  $q$-admissible decompositions which are moreover $(\lambda,0)$-admissible will be denoted by $\mathcal{P}_q^{(\lambda,0)}$.
\end{definition}

With these definitions in hand, we can prove:

\begin{lemma} \label{cohunip0}
For every $s \notin \{0, \dots, d-1 \}$, the cohomology space $H^s(\Gamma_{0,Z}, V^{0,q}_{\Psi})$ is trivial. For every $s \in \{0, \dots, d-1 \}$, it is non-trivial if and only if $\lambda$ is $(\kappa,0)$-Kostant parallel with respect to $\Psi$ and one of the following two conditions holds:
\begin{enumerate} [wide, labelwidth=!, labelindent=0pt]
\item \label{cas1} 
$I_F=I_F^0 \sqcup I_F^1$. In this case, $q \in \{0, \dots, d \}$ and $\mbox{\emph{Gr}}_{w(\lambda)-d\kappa}^{\mathbb{W}} V^{0,q}_{\Psi} \neq \{ 0 \}$;
\item $I_F=I_F^2 \sqcup I_F^3$. In this case, $q \in \{2d, \dots, 3d \}$ and $\mbox{\emph{Gr}}_{w(\lambda)-d\kappa}^{\mathbb{W}} V^{0,q}_{\Psi} \neq \{ 0 \}$.
\end{enumerate}
\end{lemma}
\begin{proof}
We make first a preliminary observation: suppose that $\lambda$ satisfies \eqref{Kostpar0} for a certain $q$-admissible decomposition $\Psi$. We see then that if $I_F^0 \sqcup I_F^1$ is non empty, then $\kappa \geq -2$, and that if $I_F^2 \sqcup I_F^3$ is non-empty, then $\kappa \leq -4$. This follows from the formulae in \eqref{Kostpar0} and from the fact that, since $\lambda$ is dominant, $k_{1,\sigma} \geq k_{2,\sigma} \geq 0$ for every $\sigma \in I_F$.

Now, Lemma \ref{nultriv} says that for each $s \in \{0, \dots, d-1 \}$, $H^s(\Gamma_{0,Z}, V^{0,q}_{\Psi})$ is non-trivial if and only if the action of $\Gamma_{0,Z}$ is trivial, i.e. if and only if condition \eqref{Kostpar0} is satisfied. But in this case, the previous observation says that only one of the subsets $I_F^0 \sqcup I_F^1$, $I_F^2 \sqcup I_F^3$ can be non-empty. Then, in both situations, the assertion on $q$ comes from the definition of $q$-admissible decomposition (Eq. \eqref{admpart}), while the assertion on the weight appearing in $V^{0,q}_{\Psi}$ comes by comparing the computation of Lemma \ref{poids0} with the formulae in \eqref{Kostpar0}.
\end{proof}

\begin{remark}\label{rmk:coincpoids}
The above lemma, which is an essential step towards Theorem \ref{thm:mainthm}, implicitly makes use of the \q{coincidences} in the computations in Lemma \ref{poids0} and in \eqref{Kostpar0}, i.e. of the fact that the linear combinations of coordinates of characters that appear in the two cases are the same. This can be rephrased as follows. With $M_0$ as in Notation \ref{groupcompl}, let $\iota: \mathbb{G}_{m,L} \rightarrow Z(M_0)_L$ be the composition of the adjunction embedding $\mathbb{G}_{m,L} \hookrightarrow \mathbb{G}_{m,L}^d$ and of the isomorphism $\mathbb{G}_{m,L}^d \simeq Z(M_0)_L$  deduced from the isomorphism $(Q_0/W_0)_L \simeq G_{0,L} \times M_{0,L} \simeq \mathbb{G}_{m,L} \times \prod\limits_{\sigma \in I_F} (\mbox{GL}_{2,L})_{\sigma}$. Let moreover $w:\mathbb{G}_{m,\mathbb{R}} \rightarrow \mathbb{S}$, resp. $k:\mathbb{S} \rightarrow G_{0,\mathbb{R}}$ be the cocharacter defined in \ref{convhodge}, resp. the morphism defining the Shimura datum corresponding to $G_0$. Then, for every $\lambda$, Lemma \ref{poids0} and \eqref{Kostpar0} show that we have
\begin{equation} \label{coincpoids}
\restr{\lambda}{G_{0,\mathbb{R}}} \circ k \circ w = \restr{\lambda}{Z(M_0)_{\mathbb{R}}} \circ \iota_{\BR}.
\end{equation}
In other words, the \q{Hodge weight}, determined by the restriction of $\lambda$ to the center of the $G_0$-component of $Q_0/W_0$, equals the (\emph{a priori} different) character obtained by restriction to the center of the $M_0$-component.

This is indeed a general phenomenon, as we explain now. Denote by $A$ the maximal $\mathbb{Q}$-split torus in the center of $(Q_0/W_0)\cap G^{\mbox{\tiny{der}}}$, which is a subgroup of $Z(M_0)$ isomorphic to $\mathbb{G}_{m}$. If $\iota_A$ is the isomorphism $A \simeq \mathbb{G}_{m}$ obtained in the same way as $\iota$, then $\restr{\lambda}{Z(M_0)_{\mathbb{R}}} \circ \iota_{\BR} =  \restr{\lambda}{A_{\mathbb{R}}} \circ \iota_{A,\BR}$. Hence, we see that \eqref{coincpoids} is a consequence of \cite[Prop. 6.4]{LR91}: the proof in \emph{loc. cit.} is valid for general Shimura data and is based on the description of the action of $A$ (through $\lambda$) via \emph{local Hecke operators}.
 \end{remark}

Observe now that, by the considerations in Costruction \ref{extarithm}, the necessary conditions for non-triviality of the cohomology of $\Gamma_{0,Z}$ give necessary conditions for non-triviality of the cohomology of the bigger group $\Gamma_0$. Applying this, and employing Definition \ref{def:parttype1}, the isomorphism from Theorem \ref{thm:PinkThm}.\ref{itm:decobj} for a stratum $Z$ of $\partial S_K^*$ contributing to $Z_0$ now becomes
\begin{equation} \label{deccoh0}
\restr{R^n i^*_0 i^* j_* \mu_{\ell}^K (V_{\lambda})}{Z} \simeq
\bigoplus_{\substack{p+q=n}} \mu_{\ell}^{\pi_0(K_0)}(\bigoplus_{\substack{\Psi \in \mathcal{P}_q^{(\lambda,0)}}} H^p(\Gamma_0, V^{0,q}_{\Psi})).
\end{equation}

We are interested in the weight-graded objects
\begin{equation} \label{gradcoh0}
\mbox{Gr}_k^{\mathbb{W}} \restr{R^n i^*_0 i^* j_* \mu_{\ell}^K(V_{\lambda})}{Z}\simeq\bigoplus_{\substack{p+q=n}} \mu_{\ell}^{\pi_0(K_0)} ( \bigoplus_{\substack{\Psi \in \mathcal{P}_q^{(\lambda,0)}}} H^p(\Gamma_0, \mbox{Gr}_k^{\mathbb{W}} V^{0,q}_{\Psi}) ).
\end{equation}
We are going to find a \emph{second} set of necessary conditions for these objects to be non-trivial, through a \emph{dévissage} which is \q{orthogonal} to the one described in Remark \ref{extarithm}.
\begin{construction} \label{extarithmbis}
The groups $\Gamma_{0,ss}:= \Gamma_0 \cap \mbox{SL}_{2}(F)$, resp. $\mbox{det} \ \Gamma_0$ are non-trivial subgroups of $\mbox{SL}_{2}(F)$, resp. $F^{\times}$, which are again arithmetic and torsion-free. In particular, $\mbox{det} \ \Gamma_0 \simeq \mathbb{Z}^{d-1} $ (as in the proof of Lemma \ref{nececond}). Moreover, $\Gamma_0$ can be written as an extension
\begin{equation*}
1 \rightarrow \Gamma_{0,ss} \rightarrow \Gamma_0 \xrightarrow{\mbox{\tiny{det}}} \mbox{det} \ \Gamma_0 \rightarrow 1,
\end{equation*}
so that the Lyndon-Hochschild-Serre spectral sequence applied to this extension
\begin{equation*}
E_2=H^r(\mbox{det} \ \Gamma_0, H^s(\Gamma_{0,ss}, V^{0,q}_{\Psi})) \Rightarrow H^{r+s}(\Gamma_0, V^{0,q}_{\Psi})
\end{equation*}
tells us that each space $H^{p}(\Gamma_0, V^{0,q}_{\Psi})$ is (non-canonically) isomorphic to a direct sum
\begin{equation*}
\bigoplus_{\substack{r+s=p}} N^{r,s}
\end{equation*}
where each $N^{r,s}$ is a subquotient of $H^r(\mbox{det} \ \Gamma_0, H^s(\Gamma_{0,ss}, V^{0,q}_{\Psi}))$. If $H^r(\mbox{det}
\ \Gamma_{0}, H^s(\Gamma_{0,ss}, V^{0,q}_{\Psi}))$ is zero for every $r$ or $H^s(\Gamma_{0,ss}, V^{0,q}_{\Psi})$ is zero for every $s$, then $H^{p}(\Gamma_0, V^{0,q}_{\Psi})$ is.
\end{construction}

For every integer $q \in \{0, \dots, 3d \}$, we know by \eqref{restrGL2} that $\mbox{det} \ \Gamma_0$ acts on $V^{0,q}_{\Psi}$, and \emph{a fortiori} on its subspace $H^0(\Gamma_{0,ss}, V^{0,q}_{\Psi})$, via multiplication by the character $\chi$ defined by
\begin{equation}\label{caracdet}
t \mapsto \prod\limits_{\sigma \in I_F^0} \sigma(t)^{k_{2,\sigma}} \cdot \prod\limits_{\sigma \in I_F^1} \sigma(t)^{-k_{2,\sigma}-2} \cdot \prod\limits_{\sigma \in I_F^2} \sigma(t)^{-k_{1,\sigma}-3} \cdot \prod\limits_{\sigma \in I_F^3} \sigma(t)^{-k_{1,\sigma}-3} 
\end{equation}
and the following lemma will allow us to identify the cohomology spaces corresponding to this action:

\begin{lemma} \label{trivbis} Let $\lambda=\lambda(k_1,k_2,c)$ be $(\kappa,0)$-Kostant parallel with respect to  a $q$-admissible decomposition $\Psi$ of $I_F$ (notation as in \ref{notadmpart}). Fix $s \in \{0, \dots, 3d\}$ and suppose that $H^s(\Gamma_{0,ss}, V^{0,q}_{\Psi})$ is non-zero. For every $r \in \{0, \dots, d-1 \}$,
\begin{equation*}
H^r(\mbox{\emph{det}}\ \Gamma_{0}, H^s(\Gamma_{0,ss}, V^{0,q}_{\Psi})) \neq \{ 0 \} \iff H^0(\mbox{\emph{det}}\ \Gamma_{0}, H^s(\Gamma_{0,ss}, V^{0,q}_{\Psi})) \neq \{ 0 \}
\end{equation*}
$\iff$ one of the following conditions is satisfied:
\begin{enumerate} [wide, labelwidth=!, labelindent=0pt]
\item $I_F=I_F^0$ and $k_2$ is parallel. In this case, $q=0$;
\item $I_F=I_F^1$ and $k_2$ is parallel. In this case, $q=d$;
\item $I_F=I_F^2 \sqcup I_F^3$ and such that $k_1$ is parallel. In this case, $q \in \{2d, \dots, 3d \}$.
\end{enumerate} 
\end{lemma}
\begin{proof}
The point is to reduce oneself to the case where the action of $\mbox{det} \ \Gamma_0$ on the spaces $H^s(\Gamma_{0,ss}, V^{0,q}_{\Psi})$, for $s >0$, remains semisimple. Now, if $[\Phi] \in H^s(\Gamma_{0,ss}, V^{0,q}_{\Psi})$ is the class of a $s$-cocycle 
\begin{equation*}
\Phi \in \mbox{Hom}_{L[\Gamma_{0,ss}]}(L[\Gamma_{0,ss}]^{s+1}, V^{0,q}_{\Psi}),
\end{equation*}
then, for $t \in \mbox{det} \ \Gamma_0$, the element $t.[\Phi] \in H^s(\Gamma_{0,ss}, V^{0,q}_{\Psi})$ equals the class of the morphism 
\begin{equation*}
t.\Phi \in \mbox{Hom}_{L[\Gamma_{0,ss}]}(L[\Gamma_{0,ss}]^{s+1}, V^{0,q}_{\Psi})
\end{equation*}
that to every $(t_0, \dots, t_s)$ associates $\chi(t)\Phi(\tilde{t}^{-1}(t_0, \dots, t_s)\tilde{t})$ (where $\tilde{t}$ is any lifting of $t$ in $\Gamma_0$, and $\chi$ is as in \eqref{caracdet}).

Consider now the subgroup of $\Gamma_0$ defined in \eqref{sgrpdiag}, which is free abelian, generated by
$\{\gamma_1, \dots, \gamma_{d-1} \}$. The elements $\{(\gamma_1)^2, \dots, (\gamma_{d-1})^2 \}$ generate a free abelian subgroup $\tilde{\Gamma}$ of $\mbox{det} \ \Gamma_0$, of rank $d-1$, each of whose elements has a central lifting in $\Gamma_0$. Then, for every $s$, $\tilde{\Gamma}$ still acts via the character $\chi$ on $H^s(\Gamma_{0,ss}, V^{0,q}_{\Psi})$. We can now apply Lemma \ref{nultriv} and Remark \ref{rmk:gennet} to $\tilde{\Gamma}$ and conclude that if $H^s(\Gamma_{0,ss}, V^{0,q}_{\Psi})$ is non-zero, then $H^r(\tilde{\Gamma}, H^s(\Gamma_{0,ss}, V^{0,q}_{\Psi})) \simeq H^0(\tilde{\Gamma}, H^s(\Gamma_{0,ss}, V^{0,q}_{\Psi}))^{\binom{d-1}{r}} \neq \{ 0 \}$ if and only if (remembering the definition of $\chi$) there exists an integer $\theta$ such that
\[
\left\{ \begin{array}{cc} 
k_{2,\sigma}=\theta & \forall \sigma \in I_F^0 \\
-k_{2,\sigma}-2=\theta & \forall \sigma \in I_F^1
\\
-k_{1,\sigma}-3=\theta & \forall \sigma \in I_F^2 \\
-k_{1,\sigma}-3=\theta & \forall \sigma \in I_F^3
\end{array}
\right.
\]
Now recall that $k_{1,\sigma} \geq k_{2,\sigma} \geq 0$: the above condition is then equivalent to the one in the statement. Remember that precisely under this condition, the character $\chi$ is trivial.
 
In order to finish the proof, put $\mathcal{F}:=\mbox{det}\ \Gamma_{0} / \tilde{\Gamma}$: it is a finite group, that we can assume non trivial, of a certain order $f$ (otherwise, there is nothing else to do). Let $\{\phi_1 \dots, \phi_f \}$ be a system of representatives of $\mathcal{F}$ inside $\mbox{det}\ \Gamma_{0}$ and denote by $e$ the endomorphism of multiplication by $\frac{1}{f}\sum\limits_{i=1}^f \chi(\phi_i)$. By considering the Lyndon-Hochschild-Serre spectral sequence associated to this quotient and by applying \cite[Prop. 6.1.10]{Wei94}, we see that 
\begin{equation*}
H^r(\mbox{det} \ \Gamma_{0}, H^s(\Gamma_{0,ss}, V^{0,q}_{\Psi})) \simeq H^0 \left( \mathcal{F},H^r(\tilde{\Gamma}, H^s(\Gamma_{0,ss}, V^{0,q}_{\Psi})) \right) \simeq e \cdot H^r(\tilde{\Gamma}, H^s(\Gamma_{0,ss}, V^{0,q}_{\Psi})).
\end{equation*} 
Now, if $e$ is not the zero endomorphism, then $e \cdot H^r(\tilde{\Gamma}, H^s(\Gamma_{0,ss}, V^{0,q}_{\Psi})) \simeq H^r(\tilde{\Gamma}, H^s(\Gamma_{0,ss}, V^{0,q}_{\Psi}))$, and the lemma is demonstrated. But in the case we are working in, $\chi$ is trivial, and $e$ is just the identity morphism.
\end{proof}

The \emph{third} and last set of necessary conditions for non-triviality of the cohomology of $\Gamma_0$ comes from general results on the cohomology of locally symmetric spaces. 

\begin{lemma} \label{restrcoh0}
The following statements hold.
\begin{enumerate} [wide, labelwidth=!, labelindent=0pt, label=(\arabic*)]
\item \label{restr1} The cohomology space $H^p(\Gamma_0, V^{0,q}_{\Psi})$ is trivial for every $p<0$ and all $p > 3d-2$.
\item \label{restr2} If the irreducible representation $V^{0,q}_{\Psi}$ is non-trivial as a $\mbox{\emph{SL}}^d_{2,L}$-representation, then $H^p(\Gamma_0, V^{0,q}_{\Psi})=\{ 0 \}$ for every $0 \leq p < d$.
\end{enumerate}
\end{lemma}
\begin{proof}
\begin{enumerate} [wide, labelwidth=!, labelindent=0pt]
\item Recall from Notation \ref{groupcompl} the group $M_0\simeq
\mbox{Res}_{F|\mathbb{Q}} \mbox{GL}_{2,F}$, and denote by $K_{\infty}$ a maximal compact subgroup of $M_0(\mathbb{R})$ (isomorphic to $\prod\limits_{\sigma \in I_F} \mbox{O}_2(\mathbb{R})$), by $A_{M_0}$ the group $S(\mathbb{R})^0$ (for $S$ the maximal $\mathbb{Q}$-split torus inside $Z(M_0)$) and by $\mathcal{H}$ the complex upper half plane. Then, the \emph{symmetric space} associated to $M_0$ is defined by $D:=M_0(\mathbb{R})/K_{\infty}A_{M_0} \simeq \mathcal{H}^d \times \mathbb{R}^{d-1}$, and every $V^{0,q}_{\Psi}$ (as a $M_{0,L}$-representation) defines a local system $\mathbb{V}^{0,q}_{\Psi}$ on $X_{\Gamma_0}:=D/\Gamma_0$ such that $H^p(\Gamma_0, V^{0,q}_{\Psi}) \simeq H^p(X_{\Gamma_0},\mathbb{V}^{0,q}_{\Psi})$ for every $p$. Let now $R(\cdot)$, resp. $\mbox{r}_{\mathbb{Q}}(\cdot)$ denote the radical, resp. the $\mathbb{Q}$-rank of a $\mathbb{Q}$-algebraic group. The statement then follows from \cite[Thm. 11.4.4]{BS73}, taking into account that $\mbox{dim}(D)-\mbox{r}_{\mathbb{Q}}(M_0/R(M_0))=3d-2$. 

\item Recall the group $\Gamma_{0,ss}$ defined in Construction \ref{extarithmbis}, define the (complex analytic, connected) Hilbert-Blumenthal variety $X_{\Gamma_{0,ss}}$ as $\Gamma_{0,ss} \backslash \mathcal{H}^d$ and abusively also denote by $\mathbb{V}^{0,q}_{\Psi}$ the local system on $X_{\Gamma_{0,ss}}$ induced by the restriction of the representation $V^{0,q}_{\Psi}$ to $M_{0,L}^{\mbox{\tiny{der}}} \simeq (\mbox{Res}_{F|\mathbb{Q}} \mbox{SL}_{2,F})_L \simeq \mbox{SL}^d_{2,L}$. Then, for every $p$, we have $H^p(\Gamma_{0,ss}, V^{0,q}_{\Psi}) \simeq H^p(X_{\Gamma_{0,ss}},\mathbb{V}^{0,q}_{\Psi})$. The statement now follows from the fact that $H^p(X_{\Gamma_{0,ss}},\mathbb{V}^{0,q}_{\Psi})=\{ 0 \}$ for every $0 \leq p < d$ if $\mathbb{V}^{0,q}_{\Psi}$ is non-trivial (\cite[Thm. 1.1(i)]{M-SSYZ15}) and by employing the considerations at the end of Construction \ref{extarithmbis}. 
\end{enumerate}
\end{proof}

\subsubsection{\textbf{Computation of weights along the Siegel strata.}}
We can finally describe the weights appearing in the degeneration of the canonical construction along the Siegel strata, in the cohomological degrees which we will need in the sequel:

\begin{proposition}\label{poidscoh0}
Let $V_{\lambda}$ be the irreducible $L$-representation of $G_L$ of highest weight $\lambda=\lambda((k_{1,\sigma}, k_{2,\sigma})_{\sigma}, c)$ and $Z$ a stratum of $\partial S_K^*$ which contributes to $Z_0$. Adopt Notation \ref{notadmpart} and the notation of Definition \ref{corang}. 
\begin{enumerate} [wide, labelwidth=!, labelindent=0pt]
\item Let $n < 0$ or $n>6d-2$. Then the cohomology sheaf $\restr{R^n i^*_0 i^* j_* \mu_{\ell}^K(V_{\lambda})}{Z}$ is zero.
\item Let $0 \leq n < d$. Then the cohomology sheaf $\restr{R^n i^*_0 i^* j_* \mu_{\ell}^K(V_{\lambda})}{Z}$ can be non-zero only if $k_1 = \underline{\kappa_0}=k_{2}$. In this case, 
\begin{equation*}
\restr{R^n i^*_0 i^* j_* \mu_{\ell}^K(V_{\lambda})}{Z} \simeq \mu_{\ell}^{\pi_0(K_0)}(H^n(\Gamma_0, V^{0,0}_{I_F^0}))
\end{equation*}
is pure of weight $w(\lambda)-2d\kappa_0$. If $n=0$, then it is non-zero.
\item Let $n \in \{d, \dots, 2d-1 \}$. Then the cohomology sheaf $\restr{R^n i^*_0 i^* j_* \mu_{\ell}^K(V_{\lambda})}{Z}$ can be non-zero only if  $k_1=\underline{\kappa_1}$ and $k_2=\underline{\kappa_2}$. In this case, 
\begin{equation}
{R^n i^*_0 i^* j_* \mu_{\ell}^K(V_{\lambda})}{Z} \simeq \mu_{\ell}^{\pi_0(K_0)}(H^n(\Gamma_0, V^{0,0}_{I_F^0}))
\end{equation}
is pure of weight $w(\lambda)-d(\kappa_1+\kappa_2)$. If $\kappa_1 \neq \kappa_2$ and $n=d$, then it is non-zero. 
\item Let $n \in \{2d, \dots, 3d-1 \}$. Then the cohomology sheaf $\restr{R^n i^*_0 i^* j_* \mu_{\ell}^K(V_{\lambda})}{Z}$ can be non-zero only if $k_1=\underline{\kappa_1}$ and $k_2=\underline{\kappa_2}$. In this case, it is isomorphic to 
\begin{equation*}
\mu_{\ell}^{\pi_0(K_0)}(H^n(\Gamma_0, V^{0,0}_{I_F^0})) \oplus \mu_{\ell}^{\pi_0(K_0)}(H^{n-d}(\Gamma_0, V^{0,d}_{I_F^1})),
\end{equation*}
where the first factor is isomorphic to 
\begin{equation*}
\mbox{\emph{Gr}}_{w(\lambda)-d(\kappa_1+\kappa_2)}^{\mathbb{W}} \restr{R^n i^*_0 i^* j_* \mu_{\ell}^K(V_{\lambda})}{Z}
\end{equation*}
and the second one to 
\begin{equation*}
\mbox{\emph{Gr}}_{w(\lambda)+2d-d(\kappa_1-\kappa_2)}^{\mathbb{W}} \restr{R^n i^*_0 i^* j_* \mu_{\ell}^K(V_{\lambda})}{Z}
\end{equation*}
If $n=2d$, then the second factor is non-zero. 
\end{enumerate}
\end{proposition}

\begin{proof}
We begin by proving the necessary conditions for the non-vanishing of the cohomology sheaves. 
\begin{enumerate} [wide, labelwidth=!, labelindent=0pt]
\item Clear from Remark \ref{restrcoh0}.\ref{restr1} and from the fact that the representations $V^{0,q}_{\Psi}$ are trivial for $q > 3d$. 
\item The isomorphisms \eqref{deccoh0} and Lemma \ref{trivbis} (taking into account the considerations at the end of Construction \ref{extarithmbis}) imply that, in order to have non-zero cohomology objects in degree $0 \leq n <d$, $k_2$ has to be parallel and that $q$ can only take the value $0$. On the other hand, adding the Kostant-parallelism conditions imposed by Lemma \ref{cohunip0}, we obtain that $k_1$ has to be parallel, too. Moreover, by hypothesis, we are in the case $p \in \{0, \dots, d-1 \}$, but in this interval, by Lemma \ref{restrcoh0}.\ref{restr2}, $V^{0,0}_{I_F^0}$ can have non-trivial cohomology objects only if it is the trivial $\mbox{SL}^d_{2,L}$-representation. Now, looking at the description in \eqref{restrGL2}, we see that this is the case if and only if $k_1$ and $k_2$ are equal. 

\item The isomorphisms \eqref{deccoh0} and Lemma \ref{trivbis} (taking into account the considerations at the end of Construction \ref{extarithmbis}) imply that, in order to have non-zero cohomology objects in degree $d \leq n <2d$, $k_2$ has to be parallel and that $q$ can only take the values $0$ or $d$. Again, the Kostant-parallelism conditions imposed by Lemma \ref{cohunip0} imply that $k_1$ has to be parallel, too. Now if $q=d$, then $p \in \{0, \dots, d-1 \}$, and in this interval, by Lemma \ref{restrcoh0}.\ref{restr2}, $V^{0,d}_{I_F^1}$ can have non-trivial cohomology objects only if it is the trivial $\mbox{SL}^d_{2,L}$-representation; but the description in \eqref{restrGL2} shows that this is never the case. The only remaining possibility is then $q=0$ and $p \in \{d, \dots, 2d-1 \}$. 

\item Arguing as above, we see that $k_1$ and $k_2$ have to be parallel, and that $q$ can only take values in $\{ 0, d, 2d, \dots, 3d \}$. The cases $q=0$ and $q=d$ give the two summands in the statement.  If $2d \leq q \leq 3d-1$, then the fact that $p \in \{0, \dots, d-1 \}$ and Lemma \ref{restrcoh0}.\ref{restr2} imply that the spaces $H^p(\Gamma_0, V^{0,q}_{\Psi})$ can give non-trivial contributions to the cohomology objects if and only if $V^{0,q}_{\Psi}$ is the trivial $\mbox{SL}^d_{2,L}$-representation. However, this is never the case, by the description in \eqref{restrGL2} (remember that in this case, $\Psi$ is of the form $(\varnothing, \varnothing, I_F^2 \neq \varnothing, I_F^3) $). 
\end{enumerate}

Finally, in all cases, the statements about weight-graded objects follow from Remark \ref{cohunip0} and from the isomorphisms \eqref{gradcoh0}, while the non-triviality statements are consequences of the following proposition.
\end{proof}

\begin{proposition} \label{poidsapp}
Let $\lambda$, $V_{\lambda}$ and $Z$ be as in Proposition \ref{poidscoh0}. If $k_1=\underline{\kappa_1}$ and $k_2 = \underline{\kappa_2}$, then:
\begin{enumerate}
\item if $\kappa_1=\kappa_2$, then the lisse $\ell$-adic sheaf $\mu_{\ell}^{\pi_0(K_0)}(H^0(\Gamma_0, V^{0,0}_{I_F^0}))$ on $Z$ is non-zero;
\item if $\kappa_1 \neq \kappa_2$, then the lisse $\ell$-adic sheaf $\mu_{\ell}^{\pi_0(K_0)}(H^d(\Gamma_0, V^{0,0}_{I_F^0}))$ on $Z$ is non-zero;
\item the lisse $\ell$-adic sheaf $\mu_{\ell}^{\pi_0(K_0)}(H^d(\Gamma_0, V^{0,d}_{I_F^1}))$ on $Z$ is non-zero. If moreover $\kappa_1$ and $\kappa_2$ have the same parity\footnote{\label{footnoteparity}This restriction on parity is necessary in order to apply the results from \cite{Fre90}, which in turn depend on the formulae for the dimension of certain spaces of cusp forms proved in \cite{Shi63}. By \cite[Note 11, pag. 63]{Shi63}, it is possible that these formulae could admit a suitable generalisation, such that the hypothesis on parity could be removed.}, it is locally of dimension $> h$, where $h:=\vert \Gamma_{0,ss} \backslash \mathbb{P}^1(F) \vert$ is the (strictly positive) number of cusps of the (complex analytic, connected) Hilbert-Blumenthal variety $X_{\Gamma_{0,ss}}$ introduced in the proof of Lemma \ref{restrcoh0}.\ref{restr2}.
\end{enumerate}
\end{proposition}
\begin{proof} If $\kappa_1=\kappa_2$, then the spectral sequence considered in Construction \ref{extarithmbis} shows that the space $H^0(\Gamma_0, V)$ is isomorphic to $H^0(\mbox{det}\ \Gamma_{0}, H^0(\Gamma_{0,ss}, V))$, which, by the proof of Lemma \ref{trivbis}, is in turn isomorphic to $H^0(\Gamma_{0,ss}, V)$. Moreover, the description in \eqref{restrGL2} tells us that here, $V$ is the trivial $\mbox{SL}^d_{2,L}$-representation, and thus, it is a 1-dimensional $L$-vector space. This shows point (1).

Assume then to be in one of the two following cases: either $\kappa_1 \neq \kappa_2$ and $V$ is the irreducible $\mbox{SL}^d_{2,L}$-representation $V^{0,0}_{I_F^0}$ (which in this case is non-trivial), or $V:=V^{0,d}_{I_F^1}$ (which by the description in \eqref{restrGL2} is then isomorphic to $\bigotimes\limits_{\sigma \in I_F} \mbox{Sym}^{\kappa_1+\kappa_2+2}\mbox{V}$, where V is the standard $2$-dimensional $L$-representation of $\mbox{SL}_{2,L}$, so that $V^{0,d}_{I_F^1}$ is never trivial). 

In both cases, the same Remark \ref{extarithmbis} and Remark \ref{restrcoh0} show that the space $H^d(\Gamma_0, V)$ is isomorphic to $H^0(\mbox{det}\ \Gamma_{0}, H^d(\Gamma_{0,ss}, V))$, which by the hypothesis on $k_1$ and $k_2$ and by the proof of Lemma \ref{trivbis} is in turn isomorphic to $H^d(\Gamma_{0,ss}, V)$. Now, for every integer $\tilde{\kappa}>0$, \cite[Thm. 1.1(iv)]{M-SSYZ15} shows that $\mbox{dim} \ H^d(\Gamma_{0,ss}, \bigotimes\limits_{\sigma \in I_F} \mbox{Sym}^{\tilde{\kappa}}\mbox{V})=h+\delta(\Gamma_{0,ss}, \tilde{\kappa})$, where $\delta(\Gamma_{0,ss}, \tilde{\kappa})$ is a non-negative integer which depends on $\Gamma_{0,ss}$ and on $\tilde{\kappa}$. 
This is enough to show $(2)$ and the first half of $(3)$.

To finish the proof of $(3)$, suppose that $\kappa_1$ and $\kappa_2$ have the same parity and put $\kappa_1 + \kappa_2=:2 \kappa$. We will show that, in this case, $\delta:=\delta(\Gamma_{0,ss}, 2\kappa+2)>0$. Actually, \cite[Thm. 1.1 (iv)]{M-SSYZ15} shows that, more precisely, $\delta=h_{I_F}+\delta^\prime$, where $\delta^\prime$ is a certain positive integer and $h_{I_F}$ is the dimension of the space of \emph{cusp forms} of (parallel) weight $2\kappa+4$ with respect to the group $\Gamma_{0,ss}$. Thus, in order to conclude, it is enough to show that this dimension is strictly positive.

Let $X_{\Gamma_{0,ss}}$ be the complex analytic Hilbert-Blumenthal variety associated to $\Gamma_{0,ss}$. According to \cite[Chap. II, Thm. 3.5]{Fre90}, we have
\begin{equation} \label{dimcuspvol}
h_{I_F}=\mbox{vol}(X_{\Gamma_{0,ss}})(2\kappa+3)^d+L_{cusp},
\end{equation} 
where $L_{cusp}$ is a (not necessarily positive) integer which does not depend on $\kappa$ (recall that $\Gamma_{0,ss}$ is neat). Now, if $d$ is odd, then the discussion in \cite[page 111]{Fre90} implies that $L_{cusp}=0$, so that we obtain $h_{I_F}>0$, as desired. 

If instead $d$ is even, let us consider a smooth \emph{toroidal compactification} $\bar{X}_{\Gamma_{0,ss}}$ of $X_{\Gamma_{0,ss}}$. Then, by applying the Hirzebruch-Riemann-Roch theorem to certain locally free (\emph{automorphic}) coherent sheaves on $\bar{X}_{\Gamma_{0,ss}}$, the authors show in \cite[Prop. 7.10]{M-SSYZ15} that
\begin{equation} \label{dimcuspHRR}
h_{I_F}=\chi(\bar{X}_{\Gamma_{0,ss}}, \mathcal{O}_{\bar{X}_{\Gamma_{0,ss}}})+\epsilon
\end{equation} 
for a certain integer $\epsilon$. Now, \cite[Chap. II, Thm. 4.8]{Fre90}, implies that, if $d$ is even, $\chi(\bar{X}_{\Gamma_{0,ss}}, \mathcal{O}_{\bar{X}_{\Gamma_{0,ss}}})>0$ (this quantity is in particular equal to $1$ plus the dimension of the space of cusp forms of weight $2$ with respect to $\Gamma_{0,ss}$) and that 

\begin{equation} \label{eulchar}
\chi(\bar{X}_{\Gamma_{0,ss}}, \mathcal{O}_{\bar{X}_{\Gamma_{0,ss}}})=\mbox{vol}(X_{\Gamma_{0,ss}})+L_{cusp}
\end{equation} 
(let us stress the fact that $L_{cusp}$ is the \emph{same} integer of equation \eqref{dimcuspvol}). By replacing the expression \label{eulchar} for $\chi(\bar{X}_{\Gamma_{0,ss}}, \mathcal{O}_{\bar{X}_{\Gamma_{0,ss}}})$ into equation \eqref{dimcuspHRR}, the equality between the two expressions \eqref{dimcuspvol} and \eqref{dimcuspHRR} for $h_{I_F}$ tells us that $\epsilon=\mbox{vol}(X_{\Gamma_{0,ss}})(2\kappa+3)^d-\mbox{vol}(X_{\Gamma_{0,ss}})>0$. The equation \eqref{dimcuspHRR} then implies that $h_{I_F}>0$ in this case too.
\end{proof}

\subsection{The degeneration along the Klingen strata} \label{degkling}

Let $\lambda$, $V_{\lambda}$ be as in Subsection \ref{degsieg} and let us now study, by using Theorem \ref{thm:PinkThm}, the degeneration of $\mu_{\ell}^K(V_{\lambda})$ along the Klingen strata. The group $G_1$ in their underlying Shimura datum is isomorphic to $\mbox{Res}_{F|\mathbb{Q}} \mbox{GL}_{2,F} \times_{\mbox{Res}_{F|\mathbb{Q}} \mathbb{G}_{m,F}, \mbox{det}} \mathbb{G}_m$ (cfr. \ref{strates}). 

\subsubsection{\textbf{Weights in the cohomology of the unipotent radical.}} \label{sec-cohunip1}
As before, let us start by identifying the possible weights appearing in the degeneration along the Siegel strata, i.e. in the $(Q_1/W_1)_L$-representations 
\begin{equation} \label{isocohunip0}
H^q(W_{1,L}, V_{\lambda}) \simeq \bigoplus\limits_{\Psi \in \mathcal{P}_q} V^{1,q}_{\Psi},
\end{equation}
for $q \in \{0, \dots, 3d \}$ (cfr. \eqref{isocohunip}). Recall from \eqref{isoklingen} that 
\begin{equation*}
(Q_1/W_1)_{L} \simeq ((\prod\limits_{\sigma \in I_F} (\mbox{GL}_{2,L})_{\sigma}) \times_{\prod\limits_{\sigma \in I_F} (\mathbb{G}_{m, L})_{\sigma}} \mathbb{G}_{m,L} ) \times \prod\limits_{\sigma \in I_F} (\mathbb{G}_{m, L})_{\sigma}
\end{equation*}

We are now going to compute the weight of the pure Hodge structure carried by each irreducible summand $V^{1,q}_{\Psi}$.

\begin{lemma}
For every $q \in \{0, \dots, 3d \}$ and for every $q$-admissible decomposition $\Psi$ as in Notation \ref{notadmpart}, the action of the factor isomorphic to 
\begin{equation*}
(\prod\limits_{\sigma \in I_F} (\GL_{2,L})_{\sigma}) \times_{\prod\limits_{\sigma \in I_F} (\mathbb{G}_{m, L})_{\sigma}, \prod\det} \mathbb{G}_{m,L}
\end{equation*}
inside $(Q_1/W_1)_L$ induces on $V^{1,q}_{\Psi}$ a pure Hodge structure of weight 
\begin{equation}\label{poidsd}
w(\lambda)-[ \sum\limits_{\sigma \in I_F^0} k_{1,\sigma} +\sum\limits_{\sigma \in I_F^1} (k_{2,\sigma}-1) -\sum\limits_{\sigma \in I_F^2} (k_{2,\sigma}+3) -\sum\limits_{\sigma \in I_F^3} (k_{1,\sigma}+4) ].
\end{equation}
\end{lemma}
\begin{proof}
By the discussion in \ref{levicomp}, the $L$-points of the factor of $(Q_1/W_1)_L$ that we are considering are identified with the subgroup 
\begin{equation*}
\{ ( \left( \begin{array} {cccc} 
\rho & & & \\
& \tau_{\sigma} & & \\
& & 1 & \\
& & & \tau_{\sigma}^{-1} \rho
\end{array} \right) )_{\sigma \in I_F}  \ \vert \ \rho \in L^{\times}, \tau_{\sigma} \in L^{\times} \ \mbox{for every} \ \sigma \in I_F \}
\end{equation*}
of $Q_1/W_1(L)$. By the convention fixed in \ref{convhodge} and the definition in \eqref{donnée1} of the Shimura datum $(G_1, \mathfrak{H}_1)$, the expression given in Lemma \ref{charkost}.\ref{charakost2} for the highest weight of the representation $V^{1,q}_{\Psi}$ yields the formula in the statement (recalling \eqref{actcar}). 
\end{proof}

\subsubsection{\textbf{Cohomology of the arithmetic subgroup.}} \label{arithm1}
Consider now the arithmetic group $\Gamma_1$ of Rmk. \ref{compl}, which is identified with a torsion-free arithmetic subgroup of $\mbox{Res}_{F|\mathbb{Q}}\mathbb{G}_{m,F}(\mathbb{Q})=F^{\times}$ (cfr. Remark \ref{compl}). We need to identify the cohomology spaces
\begin{equation*}
H^p(\Gamma_1, H^q(W_{1,L}, V_{\lambda})) \simeq \bigoplus_{\substack{\Psi \in \mathcal{P}_q}} H^p(\Gamma_1, V^{1,q}_{\Psi}).
\end{equation*} 
Reasoning as in the proof of Lemma \ref{nececond}, we obtain:
\begin{lemma} \label{nececond2}
The group $\Gamma_1$ is isomorphic to $\BZ^{d-1}$, and for every $q \in \{0, \dots, 3d \}$, its action on $V^{1,q}_{\Psi}$ is trivial if and only if there exists an integer $\kappa$ such that
\begin{equation} \label{Kostpard}
\left\{ \begin{array}{cc} 
k_{1,\sigma}=\kappa & \forall \sigma \in I_F^0 \\
k_{2,\sigma}-1=\kappa & \forall \sigma \in I_F^1
\\
-(k_{2,\sigma}+3)=\kappa & \forall \sigma \in I_F^2 \\
-(k_{1,\sigma}+4)=\kappa & \forall \sigma \in I_F^2
\end{array}
\right.
\end{equation}
(remembering Notation \ref{notadmpart}). 
\end{lemma}

Hence, we are led to pose the following:

\begin{definition} \label{triv1}
We say that $\lambda$ is $(\kappa,1)$\emph{-Kostant parallel} with respect to a $q$-admissible decomposition $\Psi$ if $\lambda$ satisfies condition \eqref{Kostpard} with respect to $\kappa \in \BZ$. 

\end{definition}

\begin{definition} \label{def:parttype2}
A $q$-admissible decomposition $\Psi$ is said to be $(\lambda,1)$\emph{-admissible} if there exists $\kappa \in \mathbb{Z}$ such that $\lambda$ is $(\kappa,1)$-Kostant parallel with respect to $\Psi$. The set of $q$-admissible decompositions which are moreover $(\lambda,1)$-admissible will be denoted by $\mathcal{P}_q^{(\lambda,1)}$.
\end{definition}

Then, the proof of the following lemma is completely analoguous to the proof of Lemma \ref{cohunip0}:
\begin{lemma} \label{cohunip1}
For every $s \notin \{0, \dots, d-1 \}$, the cohomology space $H^s(\Gamma_{1}, V^{1,q}_{\Psi})$ is trivial. For every $s \in \{0, \dots, d-1 \}$, it is non-trivial if and only if $\lambda$ is $(\kappa,1)$-Kostant parallel with respect to $\Psi$ and one of the following two conditions holds:
\begin{enumerate} [wide, labelwidth=!, labelindent=0pt]
\item \label{cohunip1.cas1}
$I_F=I_F^0 \sqcup I_F^1$. In this case, $q \in \{0, \dots, d \}$ and $\mbox{\emph{Gr}}_{w(\lambda)-d\kappa}^{\mathbb{W}} V^{1,q}_{\Psi} \neq \{ 0 \}$;
\item \label{cohunip1.cas3} $I_F=I_F^2 \sqcup I_F^3$. In this case, $q \in \{2d, \dots, 3d \}$ and $\mbox{\emph{Gr}}_{w(\lambda)-d\kappa}^{\mathbb{W}} V^{1,q}_{\Psi} \neq \{ 0 \}$.
\end{enumerate}
\end{lemma}

\begin{remark} \label{rmk:coincpoids2}
An essential ingredient for the proof of the above lemma is again the same phenomenon of Rmk. \ref{rmk:coincpoids}.
\end{remark}

Remembering Definition \ref{def:parttype2}, the isomorphism in \ref{thm:PinkThm}.\ref{itm:decobj} for a stratum $Z^\prime$ of $\partial S_K^*$ contributing to $Z_1$ becomes now
\begin{equation}\label{deccohd}
\restr{R^n i^*_1 i^* j_* \mu_{\ell}^K(V_{\lambda})}{Z^\prime} \simeq 
\bigoplus_{\substack{p+q=n}} \mu_{\ell}^{\pi_1(K_1)}(\bigoplus_{\substack{\Psi \in \mathcal{P}_q^{(\lambda,1)}}} H^p(\Gamma_1, V^{1,q}_{\Psi}))
\end{equation}
and we want to study the weight-graded objects
\begin{equation}\label{decgradd}
\mbox{Gr}_k^{\mathbb{W}} \restr{R^n i^*_1 i^* j_* \mu_{\ell}^K(V_{\lambda})}{Z^\prime}\simeq\bigoplus_{\substack{p+q=n}} \mu_{\ell}^{\pi_1(K_1)} ( \bigoplus_{\substack{\Psi \in \mathcal{P}_q^{(\lambda,1)}}} H^p(\Gamma_1, \mbox{Gr}_k^{\mathbb{W}} V^{1,q}_{\Psi}) ).
\end{equation}
\subsubsection{\textbf{Computation of weights along the Klingen strata.}}
In order to describe the weights appearing in the degeneration of the canonical construction along the Klingen strata (in the cohomological degrees which will be needed in the sequel), we just need a last preliminary remark: 

\begin{remark} 
\begin{enumerate}[wide, labelwidth=!, labelindent=0pt]
\item Suppose the dominant weight $\lambda=\lambda((k_{1,\sigma}, k_{2,\sigma})_{\sigma \in I_F}, c)$ of $G_L$ to be $(\kappa, 1)$-Kostant parallel with respect to a decomposition $\Psi$ of the form $(I_F^0, I_F^1, \varnothing, \varnothing)$. Then, by the definition of Kostant-parallelism and the hypothesis on $\lambda$, we see that such $\kappa$ and $\Psi$ are necessarily unique, except if $k_1=\underline{\kappa_1}$ and $k_2=\underline{\kappa_2}$. In this last case, there exist exactly two pairs $(\kappa, (I_F^0, I_F^1))$ such that $\lambda$ is $\kappa$-Kostant-parallel with respect to $(I_F^0, I_F^1)$, i.e. $(\kappa_1,I_F^0)$ and $(\kappa_2-1, I_F^1)$. 
\item The condition on $\lambda$ of being $(\kappa, 1)$-Kostant parallel with respect to a decomposition $\Psi$ of the form $(I_F^0, I_F^1, \varnothing, \varnothing)$ coincides with the condition of being $\kappa$-Kostant parallel introduced in Definition \ref{corang}; hence, we will adopt this terminology in the following. Moreover, by the preceding point, whenever we suppose $\lambda$ to be $\kappa$-Kostant parallel with respect to a decomposition such that $I_F^0 \neq \varnothing$, resp. $I_F^1 \neq \varnothing$, then $I_F^0$, resp. $I_F^1$, is uniquely determined by $\lambda$. 
\end{enumerate}
\label{unipart}
\end{remark}

Then, keeping in mind the above Remark, by employing Lemma \ref{cohunip1} and reasoning along the same lines of the proof of Prop. \ref{poidscoh0}, we deduce the following:
\begin{proposition}\label{poidscohd} 

Let $V_{\lambda}$ be the irreducible $L$-representation of $G_L$ of highest weight $\lambda=\lambda((k_{1,\sigma}, k_{2,\sigma})_{\sigma}, c)$ and $Z^\prime$ be a stratum of $\partial S_K^*$ contributing to $Z_1$.
\begin{enumerate}[wide, labelwidth=!, labelindent=0pt]
\item Let $n < 0$ or $n>4d-1$. Then, $\restr{R^n i^*_1 i^* j_* \mu_{\ell}^K(V_{\lambda})}{Z^\prime}$ is zero.
\item Let $n \in \{0, \dots, d-1 \}$. Then the $\ell$-adic sheaf $\restr{R^n i^*_1 i^* j_* \mu_{\ell}^K(V_{\lambda})}{Z^\prime}$ on $Z^\prime$ is non-zero if and only if the following hold:
\begin{itemize}
\item $\lambda$ is $\kappa$-Kostant parallel and $I_F^0 \neq \varnothing$;
\item posing $d_1:=\vert I_F^1 \vert \in \{0, \dots, d-1 \}$, we have $n \geq d_1$. 
\end{itemize}
In this case,
\begin{equation*}
\restr{R^n i^*_1 i^* j_* \mu_{\ell}^K(V_{\lambda})}{Z^\prime}\simeq \mu_{\ell}^{\pi_1(K_1)}( H^{n-d_1}(\Gamma_1, V^{1,d_1}_{\Psi}))
\end{equation*}
and it is pure of weight $w(\lambda)-d\kappa$. 
\item Let $n \in \{ d, \dots, 2d-1 \}$. Then the $\ell$-adic sheaf $\restr{R^n i^*_1 i^* j_* \mu_{\ell}^K(V_{\lambda})}{Z^\prime}$ on $Z^\prime$ is non-zero if and only if 
the following hold:
\begin{itemize} 
\item $\lambda$ is $\kappa$-Kostant parallel and $I_F^1 \neq \varnothing$;
\item posing $d_1:=\vert I_F^1 \vert \in \{1, \dots, d \}$, we have $n \leq d-1+d_1$. 
\end{itemize}
In this case,
\begin{equation*}
\restr{R^n i^*_1 i^* j_* \mu_{\ell}^K(V_{\lambda})}{Z^\prime}\simeq \mu_{\ell}^{\pi_1(K_1)}( H^{n-d_1}(\Gamma_1, V^{1,d_1}_{\Psi}))
\end{equation*}
and, denoting $\kappa_2:=\kappa+1$, it is pure of weight $w(\lambda)+d-d\kappa_2$.
\item Let $n \in \{2d, \dots, 3d-1 \}$. Then the $\ell$-adic sheaf $\restr{R^n i^*_1 i^* j_* \mu_{\ell}^K(V_{\lambda})}{Z^\prime}$ on $Z^\prime$ is non-zero if and only if $\lambda$ is $(\kappa,1)$-Kostant parallel with respect to a $\Psi$ such that $I_F = I_F^2 \sqcup I_F^3$ and $I_F^2$ is non-empty.
In this case, posing $d_3:=\vert I_F^3 \vert \in \{0, \dots, d-1 \}$, 
\begin{equation*}
\restr{R^n i^*_1 i^* j_* \mu_{\ell}^K(V_{\lambda})}{Z^\prime}\simeq \mu_{\ell}^{\pi_1(K_1)}( H^{n-2d-d_3}(\Gamma_1, V^{1,2d+d_3}_{\Psi}))
\end{equation*} 
and, denoting $\kappa_3=-\kappa-3$, it is pure of weight $w(\lambda)+3d+d\kappa_3$.
\end{enumerate}
\end{proposition}

\subsection{The double degeneration along the cusps of the Klingen strata}

\label{doubledeg}

Keep the notation of Thm. \ref{thm:mainthm}. In order to study the weights of the motive $i^*j_* ^{\lambda} \mathcal{V}$, the study of the degeneration of the canonical construction to each stratum of $\partial S^*_K$ will not be enough: in Lemma \ref{evittrunc}, we will also need to consider a \emph{double} degeneration, the one of mixed sheaves on the Klingen strata, \emph{already obtained by degeneration}, along the boundary of the closure in $\partial S^*_K$ of the Klingen strata themselves.

By paragraph \ref{strates}, every stratum $Z^\prime$ of $\partial S_K^*$ contributing to $Z_1$ (as defined in \ref{pink}) is (a smooth quotient by the action of a finite group of) a Hilbert-Blumenthal variety $S_{\pi_1(K_1)}$ of dimension $d$. Remember from the same paragraph \ref{strates} that the Shimura datum underlying $S_{\pi_1(K_1)}$ corresponds to the algebraic group $G_1 \simeq \mbox{Res}_{F|\mathbb{Q}} \mbox{GL}_{2,F}  \times_{\mbox{Res}_{F|\mathbb{Q}} \mathbb{G}_m, \mbox{det}} \mathbb{G}_m$, whose $L$-points are identified, up to conjugation, with

\[
G_1(L)=\{ ( \left( \begin{array} {cccc} 
\rho & & & \\
& a_{\sigma} & & b_{\sigma} \\
& & 1 & \\
& c_{\sigma} & & d_{\sigma}
\end{array} \right) )_{\sigma \in I_F}  \ \vert
a_{\sigma}, b_{\sigma}, c_{\sigma}, d_{\sigma} \in L, \rho \in L^{\times},
\]
\[
 \ \mbox{such that} \ \rho=a_{\sigma}d_{\sigma}-b_{\sigma}c_{\sigma} \ \mbox{for every} \ \sigma \in I_F
\} =
\]
\[
= \{ (A_{\sigma})_{\sigma \in I_F} \in \prod_{\sigma \in I_F} \mbox{GL}_{2, L}(L) \ \mbox{such that} \ \mbox{det}(A_{\sigma})=\mbox{det}(A_{\hat{\sigma}}) \ \forall \ \sigma, \hat{\sigma} \in I_F  \}.
\]

The boundary $\partial S_{\pi_1(K_1)}^*$ of the Baily-Borel compactification $S_{\pi_1(K_1)}^*$ of $S_{\pi_1(K_1)}$ is 0-dimensional: it is in fact a finite disjoint union of strata (called \emph{cusps}), obtained as Shimura varieties coming from the group $\mathbb{G}_m$. Fix such a stratum $Z^{\prime\prime}$, corresponding up to conjugation, in the formalism of \ref{strctcomp}, to the standard Borel subgroup of $G_1$ (denoted by $Q_2$ for the sake of coherence with the notations in the sequel): it is a representative of the unique $G_1(\mathbb{Q})$-conjugacy class of standard maximal parabolics of $G_1$. Its unipotent radical will be denoted by $W_2$. The Levi component of $Q_2$ is a torus $T_{1}$ isomorphic to
\begin{equation} \label{levihilb}
\mathbb{G}_m \times\mbox{Res}_{F|\mathbb{Q}} \mathbb{G}_{m,F}
\end{equation}
via the isomorphism defined on $L$-points by 
\begin{align*}
T_1(L) \simeq \mathbb{G}_{m}(L) &\times \prod\limits_{\sigma \in I_F} \mathbb{G}_m(L)_{\sigma} \\
\left( \begin{array} {cccc} 
\beta  & & & \\
 & \beta u_{\sigma} & & \\
 & & 1 & \\
 & & & u_{\sigma}^{-1}
\end{array} \right) )_{\sigma \in I_F} &\mapsto (\beta, (u_{\sigma})_{\sigma \in I_F}).
\end{align*}

Then, the Shimura datum $(G_2, \mathfrak{H}_2)$ underlying $Z^{\prime\prime}$ is such that the $L$-points of the group $G_2 \simeq \mathbb{G}_m$ are identified with

\[
G_2(L)=\{ ( \left( \begin{array} {cc} 
\beta \cdot I_2 & \\
& I_2
\end{array} \right) )_{\sigma \in I_F}  \ \vert
 \beta \in L^{\times} \} \hookrightarrow T_1(L)
\]
and $\mathfrak{H}_2$ is defined exactly as $\mathfrak{H}_0$ in \eqref{donnée0} (cfr. \cite[Example 12.21]{Pin90}).

\subsubsection{\textbf{The degeneration of the canonical construction along the cusps of Hilbert-Blumenthal varieties.}} \label{degcusps}

Let $j^\prime$ be the open immersion of $S_{\pi_1(K_1)}$ in $S_{\pi_1(K_1)}^*$ and adopt the notations of paragraph \ref{pink}, by replacing $j$ with $j^\prime$ and $K$ with $\pi_1(K_1)$. The stratification $\Phi$ of $\partial S_{\pi_1(K_1)}^*$ is then formed by only one element, called $Z_2$. Denote by $i_2: Z_2 \hookrightarrow S_{\pi_1(K_1)}^*$ the closed immersion complementary to $j^\prime$. Let us consider a stratum $Z^{\prime\prime}$ contributing to $Z_2$ and let us spell out, thanks to Theorem \ref{KostThm}, the conclusions of Theorem \ref{thm:PinkThm}, applied this time to $\mu_{\ell}^{\pi_1(K_1)}(U_{\chi})$, where $U_{\chi}$ is a irreducible $L$-representation of $G_{1,L}$. 

Such a representation is determined by its highest weight $\chi=\chi((h_{\sigma})_{\sigma \in I_F}, g)$, where $h_{\sigma} \in \mathbb{Z}, h_{\sigma} \geq 0 \ \forall \sigma \in I_F$, $g \in \mathbb{Z}$ (we will write \textbf{h} for the vector $(h_{\sigma})_{\sigma \in I_F}$). This character is defined on the points of the maximal torus $T_{1,L}$ of $G_{1,L}$ by 

\[
( \left( \begin{array} {cccc} 
\beta & & & \\
& \beta u_{\sigma} & &  \\
& & 1 & \\
& & & u_{\sigma}^{-1}
\end{array} \right) )_{\sigma \in I_F} \mapsto \prod\limits_{\sigma \in I_F} u_{\sigma}^{h_{\sigma}} \cdot \beta^{g}.
\]
The vector \textbf{h} is called \emph{parallel} if there exists an integer $h$ such that $h_{\sigma}=h$ for every $\sigma \in I_F$. In that case, we will write \textbf{h}$=\underline{h}$. 

\begin{remark} \label{restrcaracd}
Notice that, with these conventions, the restriction to $T_{1,L}$ of the character $\lambda((k_{1,\sigma},k_{2,\sigma})_{\sigma},c)$ defined in \ref{actcar} is given by 

\begin{equation}
\chi((k_{2,\sigma})_{\sigma}, \frac{1}{2} \cdot [c+\sum\limits_{\sigma} (k_{1,\sigma}+k_{2,\sigma})]).
\end{equation}

\end{remark}

Using the notations fixed in the beginning of this subsection, we have an identification 
\begin{equation*}
(Q_{2}/W_{2})_L \simeq T_{1,L},
\end{equation*} 
so that, by Theorem \ref{KostThm}, the cohomology spaces $H^q(W_{2,L}, U_{\chi})$ are identified with representations of the group $T_{1,L}\simeq \mathbb{G}_{m, L} \times \prod\limits_{\sigma \in I_F} \mathbb{G}_{m,L}$. Let us determine the weight of the pure Hodge structure carried by each irreducible factor of these representations. 

\begin{lemma} \label{weightcusp}
Let $\chi=\chi((h_{\sigma})_{\sigma \in I_F}, g)$ as above. Then, the spaces $H^q(W_{2,L}, U_{\chi})$ are non-trivial if and only if $q \in \{0, \dots, d \}$. For each $q \in \{0, \dots, d \}$, letting $I$ run over the subsets of $I_F$ of cardinality $q$, they are direct sums of pure Hodge structures of weight 
\begin{equation}
-2g+2\sum\limits_{\sigma \in I}(h_{\sigma}+1).
\end{equation}
\end{lemma}
\begin{proof}
We begin by making explicit the data which are needed in order to apply Theorem \ref{KostThm}. By choosing $(T_{1,L}, Q_{2,L})$ as a maximal torus and a Borel of $G_{1,L}$, we can identify the set of roots $\mathfrak{r}$ of $G_{1,L}$ with $\bigsqcup\limits_{\sigma \in I_F} \mathfrak{r}_{\sigma}$, where each $\mathfrak{r}_{\sigma}$ is a copy of the set of roots of $\mbox{GL}_{2, L}$ corresponding to the obvious choice of maximal torus and Borel. For each fixed $\hat{\sigma} \in I_F$, $\mathfrak{r}_{\hat{\sigma}}$ contains only one simple root $\rho_{\hat{\sigma}}$, which, through the inclusion of $\mathfrak{r}_{\hat{\sigma}}$ inside $\mathfrak{r}$, acquires the expression $\rho_{\hat{\sigma}} = \rho_{\hat{\sigma}}((h_{\sigma})_{\sigma \in I_F}, g)$, where

\[
h_{\sigma}=\left\{ \begin{array}{cc}
2 \ & \mbox{if} \ \sigma = \hat{\sigma} \\
0 \ & \mbox{otherwise}
\end{array} 
\right. ,\
g=1.
\]

The Weyl group $\Upsilon$ of $G_{1,L}$ is in turn isomorphic to the product $\prod\limits_{\sigma \in I_F} \Upsilon_{\sigma}$, where, for each fixed $\hat{\sigma} \in I_F$, $\Upsilon_{\hat{\sigma}}$ is a copy of the Weyl group of $\mbox{GL}_{2, L}$. The latter is a finite group of order 2,  the image of whose only non-trivial element through the inclusion of $\Upsilon_{\hat{\sigma}}$ in $\Upsilon$ is given by the element of $\tau_{\hat{\sigma}}$ which acts on $X^*(T_{1,L})$ in the following way: if $\chi=\chi((h_{\sigma})_{\sigma \in I_F}, g)$, then $\tau_{\hat{\sigma}}.\chi=\chi((\ell_{\sigma})_{\sigma \in I_F}, f)$, where 

\[
\ell_{\sigma}=\left\{ \begin{array}{cc}
-h_{\sigma} \ & \mbox{if} \ \sigma = \hat{\sigma} \\
h_{\sigma} \ & \mbox{otherwise}
\end{array} 
\right. ,\
f=g-h_{\hat{\sigma}}.
\]

By employing the notations of \ref{kost}, it is now clear that, with respect to the only parabolic of $G_{1,L}$ (up to conjugation), i.e. $Q_{2,L}$, we have $\Upsilon^{\prime}=\Upsilon$, and that, if $w=(w_{\sigma})_{\sigma \in I_F} \in \Upsilon^\prime \simeq \prod\limits_{\sigma \in I_F} \Upsilon_{\sigma}$, we have $\ell(w)= \sharp \{ \sigma \in I_F \vert w_{\sigma}=\tau_{\sigma} \}$. 

The explicit computation of $w.(\chi+\rho)-\rho$ (for $w \in \Upsilon^\prime$) and Theorem \ref{KostThm} now give the isomorphisms

\begin{equation}
H^q(W_{2,L}, U_{\chi}) \simeq \bigoplus\limits_{I \subset I_F \ \mbox{\tiny{s.t.}} \ \vert I \vert = q} U^q_I, 
\end{equation}
where the $U^q_I$'s are 1-dimensional $L$-vector spaces on which $\mathbb{G}_{m, L} \times \prod\limits_{\sigma \in I_F} \mathbb{G}_{m,L}$ acts via the character 
\begin{equation*}
\chi^\prime((l_{\sigma})_{\sigma \in I_F}, g^\prime)
\end{equation*}
defined by 

\[
l_{\sigma} = \left\{ \begin{array}{cc}
h_{\sigma} \  & \mbox{if} \ \sigma \notin I \\
-h_{\sigma}-2 \ & \mbox{if} \ \sigma \in I
\end{array} 
\right. ,\
g^\prime=g-\sum\limits_{\sigma \in I} (h_{\sigma}+1).
\]

To obtain the statement, it is now sufficient to remember that the Hodge structure on each $U_I^q$ is induced by the action of the real points of the factor $\mathbb{G}_m$ of $T_1$, corresponding to the Shimura datum $(G_2, \mathfrak{H}_2)$ defined in \ref{doubledeg}, and to employ the convention fixed in \ref{convhodge}.
\end{proof}

Consider now the group $\Gamma_{2}$ of Rmk. \ref{compl}, which is a torsion-free arithmetic subgroup of $\mbox{Res}_{F|\mathbb{Q}} \mathbb{G}_m(\mathbb{Q})$, i.e. of $F^{\times}$ (cfr. the isomorphism \eqref{levihilb}). By the same argument as in the proof of Lemma \ref{nececond}, it is isomorphic to $\mathbb{Z}^{d-1}$. We need to study the cohomology spaces 
\begin{equation*}
H^p(\Gamma_2, H^q(W_{2,L}, U_{\chi})) \simeq \bigoplus\limits_{I \subset I_F \ \mbox{\tiny{s.t.}} \ \vert I \vert = q} H^p(\Gamma_2, U^q_I).
\end{equation*}

By choosing generators $\omega_1, \dots, \omega_{d-1}$, $\Gamma_2$ is identified with the subgroup  

\begin{equation}
\{ ( \left( \begin{array} {cccc} 
1 & & & \\
 & \sigma(t) & & \\
 & & 1  & \\
 & & & \sigma(t^{-1})
\end{array} \right) )_{\sigma \in I_F}  \ \vert t=\omega_1^{p_1} \dots \omega_{d-1}^{p_{d-1}}, p_1, \dots, p_{d-1} \in \mathbb{Z} \} \hookrightarrow T_{1,L}(L),
\end{equation}
and an element $t \in \Gamma_2$ acts on $U^q_I$ by multiplication by $\prod\limits_{\sigma \notin I} \sigma(\omega)^{h_{\sigma}} \cdot \prod\limits_{\sigma \in I} \sigma(\omega)^{-h_{\sigma}-2}$. By reasoning as in the proof of Lemma \ref{cohunip0}, and employing Lemma \ref{weightcusp}, we get:

\begin{lemma} \label{cohcusps}
The cohomology space $H^p(\Gamma_{2}, U^q_I)$ is non-zero if and only if \textbf{\emph{h}}$ \ =\underline{h}$ and one of the following conditions is satisfied:
\begin{enumerate}[wide, labelwidth=!, labelindent=0pt]
\item $I= \varnothing$. In this case, $q=0$ and the Hodge structure on $U_I^0=H^0(W_{2,L}, U_{\chi})$ is pure of weight $-2g$;
\item $I= I_F$. In this case, $q=d$ and the Hodge structure on $U_I^d=H^d(W_{2,L}, U_{\chi})$ is pure of weight $-2g+2d+2dh$.
\end{enumerate}
\end{lemma}

The isomorphism of Theorem \ref{thm:PinkThm}.\ref{itm:decobj} for a stratum $Z^{\prime\prime}$ contributing to $\partial S_{\pi_1(K_1)}^*$ now becomes

\begin{equation}
\restr{R^n (i_2)^* j^\prime_* \mu_{\ell}^{\pi_1(K_1)}(U_{\chi})}{Z^{\prime\prime}} \simeq \bigoplus_{\substack{p+q=n}} \bigoplus\limits_{I \subset I_F \ \mbox{\tiny{s. t.}} \ \vert I \vert = q} \mu_{\ell}^{\pi_2(K_2)}(H^p(\Gamma_2, U^q_I)).
\end{equation}
The computation of the weights of these cohomology objects is then a direct consequence of \ref{cohcusps}:

\begin{proposition} \label{poidscohcusps}
Let $U_{\chi}$ be the irreducible representation of $G_{1,L}$ of highest weight 
$\chi=\chi(\mbox{\textbf{\emph{h}}}, g)$ and $Z^{\prime\prime}$ a stratum contributing to $\partial S_{\pi_1(K_1)}^*$. Then:
\begin{enumerate}[wide, labelwidth=!, labelindent=0pt]
\item Let $n < 0$ or $n > 2d-1$. Then $\restr{R^n (i_2)^* j^\prime_* \mu_{\ell}^{\pi_1(K_1)}(U_{\chi})}{Z^{\prime\prime}}$ is zero.
\item Let $n \in \{0, \dots, d-1\}$. Then $\restr{R^n (i_2)^* j^\prime_* \mu_{\ell}^{\pi_1(K_1)}(U_{\chi})}{Z^{\prime\prime}}$ is non-zero if and only if $ \mbox{\textbf{\emph{h}}}$ is parallel. In this case, it is isomorphic to $\mu_{\ell}^{\pi_2(K_2)}(H^n(\Gamma_2, U^0_I))$, and pure of weight $-2g$.
\item Let $n \in \{d, \dots, 2d-1 \}$. Then $\restr{R^n (i_2)^* j^\prime_* \mu_{\ell}^{\pi_1(K_1)}(U_{\chi})}{Z^{\prime\prime}}$ is non-zero if and only if $\mbox{\textbf{\emph{h}}}=\underline{h}$. In this case, it is isomorphic to $\mu_{\ell}^{\pi_2(K_2)}(H^{n-d}(\Gamma_{2}, U^d_I))$, and pure of weight $-2g+2d+2dh$.
\end{enumerate}
\end{proposition}

\begin{remark}
The results of Proposition \ref{poidscohcusps} had already been obtained in \cite[Thm. 3.5]{Wil12}, by slightly different considerations, when the representations $U_{\chi}$ are such that $g=0$. 
\end{remark}
\subsubsection{\textbf{The double degeneration.}}

Let $\lambda=\lambda((k_{1,\sigma},k_{2,\sigma})_{\sigma},c)$, $V_{\lambda}$ and $S_K$ be as in Subsection \ref{degkling}, and let $Z^\prime$ and $S_{\pi_1(K_1)}$ be as in paragraph \ref{degcusps}.

By \eqref{deccohd} and Theorem \eqref{thm:PinkThm}.\ref{itm:decder}-\ref{itm:decobj}, we have the following isomorphism in the derived category: 

\begin{equation}\label{compcusp}
\restr{i^*_1 i^* j_* \mu_{\ell}^K(V_{\lambda})}{Z^\prime} \simeq \bigoplus\limits_{m} \left[ \bigoplus_{\substack{p+q=m}} \mu_{\ell}^{\pi_1(K_1)}(V^{p,q}) \right] [-m],
\end{equation}
where 
\begin{equation}\label{restrep}
V^{p,q}:=\bigoplus_{\substack{\Psi \in \mathcal{P}_q^{(\lambda,1)}}} H^p(\Gamma_1, V^{1,q}_{\Psi}).
\end{equation}
In this latter direct sum, every factor is, by restriction, a certain  power of an irreducible representation of $G_{1,L}$, whose dominant weight is the one prescribed by Remark \eqref{restrcaracd} applied to the character $\lambda((\epsilon_{1,\sigma}, \epsilon_{2,\sigma})_{\sigma \in I_F}, c)$ defined in \eqref{carackling}.

Recall that the functor $\mu_{\ell}^{\pi_1(K_1)}$ used in the isomorphism \eqref{compcusp}, with values in $\mbox{Et}_{\ell, R}(Z^\prime)$, is deduced from the canonical construction functor, which takes values in $\mbox{Et}_{\ell, R}(S_{\pi_1(K_1)})$ (Remark \ref{rmk:explpink}.\ref{itm:foncstrat}). In order to study the degeneration of $\mu_{\ell}^{\pi_1(K_1)}(V^{p,q})$ along the points in the closure of $Z^\prime$ in $\partial S^*_K$, we will rather consider the sheaves on $S_{\pi_1(K_1)}$, denoted by the same symbol, which are obtained by interpreting this time $\mu_{\ell}^{\pi_1(K_1)}$ as the canonical construction functor.

Let us now apply Theorem \ref{thm:PinkThm}.\ref{itm:decobj} to $\mu_{\ell}^{\pi_1(K_1)}(V^{p,q})[-m]$ and to $S_{\pi_1(K_1)}$, by posing $p+q=m$ and by adopting the notations of \ref{degcusps}, for a stratum $\partial S_{\pi_1(K_1)}$, in order to study the weights of the objects 
\begin{equation*}
\restr{R^{n-m} (i_2)^* j^\prime_*\mu_{\ell}^{\pi_1(K_1)}(V^{p,q})}{Z^{\prime\prime}}.
\end{equation*}
We will only need this for $m \in \{2d, \dots, 3d-1 \}$. 

\begin{proposition}\label{poidscohdegcusps}
Fix two positive integers $p$ and $q$ such that $p+q \in \{2d, \dots, 3d-1 \}$ and let $V^{p,q}$ be the $L$-representation of $G_{1,L}$ defined in \eqref{restrep}, deduced from the irreducible $L$-representation $V_{\lambda}$ of $G_L$ of highest weight $\lambda=\lambda((k_{1,\sigma}, k_{2,\sigma})_{\sigma \in I_F}, c)$. Let $Z^{\prime\prime}$ be a stratum contributing to $\partial S_{\pi_1(K_1)}$. Then:
\begin{enumerate}[wide, labelwidth=!, labelindent=0pt, label=(\roman*)]
\item \label{poidscohdegcusps_casi}
if $m^\prime \in \{0, \dots, d-1 \}$, the $\ell$-adic sheaf $\restr{R^{m^\prime} (i_2)^* j^\prime_*\mu_{\ell}^{\pi_1(K_1)}(V^{p,q})}{Z^{\prime\prime}}$ on $Z^{\prime\prime}$ is non-zero if and only if $k_{1}=\underline{\kappa_1}$ and $k_{2}=\underline{\kappa_2}$. In this case, it is pure of weight $w(\lambda)+2d-d(\kappa_1-\kappa_2)$;
\item if $m^\prime \in \{d, \dots, 2d-1 \}$, the $\ell$-adic sheaf $\restr{R^{m^\prime} (i_2)^* j^\prime_*\mu_{\ell}^{\pi_1(K_1)}(V^{p,q})}{Z^{\prime\prime}}$ on $Z^{\prime\prime}$ is non-zero if and only if $k_{1}=\underline{\kappa_1}$ and $k_{2}=\underline{\kappa_2}$.
In this case, it is pure of weight $w(\lambda)+6d+d(\kappa_1+\kappa_2)$;
\item if $m^\prime \notin \{0, \dots, 2d-1 \}$, then the $\ell$-adic sheaf $\restr{R^{m^\prime} (i_2)^* j^\prime_*\mu_{\ell}^{\pi_1(K_1)}(V^{p,q})}{Z^{\prime\prime}}$ on $Z^{\prime\prime}$ is zero.
\end{enumerate}
\end{proposition}
\begin{proof}
Lemma \ref{cohunip1} implies that in order to have non-trivial cohomology objects, we must have $p \in \{0, \dots, d-1 \}$; hence, by hypothesis, we have $q \in \{d+1, \dots, 3d-1 \}$, and the same lemma then implies that $V^{p,q}$ is non-zero if and only if there exists a $q$-admissible decomposition $\Psi=(I_F^2 \neq \varnothing, I_F^3)$ of $I_F$ and an integer $\iota_1$ such that 
\[
\left\{ \begin{array}{cc} 
k_{2,\sigma}=\iota_1 & \forall \sigma \in I_F^2 \\
k_{1,\sigma}=\iota_1-1 & \forall \sigma \in I_F^3
\end{array}
\right.
\]

In this case, if ${\mathcal{P}_q^{\prime}}^ {(\lambda,d)}$ is the set of such decompositions, we have $V^{p,q}=\bigoplus_{\substack{\Psi \in {\mathcal{P}_q^{\prime}}^ {(\lambda,d)}}} H^p(\Gamma_1, V^{1,q}_{\Psi})$; the highest weight of the action of $G_{1,L}$ on $H^p(\Gamma_1, V^{1,q}_{\Psi})$ is then the restriction to $T_{1,L}$ of the character $\lambda((\epsilon_{1,\sigma}, \epsilon_{2,\sigma})_{\sigma \in I_F}, c)$ defined in \eqref{carackling}, where

\[
\epsilon_{1,\sigma} = \left\{ \begin{array}{cc}
-k_{2,\sigma}-3 \ & \mbox{if} \ \sigma \in I_F^2 \\
-k_{1,\sigma}-4 \ & \mbox{if} \ \sigma \in I_F^3
\end{array} 
\right. ,\
\epsilon_{2,\sigma} = \left\{ \begin{array}{cc}
k_{1,\sigma}+1 \ & \mbox{if} \ \sigma \in I_F^2 \\
k_{2,\sigma} \ & \mbox{if} \ \sigma \in I_F^3
\end{array} 
\right.
\]

By Remark \ref{restrcaracd}, this restriction, as a character of the maximal torus $T_{1,L}$ of $G_{1,L}$, has the form $\chi((\epsilon_{2,\sigma})_{\sigma}, \frac{1}{2} \cdot [c+\sum\limits_{\sigma \in I_F} (\epsilon_{1,\sigma}+\epsilon_{2,\sigma})])$.

Now, by Proposition \ref{poidscohcusps}, $\restr{R^{m^\prime} (i_2)^* j^\prime_*\mu_{\ell}^{\pi_1(K_1)}(V^{p,q})}{Z^{\prime\prime}}$ is non-zero if and only if $m^\prime \in \{0, \dots, 2d-1 \}$ and $\epsilon_{2,\sigma}$ is constant on  $I_F$, say equal to an integer $\iota_2$. This means that 

\[
\left\{ \begin{array}{cc} 
k_{1,\sigma}=\iota_2-1 & \forall \sigma \in I_F^2 \\
k_{2,\sigma}=\iota_2 & \forall \sigma \in I_F^3
\end{array}
\right.
\]

Thus, the fact that $k_{1,\sigma} \geq k_{2,\sigma}$ for every $\sigma \in I_F$ and that $I_F^2 \neq \varnothing$ imply that the sheaves we are interested in are non-zero if and only if $I_F=I_F^2$ and $k_{1}=\underline{\iota_2-1}, \ k_{2}=\underline{\iota_1}$. Denoting $\kappa_1:=\iota_2-1$ and $\kappa_2:=\iota_1$, we get the constants in the statement. 

In order to conclude, it is now enough to apply again Proposition \ref{poidscohcusps}, by observing that if $m^\prime \in \{0, \dots, d-1 \}$ then the highest weight of the action of $G_{1,L}$ on $H^p(\Gamma_1, V^{1,q}_{\Psi})$ is the character
\begin{equation}
((u_{\sigma})_{\sigma \in I_F}, \beta) \mapsto \prod\limits_{\sigma \in I_F} u_{\sigma}^{\kappa_1+1} \cdot \beta^{\frac{1}{2} \cdot [c-2d+d(\kappa_1-\kappa_2)] }
\end{equation}
(and analogously for the case $m^\prime \in \{d, \dots, 2d-1 \}$).
\end{proof}

\subsection{Weight avoidance}
\label{evit} 

In this section, we employ the notations of \ref{mainres} and \ref{pink}. Our aim is to use the results of the preceding paragraphs in order to prove Theorem \ref{thm:mainthm}, thanks to the criterion given by Theorem \ref{critint_a}. Thus, we have to relate, for $m \in \{0,1 \}$, the weights of the objects $\mathcal{H}^n i_m^* i^* j_{!*}(\mathcal{R}_{\ell}(^{\lambda}\mathcal{V}))$ to the weights of the objects $R^n i_m^* i^* j_{*}(\mathcal{R}_{\ell}(^{\lambda}\mathcal{V}))$, which are now known.

In the following, the symbols $\tau_{\mathcal{Z}}^{\geqslant \boldsymbol{\cdot}}$ will denote the truncation functors with respect to the perverse $t$-structure on $\mathcal{Z}$.

\subsubsection{\textbf{Weight avoidance on the Siegel strata.}}

Let us begin by studying the weight avoidance on the Siegel strata, by employing Propositions \ref{poidscoh0} and \ref{poidscohdegcusps}. 

\begin{remark} \label{redpoids}
Reasoning as in \cite[Rmk. 2.7 (a)-(b)-(c)-(d)]{Wil19b}, we have exact sequences

\begin{equation}\label{longex}
\mathcal{H}^{n-1}(i_0^*i_{1,*}\tau_{Z_1}^{\geqslant w(\lambda)+3d}i_1^*i^*j_*\mathcal{R}_{\ell}(^{\lambda}\mathcal{V}))\rightarrow  \mathcal{H}^{n} (i_0^*i^*j_{!*}\mathcal{R}_{\ell}(^{\lambda}\mathcal{V})) \rightarrow \mathcal{H}^{n}(i_0^*i^*j_*\mathcal{R}_{\ell}(^{\lambda}\mathcal{V}))
\end{equation}

for $n \leq w(\lambda)+3d-1$, whereas $\mathcal{H}^{n} (i_0^*i^*j_{!*}\mathcal{R}_{\ell}(^{\lambda}\mathcal{V}))$ is zero for $n \geq w(\lambda)+3d$; one also sees, again reasoning as in (loc. cit.), that $\mathcal{H}^{n}(i_0^*i_{1,*}\tau_{Z_1}^{\geqslant w(\lambda)+3d}i_1^*i^*j_*\mathcal{R}_{\ell}(^{\lambda}\mathcal{V}))$ is zero for $n<w(\lambda)+2d$, so that $\mathcal{H}^{n} (i_0^*i^*j_{!*}\mathcal{R}_{\ell}(^{\lambda}\mathcal{V})) \simeq \mathcal{H}^{n}(i_0^*i^*j_*\mathcal{R}_{\ell}(^{\lambda}\mathcal{V}))$ for $n < w(\lambda)+2d $.

\end{remark}

Thus, if $n \leq w(\lambda)+2d-1$, the weights of the perverse sheaf $\mathcal{H}^{n} (i_0^*i^*j_{!*}\mathcal{R}_{\ell}(^{\lambda}\mathcal{V}))$ are the same as the weights of $\mathcal{H}^{n}(i_0^*i^*j_*\mathcal{R}_{\ell}(^{\lambda}\mathcal{V}))$, which have already been computed, while there is nothing to do for $n \geq w(\lambda)+3d$. It remains to study the interval $[w(\lambda)+2d, w(\lambda)+3d-1]$. 

\begin{remark} \label{morfini}
Each stratum $Z^\prime$ of $\partial S_K^*$ contributing to $Z_1$ is the quotient of a Hilbert-Blumenthal variety $S_{K, Z^\prime}$ by the action of a finite group; let $S_{K, Z^\prime}^*$ be its Baily-Borel compactification. If $\bar{Z}_1$ is the closure of $Z_1$ in $\partial S_K^*$ and
\begin{equation}
Z_1^*:= \bigsqcup_{\substack{Z^\prime \ \mbox{\tiny{stratum of}} \ \partial S_K^* \\
\mbox{\tiny{contributing to}} \ Z_1}} S_{K, Z^\prime}^*,
\end{equation}
then there exists a surjective, finite morphism
\begin{equation}
q: Z_1^*  \rightarrow \bar {Z}_1
\end{equation}
whose restriction to each $S_{K,Z^\prime}$ is the quotient morphism from $S_{K,Z^\prime}$ to $Z^\prime$ (cfr. \cite[Main Thm. 12.3 (c), Sec. 7.6]{Pin90}). 
\end{remark}

Thanks to the above Remark, we can now compute the weights of 
\begin{equation*}
\mathcal{H}^{n}(i_0^*i_{1,*}\tau_{Z_1}^{\geqslant w(\lambda)+3d}i_1^*i^*j_*\mathcal{R}_{\ell}(^{\lambda}\mathcal{V}))
\end{equation*}
in the degrees we are interested in. 

\begin{lemma}\label{evittrunc}

If $n \in [w(\lambda)+2d, w(\lambda)+3d-1]$, then $\mathcal{H}^{n}(i_0^*i_{1,*}\tau_{Z_1}^{\geqslant w(\lambda)+3d}i_1^*i^*j_*\mathcal{R}_{\ell}(^{\lambda}\mathcal{V}))$ is non-zero if and only if $k_1 = \underline{\kappa_1}$ and $k_2=\underline{\kappa_2}$. In this case, it is pure of weight $w(\lambda)+2d-d(\kappa_1-\kappa_2)$.
\end{lemma}
\begin{proof}
Recall that $\mathcal{R}_{\ell}(^{\lambda}\mathcal{V})=\mu_{\ell}^K(V_{\lambda})[-w(\lambda)]$. Then, for each stratum $Z^\prime$ contributing to $Z_1$, we have, by Theorem \ref{thm:PinkThm}.\ref{itm:decder},
\begin{equation*}
i_1^*i^*j_*\mathcal{R}_{\ell}(^{\lambda}\mathcal{V})_{\vert_{Z^\prime}} \simeq \bigoplus\limits_k R^{k}i_1^*i^*j_*\mu_{\ell}^K(V_{\lambda})[-w(\lambda)-k]_{\vert_{Z^\prime}}.
\end{equation*}

Moreover, by Theorem \ref{thm:PinkThm}.\ref{itm:decobj}, the objects $R^{h-w(\lambda)}i_1^*i^*j_*\mu_{\ell}^K(V_{\lambda})$ are all lisse: since $Z_1$ is of dimension $d$, perverse truncation above degree $w(\lambda)+3d$ equals classical truncation above degree $w(\lambda)+2d$. Thus, by fixing a stratum $Z$ contributing to $Z_0$, we obtain 

\begin{equation}\label{changbasfirst}
\mathcal{H}^{n}(i_0^*i_{1,*}\tau_{Z_1}^{\geqslant w(\lambda)+3d}i_1^*i^*j_*\mathcal{R}_{\ell}(^{\lambda}\mathcal{V}))_{\vert_{Z}} \simeq \bigoplus\limits_{Z^\prime} \left( \mathcal{H}^{n} i_0^*i_{1,*}(\bigoplus\limits_{k \geq 2d} R^{k}i_1^*i^*j_*\mu_{\ell}^K(V_{\lambda})[-w(\lambda)-k]_{\vert_{Z^\prime}}) \right)_{\vert_{Z}},
\end{equation}
where the direct sum runs over all strata $Z^\prime$ contributing to $Z_1$ and containing $Z$ in their closure. Fix now such a stratum: as in \eqref{compcusp} and \eqref{restrep}, we get 

\begin{equation}
\bigoplus\limits_{k \geq 2d} R^{k}i_1^*i^*j_*\mu_{\ell}^K(V_{\lambda})[-w(\lambda)-k]_{\vert_{Z^\prime}} \simeq \bigoplus\limits_{k \geq 2d} (\bigoplus_{\substack{p+q=k}} \mu_{\ell}^{\pi_1(K_1)}(V^{p,q}))[-w(\lambda)-k],
\end{equation}
and as a consequence, by taking into account the fact that $Z$ is of dimension $0$, 

\begin{equation} \label{sumobj}
\left( \mathcal{H}^n i_0^*i_{1,*}(\bigoplus\limits_{k \geq 2d} R^{k}i_1^*i^*j_*\mu_{\ell}^K(V_{\lambda})[-w(\lambda)-k]_{\vert_{Z^\prime}})) \right)_{\vert_{Z}} \simeq  \bigoplus\limits_{k \geq 2d} \bigoplus_{\substack{p+q=k}} \restr{R^{n-w(\lambda)-k} i_0^*i_{1,*} \mu_{\ell}^{\pi_1(K_1)}(V^{p,q})}{Z}
\end{equation}

Now, let us adopt the notations of Remark \ref{morfini}, and extend the notations of \ref{degcusps} in the following way: $j^\prime$ will denote the open immersion of the union of the $S_{K, Z^\prime}$'s in the union of the $S_{K, Z^\prime}^*$'s, while $i_2$ will denote the complementary closed immersion of the union of the $\partial S_{K, Z^\prime}^*$'s in the union of the $S_{K, Z^\prime}^*$'s. 
By restriction to $Z^\prime$, we get, by proper base change, the relation

\begin{equation}
i_0^*i_{1,*} \mu_{\ell}^{\pi_1(K_1)} \simeq q_* i_2^* j^\prime_* \mu_{\ell}^{\pi_1(K_1)},
\end{equation}
where on the left, resp. right hand side, we have interpreted the functor $\mu_{\ell}^{\pi_1(K_1)}$ as a functor with values in $\mbox{Et}_{\ell, R}(Z^\prime)$, resp. in $\mbox{Et}_{\ell, R}(S_{\pi_1(K_1)})$. Denote now by $\partial_{Z^\prime}$ the stratum of  $S_{K,Z^\prime}$ such that $q(\partial_{Z^\prime})=Z$ (such a stratum is unique, because two rational boundary components (cfr. \ref{strctcomp}) are conjugated by $G_1(\mathbb{Q})$ if and only if they are conjugated by $G(\mathbb{Q})$, by \cite[Rmk. at page 91, (iii)]{Pin90}). Since the morphism $q$ is finite, we deduce that, for every $k, p, q$,
 
\begin{equation} \label{changbaslast}
\restr{R^{n-w(\lambda)-k} i_0^*i_{1,*} \mu_{\ell}^{\pi_1(K_1)}(V^{p,q})}{Z} \simeq q_*( \restr{R^{n-w(\lambda)-k} i_2^* j^\prime_* \mu_{\ell}^{\pi_1(K_1)}(V^{p,q})}{\partial_{Z^\prime}}).
\end{equation}

Now, the functor $q_*$ preserves weights, because the morphism $q$ is finite. Thus, the isomorphisms \eqref{changbasfirst}-\eqref{changbaslast} allow us to deduce the weights of $\mathcal{H}^{n}(i_0^*i_{1,*}\tau_{Z_1}^{\geqslant w(\lambda)+3d}i_1^*i^*j_*\mathcal{R}_{\ell}(^{\lambda}\mathcal{V}))$ from the weights of the sheaf $\restr{R^{n-w(\lambda)-k} i_2^* j^\prime_* \mu_{\ell}^{\pi_1(K_1)}(V^{p,q})}{Z^{\prime\prime}}$. But since $n \in [w(\lambda)+2d, w(\lambda)+3d-1]$ and $k \geq 2d$, then, by Proposition \ref{poidscohdegcusps}, the objects which appear as summands in the right hand side of \eqref{sumobj} are non-zero only for indices $n-w(\lambda)-k \in \{0, \dots d-1 \}$. We can then conclude by Proposition \ref{poidscohdegcusps}.\ref{poidscohdegcusps_casi}.
\end{proof}

We now dispose of all the necessary information in order to determine an interval of weight avoidance on the Siegel strata:

\begin{proposition} \label{evit0} The perverse cohomology sheaf $\mathcal{H}^{n}(i_0^*i^*j_{!*}\mathcal{R}_{\ell}(^{\lambda}\mathcal{V}))$ can be non-zero only if $k_{1}=\underline{\kappa_1}$, $k_2=\underline{\kappa_2}$ and $n \in \{w(\lambda), \dots, w(\lambda)+3d-1 \}$. In this case, $\mathcal{H}^{n} (i_0^*i^*j_{!*}\mathcal{R}_{\ell}(^{\lambda}\mathcal{V}))$ is of weight $\leq n-d(\kappa_1-\kappa_2)$ for each $n \in \mathbb{Z}$. 

If $\kappa_1$ and $\kappa_2$ have the same parity, then the weight-graded object of weight $w(\lambda)+2d-d(\kappa_1-\kappa_2)$ of the perverse sheaf
$\mathcal{H}^{w(\lambda)+2d} (i_0^*i^*j_{!*}\mathcal{R}_{\ell}(^{\lambda}\mathcal{V}))$ is non-zero.

\end{proposition}
\begin{proof}
If $n < w(\lambda)+2d $, then $\mathcal{H}^{n} (i_0^*i^*j_{!*}\mathcal{R}_{\ell}(^{\lambda}\mathcal{V})) \simeq \mathcal{H}^{n}(i_0^*i^*j_*\mathcal{R}_{\ell}(^{\lambda}\mathcal{V}))$, by Remark \ref{redpoids}. Now, $Z_0$ is of dimension 0; hence, 
\begin{equation} \label{isoperv0}
\mathcal{H}^{n}(i_0^*i^*j_*\mathcal{R}_{\ell}(^{\lambda}\mathcal{V})) \simeq R^{n}(i_0^*i^*j_*\mathcal{R}_{\ell}(^{\lambda}\mathcal{V})) = R^{n-w(\lambda)}i_0^*i^*j_*\mu_{\ell}^K(V). 
\end{equation}
Thus, the perverse sheaf $\mathcal{H}^{n} (i_0^*i^*j_{!*}\mathcal{R}_{\ell}(^{\lambda}\mathcal{V}))$ is zero for every $n<w(\lambda)$. If $k_1$ and $k_2$ are not parallel, then Proposition \ref{poidscoh0} tells us that $\mathcal{H}^{n} (i_0^*i^*j_{!*}\mathcal{R}_{\ell}(^{\lambda}\mathcal{V}))$ is zero for every $w(\lambda) \leq n < w(\lambda)+2d$. If instead $k_1=\underline{\kappa_1}$, $k_2=\underline{\kappa_2}$, then the same proposition tells us that for every $w(\lambda) \leq n < w(\lambda)+2d$, $\mathcal{H}^{n} (i_0^*i^*j_{!*}\mathcal{R}_{\ell}(^{\lambda}\mathcal{V}))$ is of weight $\leq w(\lambda)-d(\kappa_1+\kappa_2) \leq n-d(\kappa_1-\kappa_2)$. 

Let now $n \geq w(\lambda)+2d$. If $n \geq w(\lambda)+3d$, Remark \ref{redpoids} implies that $\mathcal{H}^{n}(i_0^*i^*j_{!*}\mathcal{R}_{\ell}(^{\lambda}\mathcal{V}))$ is zero. Assume then $n \in \{ w(\lambda)+2d, \dots, w(\lambda)+3d-1 \}$. In this case, by reasoning as in \eqref{isoperv0} and by applying again Proposition \ref{poidscoh0}, along with Lemma \ref{evittrunc}, we also see, by the exact sequence \eqref{longex}, that if $k_1$ and $k_2$ are not parallel, then $\mathcal{H}^{n} (i_0^*i^*j_{!*}\mathcal{R}_{\ell}(^{\lambda}\mathcal{V}))$ is zero for every $n \in \{ w(\lambda)+2d, \dots, w(\lambda)+3d-1 \}$. If instead $k_1=\underline{\kappa_1}$, $k_2=\underline{\kappa_2}$, then we see in same way that the weights that can appear in $\mathcal{H}^{n} (i_0^*i^*j_{!*}\mathcal{R}_{\ell}(^{\lambda}\mathcal{V}))$ are of the form $w(\lambda)-d(\kappa_1+\kappa_2) $ or $w(\lambda)+2d-d(\kappa_1-\kappa_2)$. In any case, we get weights $\leq n-d(\kappa_1-\kappa_2)$. 

To see that if $\kappa_1$ and $\kappa_2$ have the same parity, then weight $w(\lambda)+2d-d(\kappa_1-\kappa_2)$ does appear in the perverse sheaf $\mathcal{H}^{w(\lambda)+2d} (i_0^*i^*j_{!*}\mathcal{R}_{\ell}(^{\lambda}\mathcal{V}))$, notice that the long exact sequence \eqref{longex} gives a short exact sequence  

\begin{equation}
0 \rightarrow \mathcal{H}^{w(\lambda)+2d} (i_0^*i^*j_{!*}\mathcal{R}_{\ell}(^{\lambda}\mathcal{V})) \hookrightarrow \mathcal{H}^{w(\lambda)+2d}(i_0^*i^*j_*\mathcal{R}_{\ell}(^{\lambda}\mathcal{V})) \xrightarrow{ad} \mathcal{H}^{w(\lambda)+2d}(i_0^*i_{1,*}\tau_{Z_1}^{\geqslant w(\lambda)+3d}i_1^*i^*j_*\mathcal{R}_{\ell}(^{\lambda}\mathcal{V})), 
\end{equation}
so that $\mathcal{H}^{w(\lambda)+2d} (i_0^*i^*j_{!*}\mathcal{R}_{\ell}(^{\lambda}\mathcal{V}))$ is identified with the kernel of the arrow $ad$. Proposition \ref{poidsapp} shows that if $\kappa_1$ and $\kappa_2$ have the same parity, then $\mathcal{H}^{w(\lambda)+2d}(i_0^*i^*j_*\mathcal{R}_{\ell}(^{\lambda}\mathcal{V}))$ contains a direct factor, which is pure of weight $w(\lambda)+2d-d(\kappa_1-\kappa_2)$ and locally of dimension $> h$, where $h:=\vert \Gamma_{0,ss} \backslash \mathbb{P}^1(F)\vert$ is the (strictly positive) number of cusps of the Hilbert-Blumenthal variety $X_{\Gamma_{0,ss}}$. In order to conclude, it is then enough to show that locally, the kernel of $ad$ has non-trivial intersection with this sub-object. The isomorphisms \eqref{changbasfirst}-\eqref{changbaslast} in the proof of Lemma \ref{evittrunc} show that, if we let the index $Z^\prime$ run over all strata $Z^\prime$ contributing to $Z_1$ and containing $Z$ in their closure, and if the finite morphism $q: Z_1^* \rightarrow \bar{Z}$ is the one of Remark \ref{morfini}, then, above a stratum $Z$ of $\partial S_K^*$ contributing to $Z_0$, the arrow $ad$ has the form 
\begin{equation}
\restr{R^{2d} i^*_0 i^* j_* \mu_{\ell}^K(V_{\lambda})}{Z} \rightarrow q_*\bigoplus\limits_{Z^\prime} (\restr{R^{0} i_2^* j^\prime_* \mu_{\ell}^{\pi_1(K_1)}(V^{p,q})}{\partial_{Z^\prime}})
\end{equation}
Moreover, we know that, for each fixed $Z^\prime$, we have $\restr{R^{0} i_2^* j^\prime_* \mu_{\ell}^{\pi_1(K_1)}(V^{p,q})}{\partial_{Z^\prime}} \simeq \mu_{\ell}^{\pi_2(K_2)}(H^0(\Gamma_2, U^0_I))$, where $H^0(\Gamma_2, U^0_I)$ is a 1-dimensional $L$-vector space (cfr. the proof of Lemma \ref{weightcusp}). 

We are then reduced to show that, locally, the dimension of the target of $ad$ is strictly smaller than the dimension of the source (remember that by Lemma \ref{evittrunc}, the target is pure of weight $w(\lambda)+2d-d(\kappa_1-\kappa_2)$). But this is true, thanks to the following proposition. 
\end{proof}

\begin{proposition}
Let $Z$, $q$ and $h$ be as in the proof of Proposition \ref{evit0}. 

Then, above $Z \subset Z_0$, the number of points in the geometrical fibers of $\restr{q}{\partial Z_1^*}: \partial Z_1^* \rightarrow Z_0$ is $\leq h$. 
\end{proposition}
\begin{proof} 
The part of the proof of \cite[Prop. 2.9]{Wil19b} at pages 27-28 can be translated word by word in this context: the statement can be proven on $\mathbb{C}$-points, and by using the adelic description of the morphism obtained from $q$ by analytification, one identifies the fibre $q^{-1}(z)$ above $z \in Z(\mathbb{C})$ with $ B_2(F)\backslash\mbox{GL}_2(F)/\Gamma_0$, where $B_2$ is the standard Borel subgroup of $\mbox{GL}_2$. Now, this set is exactly the set of cusps of $\Gamma_0$, whose cardinality is $\leq$ the cardinality of the set of cusps of $\Gamma_{0,ss}$.
\end{proof}

\begin{remark}\label{fibres}
By the preceding proof, the number of $d$-dimensional strata in $\partial S_K^*$ which contain a fixed cusp is related to the number of cusps in the Baily-Borel compactification of a (complex analytic, connected) \emph{virtual} Hilbert-Blumenthal variety, which does not appear in $\partial S_K^*$, i.e. $X_{\Gamma_{0,ss}}$. Cfr. \cite[Rmk. 2.10 (c)]{Wil19b} for an analogous remark.
\end{remark}
 
\subsubsection{\textbf{Weight avoidance on the Klingen strata and proof of the main theorem.}} \label{weightevitd}

Let us now study the weight avoidance on the Klingen strata, by means of Proposition \ref{poidscohd}. 

\begin{remark} \label{redpoidsd}
By reasoning as in \cite[Rmk. 2.7 (a)-(b)-(c)]{Wil19b}, we see that

\begin{equation}
i_1^* i^* j_{!*}(\mathcal{R}_{\ell}(^{\lambda}\mathcal{V})) \simeq \tau_{Z_1}^{t\leqslant w(\lambda)+3d-1}i_1^* i^* j_{*}\mathcal{R}_{\ell}(^{\lambda}\mathcal{V}),
\end{equation}

\end{remark}

Thanks to the latter remark, and remembering Remark \ref{unipart}, we are ready to determine the interval of weight avoidance on the Klingen strata. 

\begin{proposition}
The perverse cohomology sheaf $\mathcal{H}^{n} (i_1^*i^*j_{!*}\mathcal{R}_{\ell}(^{\lambda}\mathcal{V}))$ can be non-zero only if $\lambda$ is $\kappa$-Kostant parallel and if $n \in \{w(\lambda)+d, \dots, w(\lambda)+3d-1 \}$. 

In this case, let $\kappa_2:=\kappa+1$. Then:
\begin{enumerate}[wide, labelwidth=!, labelindent=0pt]
\item if $I_F^0 \neq \varnothing$, denote $d_1:=\vert I_F^1\vert \in \{0, \dots, d-1 \}$. Then, the perverse sheaf $\mathcal{H}^{n} (i_1^*i^*j_{!*}\mathcal{R}_{\ell}(^{\lambda}\mathcal{V}))$ is of weight $\leq n-d_1-d\kappa$ for every $n \in \{ w(\lambda)+d, \dots, w(\lambda)+2d-1 \}$, and the perverse sheaf $\mathcal{H}^{w(\lambda)+d+d_1} (i_1^*i^*j_{!*}\mathcal{R}_{\ell}(^{\lambda}\mathcal{V}))$ is non-zero and pure of weight $w(\lambda)+d-d\kappa$. 

Otherwise, $\mathcal{H}^{n} (i_1^*i^*j_{!*}\mathcal{R}_{\ell}(^{\lambda}\mathcal{V}))$ is zero for every $n \in \{ w(\lambda)+d, \dots, w(\lambda)+2d-1 \}$;

\item if $I_F^1 \neq \varnothing$, then the perverse sheaf $\mathcal{H}^{n} (i_1^*i^*j_{!*}\mathcal{R}_{\ell}(^{\lambda}\mathcal{V}))$ is of weight $\leq n-d\kappa_2$ for every $n \in \{ w(\lambda)+2d, \dots, w(\lambda)+3d-1 \}$, and the perverse sheaf $\mathcal{H}^{w(\lambda)+2d} (i_1^*i^*j_{!*}\mathcal{R}_{\ell}(^{\lambda}\mathcal{V}))$ is non-zero and pure of weight $w(\lambda)+2d-d\kappa_2$.

 Otherwise, $\mathcal{H}^{n} (i_1^*i^*j_{!*}\mathcal{R}_{\ell}(^{\lambda}\mathcal{V}))$ is zero for every $n \in \{ w(\lambda)+2d, \dots, w(\lambda)+3d-1 \}$.  
\end{enumerate}

\end{proposition}
\begin{proof}
By Remark \ref{redpoidsd}, 
\begin{equation*}
\mathcal{H}^{n} (i_1^*i^*j_{!*}\mathcal{R}_{\ell}(^{\lambda}\mathcal{V}))\simeq \mathcal{H}^n i_1^* i^* j_{*}\mathcal{R}_{\ell}(^{\lambda}\mathcal{V}))
\end{equation*}
for every $n \leq w(\lambda)+3d-1$, and $\mathcal{H}^{n} (i_1^*i^*j_{!*}\mathcal{R}_{\ell}(^{\lambda}\mathcal{V}))$ is zero for every $n \geq w(\lambda)+3d$. Moreover, 
\begin{equation*}
\mathcal{H}^n i_1^* i^* j_{*}(\mathcal{R}_{\ell}(^{\lambda}\mathcal{V}))=(R^{n-w(\lambda)-d}i_1^* i^* j_{*}\mu_{\ell}(V_{\lambda}))[d].
\end{equation*}
Then, by applying Proposition \ref{poidscohd}, we see the following facts. 

If $n< w(\lambda)+d$, then $\mathcal{H}^{n} (i_1^*i^*j_{!*}\mathcal{R}_{\ell}(^{\lambda}\mathcal{V}))$ is zero.

If $n \in \{ w(\lambda)+d, \dots, w(\lambda)+2d-1 \}$, then $\mathcal{H}^{n} (i_1^*i^*j_{!*}\mathcal{R}_{\ell}(^{\lambda}\mathcal{V}))$ is non-zero if and only if $\lambda$ is $\kappa$-Kostant parallel and $I_F^0 \neq \varnothing$, and if $n\geq w(\lambda)+d+d_1$, where $d_1=\vert I_F^1 \vert \in \{0, \dots, d-1 \}$. In this case, it is pure of weight $w(\lambda)+d-d\kappa$, in particular of weight $\leq n-d_1-d \kappa$.

If $n \in \{w(\lambda)+2d, \dots, w(\lambda)+3d-1 \}$, then $\mathcal{H}^{n} (i_1^*i^*j_{!*}\mathcal{R}_{\ell}(^{\lambda}\mathcal{V}))$ is non-zero if and only if $\lambda$ is $\kappa$-Kostant parallel and $I_F^1 \neq \varnothing$, and if $n \leq w(\lambda)+2d+d_1-1$, where $d_1=\vert I_F^1 \vert \in \{1, \dots, d \}$. In this case, it is pure of weight $w(\lambda)+2d-d\kappa_2$, in particular of weight $\leq n-d \kappa_2$. 
\end{proof}

\begin{corollary} \label{evitd}
Suppose $\lambda$ to be $\kappa$-Kostant parallel and fix $n \in \{w(\lambda)+d, \dots, w(\lambda)+3d-1 \}$. Let $d_1$ and $\kappa_2$ be as in the previous proposition, and recall the notation $\mbox{\emph{cor}}(\lambda)$ from Def. \ref{corang}. Then:
\begin{enumerate}[wide, labelwidth=!, labelindent=0pt]
\item if $\mbox{\emph{cor}}(\lambda)=0$, then $\mathcal{H}^{n} (i_1^*i^*j_{!*}\mathcal{R}_{\ell}(^{\lambda}\mathcal{V}))$ is of weights $\leq n-d_1-d\kappa$, and the weight $n-d_1-d\kappa$ \emph{does appear} when $n=w(\lambda)+d+d_1$;
\item if $\mbox{\emph{cor}}(\lambda) \geq 1$, then $\mathcal{H}^{n} (i_1^*i^*j_{!*}\mathcal{R}_{\ell}(^{\lambda}\mathcal{V}))$ is of weights $\leq n-d\kappa_2$, and the weight $n-d\kappa_2$ \emph{does appear} when $n=w(\lambda)+2d$.
\end{enumerate}
\end{corollary}
\begin{proof} 
Everything follows from the previous proposition, by observing that:
\begin{enumerate}[wide, labelwidth=!, labelindent=0pt]
\item if $\mbox{cor}(\lambda)=0$, then $I_F^0 \neq \varnothing$, and either $I_F^1= \varnothing$ (in which case only the weights in the first point of the previous proposition can contribute) or $I_F^1 \neq \varnothing$ (in which case, by definition of Kostant parallelism, $n-d\kappa_2=n-d-d\kappa \leq n-d_1-d \kappa$);
\item if $\mbox{cor}(\lambda)\geq1$, then either $k_1$ is not parallel (in which case, since $\lambda$ is dominant, $\lambda$ can't be Kostant parallel with $I_F^0 \neq \varnothing$, and only the weights in the second point of the previous proposition can contribute) or $k_1=\underline{\kappa_1}$ (in which case $\kappa_1 \geq \kappa_2$ and the first point of the previous proposition gives weights $\leq n-d\kappa_1 \leq n-d \kappa_2$).
\end{enumerate}
\end{proof}

We now have all the necessary ingredients for the proof of Theorem \ref{thm:mainthm}. 

\begin{proof}{(of Theorem \eqref{thm:mainthm})}

\item We only have to apply the criterion \ref{critint_a}.\ref{realpoids} and use Proposition \ref{evit0} and Corollary \ref{evitd}. 
\end{proof}

\section{The intersection motive of genus two Hilbert-Siegel varieties} \label{consapp} 

\noindent In this section we study the properties of the \emph{intersection motive} of genus 2 Hilbert-Siegel varieties (with coefficients in suitable irreducible representations $V_{\lambda}$), whose existence follows from Thm. \ref{thm:mainthm} and Wildeshaus' theory, and the implications for the construction of motives associated to automorphic representations.

\subsection{Properties of the intersection motive} 

Adopt the notation of \ref{mainres} and assume from now on that $\lambda$ is either not completely irregular or of corank $0$. Then, the weight avoidance proved in Corollary \ref{evit01} allows us to apply the general theory developed in \cite{Wil19a}. In fact, absence of the weights 0 and 1 for $i^* j_{*} ^\lambda \mathcal{V}$ implies that $^\lambda \mathcal{V}$ belongs to a full subcategory of $CHM(S_K)$, on which an \emph{intermediate extension} functor $j_{!*}$ towards the category $CHM(S_K^*)$ is defined. Denoting by $s:S_K^* \rightarrow \mbox{Spec} \ \mathbb{Q}$ the structural morphism, we can then define, by applying \cite[Def. 3.7]{Wil19a}, the \emph{intersection motive} of $S_K$ \emph{with respect to} $S_K^*$ \emph{with coefficients in} $^\lambda \mathcal{V}$ as the object $s_*j_{!*}^\lambda \mathcal{V}$ of the category $CHM(\mathbb{Q})_L$. In the following, this object will be simply called \emph{intersection motive}. 

Let us spell out its main properties. For doing so, if $\lambda$ satisfies in addition the hypotheses of point $(1)$ or $(2)$ or $(3)$, resp. $(4)$, of Theorem \ref{thm:mainthm}, put $\beta:=d \kappa$, resp. $\beta:=\mbox{min} \{ d\kappa_1, d(\kappa_1-\kappa_2) \}$ (with notations as in the Theorem). The general theory then implies the following: 

\begin{corollary} 
Let $s$ and $\beta$ be as above, and let $\tilde{s}$ be the structural morphism of $S_K$. 
\begin{enumerate} [wide, labelwidth=!, labelindent=0pt, label=(\arabic*)]
\item \label{itm:filt} The motive $\tilde{s}_! ^\lambda \mathcal{V} \in \DBcM(\mathbb{Q})_L$ avoids weights $-\beta, -\beta+1,\dots,-1$, and the motive $\tilde{s}_* ^\lambda \mathcal{V} \in \DBcM(\mathbb{Q})_L$ avoids weights $1,2,\dots,\beta$. More precisely, the exact triangles 
\begin{equation*}
s_*i_*i^*j_{!*} ^\lambda \mathcal{V}[-1] \rightarrow \tilde{s}_! ^\lambda \mathcal{V} \rightarrow s_*j_{!*} ^\lambda \mathcal{V} \rightarrow s_*i_*i^*j_{!*} ^\lambda \mathcal{V}
\end{equation*}
and
\begin{equation*}
s_*j_{!*} ^\lambda \mathcal{V} \rightarrow \tilde{s}_* ^\lambda \mathcal{V} \rightarrow d_*i_*i^!j_{!*} ^\lambda \mathcal{V} [1] \rightarrow s_*j_{!*} ^\lambda \mathcal{V}[1]
\end{equation*}
are weight filtrations of $\tilde{s}_! ^\lambda \mathcal{V}$, resp. of $\tilde{s}_* ^\lambda \mathcal{V}$, which avoid weights $-\beta, -\beta+1,\dots,-1$, resp. $1,2,\dots,\beta$.
\item \label{itm:fonct} The intersection motive $s_*j_{!*}^\lambda \mathcal{V}$ is functorial with respect to $\tilde{s}_! ^\lambda \mathcal{V}$ and to $\tilde{s}_* ^\lambda \mathcal{V}$. In particular, every endomorphism of $\tilde{s}_! ^\lambda \mathcal{V}$ or $\tilde{s}_* ^\lambda \mathcal{V}$ induces an endomorphism of $s_*j_{!*}^\lambda \mathcal{V}$.
\item \label{itm:compl} If $\tilde{s}_! ^\lambda \mathcal{V} \rightarrow N \rightarrow \tilde{s}_* ^\lambda \mathcal{V}$ is a factorisation of $\tilde{s}_! ^\lambda \mathcal{V} \rightarrow \tilde{s}_* ^\lambda \mathcal{V}$ through a Chow motive $N \in CHM(\mathbb{Q})_L$, then the intersection motive $s_*j_{!*}^\lambda \mathcal{V}$ is canonically identified with a direct factor of $N$, with a canonical direct complement. 
\end{enumerate}
\label{cor:propint}
\end{corollary} 
\begin{proof}
We only have to apply \cite[Thm. 3.4]{Wil19a}, resp. \cite[Thm. 3.5]{Wil19a}, resp. \cite[Thm. 3.6]{Wil19a}, in order to obtain the point \ref{itm:filt}, resp. \ref{itm:fonct}, resp. \ref{itm:compl} (thanks to Theorem \ref{thm:mainthm}).
\end{proof}

Fix now an integer $N$ such that, as in Remark \ref{factpowab}, $^\lambda \mathcal{V}$ is a direct factor of a Tate twist of $p_{N,*} \mathbbm{1}_{\mathcal{A}_K^{N}}$, where $p_{N}: \mathcal{A}_K^{N} \rightarrow S_K$ denotes the $N$-fold fibred product of the universal abelian variety $\mathcal{A}_K$ with itself over $S_K$. The property stated in Corollary \ref{cor:propint}.\ref{itm:fonct} has important consequences for the \emph{Hecke algebra} $\mathfrak{H}(K, G(\mathbb{A}_f))$ associated to the open compact subgroup $K$, as defined in \cite[pp. 591-592]{Wil17a}, along with the action of each of its elements on $\tilde{s}_*p_{N,*} \mathbbm{1}_{\mathcal{A}_K^{N}}$. By \emph{loc. cit.}, each element of $\mathfrak{H}(K, G(\mathbb{A}_f))$ also acts\footnote{\label{algebract}From the results  in the literature, it does not appear clear, though expected, that there is an \emph{algebra} action of $\mathfrak{H}(K, G(\mathbb{A}_f))$ on the motives $\tilde{s}_*^\lambda \mathcal{V}$. We will come back to this point in future work. However, $\mathfrak{H}(K, G(\mathbb{A}_f))$ acts as an algebra \emph{after realization}, and hence acts as an algebra on the homological motive underlying the intersection motive: this is what will matter for us.} on $\tilde{s}_*^\lambda \mathcal{V}$. Then, corollary \ref{cor:propint}.\ref{itm:fonct} gives us immediately the following consequence:

\begin{corollary} \label{heckeaction}
Each element of $\mathfrak{H}(K, G(\mathbb{A}_f))$ acts naturally on the intersection motive $s_*j_{!*}^\lambda \mathcal{V}$.
\end{corollary}

It is also useful to explicitly formulate the property stated in Corollary \ref{cor:propint}.\ref{itm:compl} in a specific context:

\begin{corollary} \label{factcomp}
Let $\tilde{\mathcal{A}}_K^N$ be a smooth compactification of $\mathcal{A}_K^N$. Then, the intersection motive $s_*j_{!*}^\lambda \mathcal{V}$ is canonically identified with a direct factor of a Tate twist of $a_*\mathbbm{1}_{\tilde{\mathcal{A}}_K^N}$ (where $a$ is the structural morphism of $\tilde{\mathcal{A}}_K^N$ towards $\mbox{\emph{Spec}} \ \mathbb{Q}$), with a canonical direct complement.
\end{corollary}

This corollary has important consequences for the realizations of $s_*j_{!*}^\lambda \mathcal{V}$.

\begin{corollary} \label{cor:realadiques}

Let $\mathcal{O}$ be the order of $F$ prescribed by the PEL datum corresponding to $S_K$ (cfr. Remark \ref{rmk:abuniv}), $D$ the \emph{discriminant} of $\mathcal{O}$ as defined in \cite[Def. 1.1.1.6]{Lan13}, and $N$ the \emph{level} of $K$. Let $p$ be a prime which does not divide $D \cdot N$. Then:
\begin{enumerate} [wide, labelwidth=!, labelindent=0pt, label=(\arabic*)]
\item \label{itm:crystram} the $p$-adic realization of $s_*j_{!*}^\lambda \mathcal{V}$ is crystalline, and if $\ell$ is a prime different from $p$, the $\ell$-adic realization of $s_*j_{!*}^\lambda \mathcal{V}$ is unramified at $p$;
\item \label{itm:polcar} consider on the one hand the action of Frobenius $\phi$ on the $\phi$-filtered module associated to the (crystalline) $p$-adic realization of $s_*j_{!*}^\lambda \mathcal{V}$, and on the other hand the action of a geometrical Frobenius at $p$ on the $\ell$-adic realization of $s_*j_{!*}^\lambda \mathcal{V}$ (unramified at $p$). Then, the characteristic polynomials of the two actions coincide. 
\end{enumerate}
\end{corollary}
\begin{proof}
\begin{enumerate}
\item By \cite[Thm. 4.14]{Wil09}, and with the notations of the preceding corollary, the existence of a smooth compactification of $\mathcal{A}^N_K$ with good reduction properties is enough to get the conclusion. Now, we have at our disposal the very general results of \cite{Lan12} on the existence of smooth projective integral models of smooth compactifications of \emph{PEL-type Kuga-Sato families}: namely, Thm. 2.15 of \emph{loc. cit.} (by taking into account Definition 1.6 of \emph{loc. cit.} and \cite[Prop. 1.4.4.3]{Lan13}) implies that there exists a smooth compactification of $\mathcal{A}^N_K$ with good reduction at every prime $p$ which does not divide $D \cdot N$. Thus, we can invoke \cite[Thm. 4.14]{Wil09} to conclude.
\item We argue exactly as in \cite[Cor. 1.13]{Wil19b}, in order to use \cite[Thm. 2.2]{KM74} and conclude.
\end{enumerate} \end{proof}

In order to end this list of properties of $s_*j_{!*}^\lambda \mathcal{V}$, we recall that the reason for the name of the \emph{intersection motive} is the behaviour of its realizations (recall that we are supposing that $\lambda$ is either not completely irregular or of corank 0):

\begin{corollary} \label{intmotintemot}
\begin{enumerate} [wide, labelwidth=!, labelindent=0pt, label=(\arabic*)]
\item \label{itm:integint} For every $n \in \mathbb{Z}$, the natural maps 
\begin{equation*}
H^n(S_K^*(\mathbb{C}), j_{!*} \mu_{H}^{K}(V_{\lambda})) \rightarrow H^n(S_K(\mathbb{C}),  \mu_{H}^{K}(V_{\lambda}))
\end{equation*}
(between cohomology spaces of Hodge modules) and 
\begin{equation*}
H^n((S_K^*) \times_{\mathbb{Q}} \bar{\mathbb{Q}}, j_{!*} \mu^{K}_{\ell}(V_{\lambda})) \rightarrow H^n((S_K) \times_{\mathbb{Q}} \bar{\mathbb{Q}},  \mu^{K}_{\ell}(V_{\lambda}))
\end{equation*}
(between cohomology spaces of $\ell$-adic perverse sheaves) are injective, and dually, the natural maps
\begin{equation*}
H^n_c(S_K(\mathbb{C}),  \mu_{H}^{K}(V_{\lambda})) \rightarrow H^n_c(S_K^*(\mathbb{C}), j_{!*} \mu_{H}^{K}(V_{\lambda}))
\end{equation*}
and
\begin{equation*}
H^n_c((S_K) \times_{\mathbb{Q}} \bar{\mathbb{Q}},  \mu^{K}_{\ell}(V_{\lambda})) \rightarrow  H^n_c((S_K^*) \times_{\mathbb{Q}}\bar{\mathbb{Q}}, j_{!*} \mu^{K}_{\ell}(V_{\lambda}))
\end{equation*}
are surjective. Consequently, the natural maps from intersection cohomology of $S_K$ towards interior cohomology (with coefficients in $\mu_{H}^{K}(V_{\lambda})$, resp. $\mu^{K}_{\ell}(V_{\lambda}))$ are isomorphisms. 
\item \label{itm:real} The Hodge realization, resp. $\ell$-adic realization of the intersection motive $s_*j_{!*}^\lambda \mathcal{V} \in CHM(\mathbb{Q})_L$ is identified with interior cohomology $H^*_!(S_K(\mathbb{C}), \mathcal{R}_H(^\lambda \mathcal{V}))$, resp. $H^*_!((S_K)_{\times_{\mathbb{Q}}}\bar{\mathbb{Q}}, \mathcal{R}_{\ell}(^\lambda \mathcal{V}))$.
\end{enumerate}
\label{cor:realint}
\end{corollary} 
\begin{proof}
Theorem \ref{thm:mainthm} tells us that the motive $i^* j_* ^\lambda \mathcal{V}$ avoids weights $-\beta, \dots, \beta +1$. Then, it is enough to apply \cite[Rmk. 3.13 (c)]{Wil19a} to the complexes $i^* j_* R(^\lambda \mathcal{V})$, where $R$ is the Hodge or $\ell$-adic realization functor on $CHM(S_K)_L$ .

Point \ref{itm:real} follows from \ref{itm:integint} and from the fact that realizations of the intersection motive are identified with intersection cohomology (\cite{Wil19a}, before Proposition 3.8; cfr. \cite[Thm. 7.2]{Wil17a} for details). 
\end{proof}

\begin{remark}
\begin{enumerate} [wide, labelwidth=!, labelindent=0pt, label=(\arabic*)]
\item The vanishing theorems for cohomology of locally symmetric spaces with coefficients in regular representations of the underlying algebraic group ((\cite[Cor. 5.6]{LS04} or \cite[Thm. 5]{Sap05}) imply that, if $\lambda$ is regular, the spaces $H^n(S_K(\mathbb{C}), \mu_H^{K}(V_{\lambda}))$, and so (by comparison) $H^n(S_K \times_{\mathbb{Q}} \bar{\mathbb{Q}}, \mu_{\ell}^{K}(V_{\lambda}))$ are zero for $n<3d=\mbox{dim} S_K$. Dually, we get $H_c^n(S_K(\mathbb{C}), \mu_H^{K}(V_{\lambda}))=0$ and $H_c^n(S_K \times_{\mathbb{Q}} \bar{\mathbb{Q}}, \mu_{\ell}^{K}(V_{\lambda}))=0$ for $n>3d$. As a consequence, if $\lambda$ is regular, then the interior cohomology spaces $H_!^n(S_K(\mathbb{C}), \mu_H^{K}(V_{\lambda}))$ and $H_!^n(S_K \times_{\mathbb{Q}} \bar{\mathbb{Q}}, \mu_{\ell}^{K}(V_{\lambda}))$ are zero in degrees different from $n=3d$.
\item \label{itm:degreal} Corollary \ref{cor:realint}.\ref{itm:real} and the preceding point imply that, if $\lambda$ is regular, the Hodge realization of the intersection motive $s_*j_{!*} ^\lambda \mathcal{V} \in CHM(\mathbb{Q})_L$ is given by $H_!^{3d}(S_K(\mathbb{C}), \mu_H^{K}(V_{\lambda}))[-(w(\lambda)+3d)]$, and that its $\ell$-adic realization is given by $H_!^{3d}(S_K \times_{\mathbb{Q}} \bar{\mathbb{Q}}, \mu_{\ell}^{K}(V_{\lambda}))[-(w(\lambda)+3d)]$.
\end{enumerate}
\label{rmk:proprealint}
\end{remark}

\subsection{Homological motives for automorphic representations}

Keep the notations of the preceding subsection and assume moreover that $\lambda$ is regular. In this last part, following \cite[Sec. 3]{Wil19b}, we would like to exploit the action of the elements of the algebra $\mathfrak{H}(K,G(\mathbb{A}_f))$ on the intersection motive $s_*j_{!*} ^\lambda \mathcal{V} \in CHM(\mathbb{Q})_L$ (cfr. Corollary \ref{heckeaction}) to cut out certain \emph{homological} sub-motives thereof. Recall that the algebra $\mathfrak{H}(K,G(\mathbb{A}_f)$ acts on $H_!^{3d}(S_K(\mathbb{C}), \mu_H^{K}(V_{\lambda}))$ and on $H_!^{3d}(S_K \times_{\mathbb{Q}} \bar{\mathbb{Q}}, \mu_{\ell}^{K}(V_{\lambda}))$.

\begin{theorem}{(\cite[Sect. 8.1.7, page 253]{Hard})}

For every extension $E$ of $L$, the $\mathfrak{H}(K,G(\mathbb{A}_f))\otimes_{L} E$-module $H_!^{3d}(S_K(\mathbb{C}), \mu_H^{K}(V_{\lambda})) \otimes E$ is semisimple.
\end{theorem}

\begin{corollary}

Let $R(\mathfrak{H}):=R(\mathfrak{H}(K,G(\mathbb{A}_f)))$ be the image of $\mathfrak{H}(K,G(\mathbb{A}_f))$ in the endomorphism algebra of $H_!^{3d}(S_K(\mathbb{C}), \mu_H^{K}(V_{\lambda}))$. Then, for every extension $E$ of $L$, the algebra $R(\mathfrak{H})\otimes_{L}E$ is semisimple. 
\end{corollary}

In particular, isomorphism classes of simple right $R(\mathfrak{H})\otimes_{L}E$-modules are in bijection with isomorphism classes of minimal right ideals. Now, by fixing $E$, and one of these minimal right ideals $Y_{\pi_f}$ of $R(\mathfrak{H})\otimes_{L}E$, there exists an idempotent $e_{\pi_f} \in R(\mathfrak{H})\otimes_{L}E$ which generates $Y_{\pi_f}$.

\begin{definition}
\begin{enumerate}[wide, labelwidth=!, labelindent=0pt, label=(\arabic*)]
\item The \emph{Hodge structure} $W(\pi_f)$ \emph{associated to} $Y_{\pi_f}$ is defined by
\begin{equation*}
W(\pi_f):=\mbox{Hom}_{R(\mathfrak{H})\otimes_{L}E}(Y_{\pi_f}, H_!^{3d}(S_K(\mathbb{C}), \mu_H^{K}(V_{\lambda})) \otimes E).
\end{equation*}
\item Let $E$ be a finite extension. For every prime number $\ell$, and for every prime $l$ of $E$ above $\ell$, the \emph{Galois module} $W(\pi_f)_{\ell}$ \emph{associated to} $Y_{\pi_f}$ is defined by
\begin{equation*}
W(\pi_f)_{\ell}:=\mbox{Hom}_{R(\mathfrak{H})\otimes_{L}E_l}(Y_{\pi_f}, H_!^{3d}(S_K \times_{\mathbb{Q}} \bar{\mathbb{Q}}, \mu_{\ell}^{K}(V_{\lambda})) \otimes E_l).
\end{equation*}
\end{enumerate}
\end{definition}

The proof of the following is immediate (see for example \cite[Prop. 3.4]{Wil19b}):
\begin{proposition}
There are canonical isomorphisms of Hodge structures, resp. of Galois modules
\begin{equation*}
W(\pi_f) \simeq H_!^{3d}(S_K(\mathbb{C}), \mu_H^{K}(V_{\lambda}) \otimes E) \cdot e_{\pi_f},
\end{equation*}
resp.
\begin{equation*}
W(\pi_f)_{\ell} \simeq H_!^{3d}(S_K \times_{\mathbb{Q}} \bar{\mathbb{Q}}, \mu_{\ell}^{K}(V_{\lambda}) \otimes E) \cdot e_{\pi_f}.
\end{equation*}
\end{proposition}

Since we do not know if $e_{\pi_f}$ lifts to an idempotent element of $\mathfrak{H}(K,G(\mathbb{A}_f))\otimes E$, we are forced to consider its action on the \emph{homological} (or Grothendieck) motive which underlies the intersection motive $s_*j_{!*} ^\lambda \mathcal{V} \in CHM(\mathbb{Q})_L$. Denote then by $s_*j_{!*} ^\lambda \mathcal{V}^\prime$ this homological motive, and define, thanks to Corollary \ref{heckeaction}:

\begin{definition}
The \emph{(homological) motive corresponding to} $Y_{\pi_f}$ is defined by $\mathcal{W}(\pi_f):=s_*j_{!*} ^\lambda \mathcal{V}^\prime \cdot e_{\pi_f}$.
\end{definition}

We finish by making explicit the properties of the latter motive which follow from the preceding constructions:

\begin{theorem}
The realizations of the motive $\mathcal{W}(\pi_f)$ are concentrated in cohomological degree $w(\lambda)+3d$, where in particular the Hodge realization equals $W(\pi_f)$, and the $\ell$-adic realizations equal $W(\pi_f)_{\ell}$, for every prime $\ell$.
\end{theorem}
\begin{proof}
Follows from the construction of $\mathcal{W}(\pi_f)$ and Remark \ref{rmk:proprealint}.\ref{itm:degreal} (remember that we are supposing $\lambda$ to be regular).
\end{proof}

\begin{corollary} \label{cristram}
Let p be a prime number which does not divide the integer $D \cdot N$ from Corollary \ref{cor:realadiques}, and $\ell$ a prime different from $p$. Then:
\begin{enumerate}
[wide, labelwidth=!, labelindent=0pt, label=(\arabic*)]
\item the $p$-adic realization of $\mathcal{W}(\pi_f)$ is crystalline, and the $\ell$-adic realization of $\mathcal{W}(\pi_f)$ is unramified at $p$;
\item consider on the one hand the action of the Frobenius $\phi$ on the $\phi$-filtered module associated to the $p$-adic (crystalline) realization of $\mathcal{W}(\pi_f)$, and on the other hand the action of a geometrical Frobenius at $p$ on the $\ell$-adic realization of $\mathcal{W}(\pi_f)$ (unramified at $p$). Then, the characteristic polynomials of the two actions coincide. 
\end{enumerate}
\end{corollary}
\begin{proof}
\begin{enumerate}[wide, labelwidth=!, labelindent=0pt, label=(\arabic*)]
\item Follows from Corollary \ref{cor:realadiques}.\ref{itm:crystram}, by taking into account the fact that $\mathcal{W}(\pi_f)$ is a direct factor of $s_*j_{!*} ^\lambda \mathcal{V}^\prime$.
\item We can argue as in Corollary \ref{cor:realadiques}.\ref{itm:polcar} to apply \cite[Thm. 2.2]{KM74} and conclude. 
\end{enumerate}
\end{proof}

\begin{remark} \label{modmot}
\begin{enumerate} [wide, labelwidth=!, labelindent=0pt, label=(\arabic*)]
\item Suppose that $\lambda$ is a regular weight of $G$ whose restriction to the center is trivial and that the image of $K$ along the natural projection $\mbox{GSp}_{4} \rightarrow \mbox{PGSp}_{4}$ is still a compact open subgroup of $\mbox{PGSp}_{4}(\mathbb{A}_f)$. Then, the $\ell$-adic realizations $W(\pi_f)_{\ell}$ of the motive $\mathcal{W}(\pi_f)$ coincide with the Galois modules $H^*_c(\pi_{Hf})$ associated to suitable automorphic representations of $G$ in \cite[Part 2, Chap. I.2, Thm. 2]{Fli05} (notice that in general, under the regularity assumption, cuspidal cohomology equals interior cohomology, as can be seen for example from \cite[3.2.4]{Harr91}, and that in our case, by Cor. \ref{intmotintemot}.\ref{itm:integint}, regularity implies in turn equality of interior and intersection cohomology). The existence of the motive $\mathcal{W}(\pi_f)$ then adds to the description in \cite{Fli05} the information about the behaviour at $p$ of the Galois module $W(\pi_f)_p$, which has been obtained in Corollary \ref{cristram}. 
\item Keep the assumptions of the preceding point and let $l$ be a place of  $E$ above the prime $\ell$. The Galois modules $W(\pi_f)_{\ell}$ are then of dimension $4^d$ or $\frac{1}{2} \cdot 4^d$ over $E_l$ (\cite[Part 2, Chap. I.2, Thm. 2.(1),(4)]{Fli05}). One can expect that, in the case of a (Hilbert-Siegel) \emph{eigenform} $f$, the motives $\mathcal{W}(\pi_f)$ over $\mathbb{Q}$ can be written as tensor products over $E$ of rank-4 motives over $F$, whose $L$-function has the correct relation with the $L$-function of $f$. However, there are no known methods for constructing motives with such properties. It is the same problem which arises for motives corresponding to Hilbert modular forms, when cut out inside Kuga-Sato varieties over Hilbert modular varieties, cfr. for example \cite[5.2]{Harr91}.
\end{enumerate}
\end{remark}

\end{document}